\documentclass{amsart}
\usepackage{import}

\usepackage[utf8]{inputenc}

\usepackage{amsthm}
\usepackage{amsfonts}
\usepackage{graphicx}
\usepackage{amsmath, amssymb}
\usepackage{hyperref}
\usepackage{cleveref}
\usepackage{mathrsfs}
\usepackage{multirow}
\usepackage{centernot}
\usepackage{tikz}
\usepackage{tikz-cd}
\usepackage{color}
\usetikzlibrary{arrows}
\usepackage{pgf}
\usepackage{mathtools}
\usepackage{enumerate}
\usepackage{subcaption} 
\usepackage[myheadings]{fullpage}

\newtheorem{theorem}{Theorem}[section]
\newtheorem{assumption}{Assumption}[section]
\newtheorem{lemma}[theorem]{Lemma}
\newtheorem{proposition}[theorem]{Proposition}

\theoremstyle{definition}
\newtheorem{definition}[theorem]{Definition}
\newtheorem{notation}{Notation}

\newtheorem{example}{Example}[section]
\newtheorem{remark}{Remark}[section]

\newcommand{\R}{\mathbb{R}}
\newcommand{\C}{\mathbb{C}}

\newcommand{\GT}{\mathbb{GT}}

\newcommand{\Z}{\mathbb{Z}}

\newcommand{\e}{\varepsilon}

\newcommand{\1}{\mathbf{1}}

\newcommand{\bi}{\mathbf{i}}

\newcommand{\bm}{\mathbf{m}}

\renewcommand{\vec}[1]{\boldsymbol{#1}}

\newcommand{\fB}{\mathfrak{B}}

\newcommand{\fD}{\mathfrak{D}}
\newcommand{\fE}{\mathfrak{E}}

\newcommand{\fM}{\mathfrak{M}}
\newcommand{\fO}{\mathfrak{O}}

\newcommand{\fV}{\mathfrak{V}}

\newcommand{\fa}{\mathfrak{a}}

\newcommand{\fc}{\mathfrak{c}}

\newcommand{\fp}{\mathfrak{p}}

\newcommand{\ft}{\mathfrak{t}}
\newcommand{\fu}{\mathfrak{u}}
\newcommand{\fv}{\mathfrak{v}}

\newcommand{\fz}{\mathfrak{z}}
\newcommand{\cA}{\mathcal{A}}
\newcommand{\cB}{\mathcal{B}}

\newcommand{\cD}{\mathcal{D}}
\newcommand{\cE}{\mathcal{E}}
\newcommand{\cF}{\mathcal{F}}
\newcommand{\cG}{\mathcal{G}}
\newcommand{\cH}{\mathcal{H}}

\newcommand{\cK}{\mathcal{K}}
\newcommand{\cL}{\mathcal{L}}
\newcommand{\cM}{\mathcal{M}}

\newcommand{\cQ}{\mathcal{Q}}

\newcommand{\cS}{\mathcal{S}}
\newcommand{\cT}{\mathcal{T}}
\newcommand{\cU}{\mathcal{U}}

\newcommand{\cW}{\mathcal{W}}
\newcommand{\cX}{\mathcal{X}}

\newcommand{\sL}{\mathscr{L}}

\newcommand{\sfI}{\mathsf{I}}
\newcommand{\sfM}{\mathsf{M}}

\DeclareMathOperator{\supp}{supp}
\DeclareMathOperator{\dist}{dist}

\makeatletter
\newcommand{\vast}{\bBigg@{4}}
\newcommand{\Vast}{\bBigg@{5}}
\makeatother

\newcommand{\wh}[1]{\widehat{#1}}
\newcommand{\wt}[1]{\widetilde{#1}}

\DeclareMathOperator{\res}{Res}
\DeclareMathOperator{\diag}{diag}

\DeclareMathOperator{\tr}{tr}

\DeclareMathOperator{\E}{\mathbb{E}}

\let\Re\relax
\DeclareMathOperator{\Re}{Re}
\let\Im\relax
\DeclareMathOperator{\Im}{Im}

\numberwithin{equation}{section}

\title{Airy Point Process via Supersymmetric Lifts}
\author{Andrew Ahn}

\DeclareFontFamily{OT1}{pzc}{}
\DeclareFontShape{OT1}{pzc}{m}{it}{<-> s * [1.10] pzcmi7t}{}
\DeclareMathAlphabet{\mathpzc}{OT1}{pzc}{m}{it}

\begin{document}
\maketitle

\begin{abstract}
We study the local asymptotics at the edge for particle systems arising from: (i) eigenvalues of sums of unitarily invariant random Hermitian matrices and (ii) signatures corresponding to decompositions of tensor products of representations of the unitary group. Our method treats these two models in parallel, and is based on new formulas for observables described in terms of a special family of lifts, which we call supersymmetric lifts, of Schur functions and multivariate Bessel functions. We obtain explicit expressions for a class of supersymmetric lifts inspired by determinantal formulas for supersymmetric Schur functions due to \cite{MJ03}. Asymptotic analysis of these lifts enable us to probe the edge. We focus on several settings where the Airy point process arises.
\end{abstract}

\setcounter{tocdepth}{2}
\tableofcontents

\section{Introduction}

\subsection{Preface}

It is known due to Weyl that the isomorphism classes of irreducible representations of the $N\times N$ unitary group $\cU(N)$ are in bijection with the set $\cW_{\Z}^N := \{\lambda = (\lambda_1 \ge \cdots \ge \lambda_N) \in \Z^N \}$ of \emph{signatures}, where $\lambda$ corresponds to the highest weight of an irreducible, see e.g. \cite{Wey97}. Thus given a representation $V$ of $\cU(N)$, we have a decomposition
\[ V = \bigoplus_{\lambda \in \cW_{\Z}^N} V_\lambda^{\oplus c_\lambda} \]
into irreducible components.

There is an analogous decomposition for random $N\times N$ complex Hermitian matrices with unitarily invariant distributions, that is $UXU^{-1} \overset{d}{=} X$ for any $U \in \cU(N)$. Such a random matrix $X$ induces a probability measure $\varrho_X$ on its eigenvalues, supported in $\cW_{\R}^N := \{\vec{x} = (x_1 \ge \cdots \ge x_N) \in \R^N\}$. Tautologically, we have the decomposition
\begin{align} \label{eq:X_decomp}
\varrho_X = \int_{\cW_{\R}^N} \delta_{\vec{x}} \, d\varrho_X(\vec{x})
\end{align}
To make the analogy with Hermitian matrices direct, we may encode the data of the decomposition of $V$ into a probability measure $\rho_V$ on $\cW_{\Z}^N$ where $\rho_V(\lambda) \propto c_\lambda \dim V_\lambda$, or alternatively
\begin{align} \label{eq:V_decomp}
\rho_V = \sum_{\lambda \in \cW_{\Z}^N} \frac{c_\lambda \dim V_\lambda}{\dim V} \delta_\lambda
\end{align}
In this form, the irreducible representations become the extremal measures $\delta_\lambda$ on $\cW_{\Z}^N$, corresponding to the extremal measures $\delta_{\vec{x}}$ on $\cW_{\R}^N$, and the weights $\rho_V(\lambda)$ correspond to the density of $d\varrho_X$ at points.

We study the effect of certain operations on these decompositions. We study tensor products for representations and sums for matrices. More concretely, we consider the decompositions \eqref{eq:X_decomp} and \eqref{eq:V_decomp} with
\begin{enumerate}[(i)]
    \item\label{item:V} $V = V^{(1)} \otimes \cdots \otimes V^{(n)}$ for $V^{(i)}$ an irreducible representation of $\cU(N)$, and
        
    \item\label{item:X} $X = X^{(1)} + \cdots + X^{(n)}$ for independent unitarily invariant $N \times N$ complex Hermitian matrices $X^{(i)}$ with deterministic eigenvalues.
\end{enumerate}

The measures $\rho_V$ are quantized analogues of $\varrho_X$. By a semiclassical limit, $\rho_V$ recovers $\varrho_X$, mapping $\otimes$ to $+$, see \cite[Proposition 1.5]{BuG15}. In the other direction, decompositions of irreducible representations of $\cU(N)$ may be viewed as the geometric quantization of addition of random Hermitian matrices, see \cite[\S 1.3 \& Appendix D]{CNS18}.

Furthermore, the operation $\otimes$ may be viewed as a discrete equivalent of $+$. If $\vec{a},\vec{b} \in \cW_{\R}^N$ are the eigenvalues of independent unitarily invariant $N\times N$ Hermitian matrices $A,B$, the set of possible eigenvalues $\vec{c}$ of $A+B$ is a polytope $P_{\vec{a},\vec{b}} \subset \cW_{\R}^N$ carved out by a set of relations shown to be necessary to satisfy by \cite{Hor62} and sufficient by the combined works \cite{Kly98} and \cite{KT99}; the problem of sufficiency is known as Horn's problem. In particular, \cite{Kly98} reduced Horn's problem to the \emph{saturation conjecture}, solved by \cite{KT99}, which states that the multiplicity of $V_{\vec{c}}$ in the decomposition of $V_{\vec{a}}\otimes V_{\vec{b}}$ is nonzero if and only if $\vec{c} \in P_{\vec{a},\vec{b}}\cap\Z$ whenever $\vec{a},\vec{b} \in \cW_{\Z}^N = \cW_{\R}^N \cap \Z^N$. One can further relate the multiplicity of $V_{\vec{c}}$ in $V_{\vec{a}} \otimes V_{\vec{b}}$ with the density of $\varrho_{A+B}$ at $\vec{c}$ in terms of the number or density of certain honeycomb configurations given boundary data prescribed by $(\vec{a},\vec{b},\vec{c})$, see \cite[Theorems 4 \& 7]{KT01}.

In this article, we study the measures $\rho_V$ and $\varrho_X$ in the language of particle systems. Naturally, the eigenvalues of $X$ may be viewed as a random $N$-particle system on $\R$. Similarly, the pushforward of $\rho_V$ by the mapping $\lambda \in \cW_{\Z}^N$ to $(\lambda_1 + N - 1 > \lambda_2 + N - 2 > \ldots > \lambda_N)$ defines a random $N$-particle system on $\Z$. Thus the above operations on representations and matrices may be viewed as operations on particle systems.

We explore the limiting $N \to \infty$ behavior of the particle systems arising from (\ref{item:V}) and (\ref{item:X}), with the motivation of understanding the effects of the operations above at the \emph{local level}. Our goal is to present a unified moment-based method to access the local fluctuations of the largest particles in these systems. We remark that while our local results were not previously accessible, the complementary global results which capture the entire particle system are well known, see \cite{V85,Bia95,BuG15}.

Our main results (\Cref{thm:samelimit,thm:multisum,thm:two_sum,thm:multitensor}) establish, under several regimes, that the largest particles are given by the Airy point process, the local limit of the spectral edge of the Gaussian Unitary Ensemble (GUE) as the matrix size tends to infinity, see \Cref{intro:matrix} for a formal description. These results demonstrate that the operations $\otimes$ and $+$ tend to have a regularizing effect in the $N \to \infty$ limit, in the sense that sufficient applications of these operations lead to GUE edge fluctuations. We state our main theorems and review existing results for the random matrix and quantized models in separate sections (\Cref{intro:matrix,intro:quantized}).

Although we focus on regimes where the Airy point process arises in this work, we expect that our method can probe alternative limit settings, such as outliers at the edge or growing order of summands and tensor products. We plan to explore this in future works.

The key input to these results are novel expressions for a family of observables of these particle systems in terms of special lifts --- which we call \emph{supersymmetric lifts} --- of Schur functions and multivariate Bessel functions (continuous analogues of Schur functions), see \Cref{ssec:ssym}. Intriguingly, the \emph{supersymmetric Schur functions}, certain irreducible characters of the Lie superalgebra $\mathfrak{gl}(m/n)$\footnote{Recall that the ordinary Schur functions are irreducible characters of the Lie algebra $\mathfrak{gl}(n)$}, provide an example of these lifts. Inspired by determinantal formulas for the supersymmetric Schur functions due to \cite{MJ03}, we find a family of supersymmetric lifts of Schur and multivariate Bessel functions with explicit contour integral formulas amenable for asymptotic analysis. The connection with observables consequently allows us to study local extremal asymptotics of our particle systems. A nice feature of the approach is that the proofs of the local theorems in the quantized and continuous particle systems parallel each other almost perfectly.

\subsection{Main Results: Unitarily Invariant Random Matrices} \label{intro:matrix}

We study the eigenvalues of
\begin{align} \label{rmt_model}
X^{(1)}_N + \cdots + X^{(n)}_N
\end{align}
where the $X^{(i)}_N$ are independent $N\times N$ unitarily invariant complex Hermitian matrices with deterministic eigenvalues. 

In general, the $N\to\infty$ limiting edge behavior of \eqref{rmt_model} depends on the spectrum of $X_N^{(i)}$ for $1 \le i \le n$. We focus on several settings where the Airy point process appears, which is a determinantal point process with correlation kernel
\[ K_{\mathrm{Airy}}(x,y) = \frac{\operatorname{Ai}(x) \operatorname{Ai}'(y) - \operatorname{Ai}(y) \operatorname{Ai}'(x)}{x - y} \]
where $\operatorname{Ai}$ is the Airy function. This is the $N\to\infty$ limiting point process at the edge of the spectrum of the $N\times N$ GUE \cite{TW94}; recall that a GUE is distributed as $\tfrac{X+X^*}{2}$ where $X$ is an $N\times N$ matrix of i.i.d. standard complex Gaussians, see e.g. \cite[Chapter 2.5]{AGZ10}. 

The Airy point process is a universal object appearing within and beyond random matrix theory. In the context of random matrix theory, the seminal works \cite{Sos99,TV10,EYY12b} establish universality of the Airy point process for Wigner matrices, which have i.i.d. (up to being Hermitian) centered, variance $1$ entries, and additional moment assumptions which typically vary among works. We note that local GUE statistics in the bulk of the spectrum are also known to be universal among Wigner matrices, due to \cite{EPRSY10,TV11} and further extended by \cite{EYY12a,Agg19}.

Let $\cW_{\R}^N := \{\vec{x} \in \R^N: x_1 \ge \cdots \ge x_N\}$ and $\cM$ be the set of compactly supported Borel probability measures on $\R$. 

\begin{assumption} \label{assum:rmt_conv}
For each positive integer $N$, let $X^{(1)}_N,X^{(2)}_N,\ldots$ be a sequence of independent $N\times N$ unitarily invariant, random Hermitian matrices with deterministic eigenvalues $\vec{\ell}^{(1)} := \vec{\ell}^{(1)}(N),\vec{\ell}^{(2)} := \vec{\ell}^{(2)}(N),\ldots \in \cW_{\R}^N$ respectively. Assume that there exists $\bm^{(i)} \in \cM$ such that
\[ \lim_{N\to\infty} \frac{1}{N} \sum_{j=1}^N \delta_{\ell_j^{(i)}} := \bm^{(i)} \quad \mbox{weakly}, \quad \quad \lim_{N\to\infty} \sup_{1 \le j \le N} \dist(\ell_j^{(i)}, \supp \bm^{(i)}) = 0 \]
for each $i \ge 1$.
\end{assumption}

We say that a sequence of point processes $\cX_N$ converges to a point process $\cX$ if for each $k \ge 1$, the $k$-point correlation functions of $\cX_N$ converge weakly to the $k$-point correlation function of $\cX$ as $N\to\infty$, see \cite[Chapter 1]{HKPV09} for background on point processes and $k$-point correlation functions.

In the statements below, $\tau(\mu)$ is a constant depending only on $\mu \in \cM$, given by \eqref{eq:tau} in \Cref{sec:applications}.

\begin{theorem} \label{thm:samelimit}
Suppose $\mu \in \cM$ and $n \ge \tau(\mu)$. Under \Cref{assum:rmt_conv}, if $\vec{\ell} \in \cW_{\R}^N$ are the eigenvalues of
\[ X^{(1)}_N + \cdots + X^{(n)}_N \]
such that $\mu = \bm^{(1)} = \cdots = \bm^{(n)}$, then there exist explicit $\fE_N \in \R$, $\fV_N > 0$ such that
\[  N^{2/3} \frac{\ell_i - \fE_N}{\fV_N} \quad i = 1,\ldots,N\]
converges to the Airy point process as $N\to\infty$.
\end{theorem}

\begin{remark}
An explicit description of $\fE_N,\fV_N$ is given in \Cref{thm:master}.
\end{remark}

\begin{remark}
The dependence $\tau(\mu)$ on $\mu$ should be necessary --- this is evidenced by the fact that for arbitrarily large $n$, one can find a measure $\mu$ such that the $N\to\infty$ limiting spectral measure of $X_N^{(1)} + \cdots + X_N^{(n)}$ (under the hypotheses of \Cref{thm:samelimit}) does not exhibit square root behavior at the right edge, a global indicator of the Airy point process at the edge. For example, if $\mu = (1 - x)^p\1_{[0,1]}(x) \, dx$ for $p > 1$, then the aforementioned decay is not square root for $n < p^2$, see \Cref{rmk:power_law}.

On the other hand, in general, $\tau(\mu)$ is not the optimal threshold for which the conclusion of \Cref{thm:samelimit} holds. For example, if $\mu$ is the Rademacher distribution $\frac{1}{2}(\delta_{-1} + \delta_1)$, we can compute that $\tau(\mu) = 82$, and with minor modifications to our proofs (see \Cref{rmk:tau_optimize}) we can lower this to approximately $68$. However, we expect $n = 3$ is the minimal $n$ for which the conclusion of \Cref{thm:samelimit} holds in this case. See \Cref{rmk:tau_optimal} for a conjecturally optimal statement for general $\mu$.
\end{remark}

\begin{remark}
Although \Cref{thm:samelimit} is stated for $n$ fixed, our methods can consider the regime where $n$ tends to infinity with $N$. We consider this in future works.
\end{remark}

We can generalize \Cref{thm:samelimit} so that the limiting measures are not all identical:

\begin{theorem} \label{thm:multisum}
Suppose $\mu_1,\ldots,\mu_k \in \cM$, $n_1,\ldots,n_k$ are integers such that $n_i \ge \tau_i(\mu_i)$ for $1 \le i \le k$, and $n = n_1 + \cdots + n_k$. Under \Cref{assum:rmt_conv}, if $\vec{\ell} \in \cW_{\R}^N$ are the eigenvalues of
\[ X^{(1)}_N + \cdots + X^{(n)}_N \]
such that $\mu_1 = \bm^{(1)} = \cdots = \bm^{(n_1)}, \ldots, \mu_k = \bm^{(n_1 + \cdots + n_{k-1}+1)} = \cdots = \bm^{(n)}$, then there exist explicit $\fE_N \in \R$, $\fV_N > 0$ such that
\[  N^{2/3} \frac{\ell_i - \fE_N}{\fV_N} \quad i = 1,\ldots,N\]
converges to the Airy point process as $N\to\infty$.
\end{theorem}

\begin{remark} \label{rmk:projection}
Our methods can also treat matrix projections. More specifically, we can study eigenvalues of matrices of the form
\[ \pi_N X_{M_1}^{(1)} + \cdots + \pi_N X_{M_n}^{(n)} \]
where $\pi_N X_M$ denotes the principal $N\times N$ top-left corner submatrix of $X_M$. Then the Airy point process appears at the spectral edge as long as the limiting ratios $\lim_{N\to\infty} M_i/N$ are suitably large (depending on the limiting spectral measures of the $X_{M_i}^{(i)}$). We do not pursue this direction for simplification and due to the availability of methods to treat edge asymptotics of projections (e.g. \cite{DM18}), in contrast with the lack of methods for general additive models. Below, we provide a more detailed discussion on special cases of additive models with accessible local edge asymptotics.
\end{remark}

We interpret \Cref{thm:samelimit,thm:multisum} further. It is known due to Voiculescu \cite{V91} that if the empirical measures of $X^{(1)}_N,\ldots,X^{(n)}_N$ converge (e.g. under \Cref{assum:rmt_conv}), then the empirical measure of the eigenvalues of \eqref{rmt_model} converges to a probability measure $\bm^{(1)}\boxplus \cdots \boxplus \bm^{(n)}$ as $N\to\infty$. Here, $\boxplus$ is a commutative and associative binary operation on probability measures known as the \emph{free additive convolution}.

The free central limit theorem \cite{V85,Spe90} states that as $n \to \infty$, after translation and dilation, $\bm^{\boxplus n}$ weakly converges to the semicircle law, which is the large $N$ global spectral limit of the GUE.

The appearance of GUE statistics in the $n\to\infty$ limit can be understood by the following heuristic. Suppose $X_N$ is a $N\times N$ unitarily invariant, random Hermitian matrix with deterministic eigenvalues and $X_N^{(1)},X_N^{(2)},\ldots$ is a sequence of i.i.d. copies of $X_N$. Using moment formulas for Haar unitary matrix elements and the classical central limit theorem, it is shown in \cite[Appendix A]{GS} that for fixed $N$
\[ \lim_{n\to\infty} \frac{1}{\sqrt{n}}\left( X_N^{(1)} + \cdots + X_N^{(n)} - \frac{n}{N} \tr(X_N) \cdot \mathrm{Id}_N \right) = \sqrt{\frac{N}{N^2-1} \sigma_N} \left( Y_N - \frac{1}{N} \tr(Y_N) \cdot \mathrm{Id}_N \right) \]
where $\sigma_N := \tfrac{1}{N} \tr(X_N^2) - (\tfrac{1}{N} \tr(X_N))^2$ and $Y_N$ is an $N\times N$ GUE matrix. Since $\tfrac{1}{N} \tr(Y_N) \to 0$ in probability as $N\to\infty$, we get the heuristic approximation
\begin{align} \label{GUE_approx_sum}
X^{(1)}_N + \cdots + X^{(n)}_N \sim \frac{n}{N} \tr(X_N) \cdot \mathrm{Id}_N + \sqrt{\frac{n}{N} \sigma_N} \cdot Y_N
\end{align}
for $n$ and $N$ large, ignoring commutation of limit issues.

Thus the limits of self-convolutions to the semicircle law may be viewed as a realization of the approximation \eqref{GUE_approx_sum} at the level of global statistics. These heuristic approximations also suggest that local GUE edge statistics should appear for large $n$ with $N$ tending to $\infty$. \Cref{thm:samelimit} shows that a large but finite $n$ is sufficient to see these statistics, and gives a lower bound $\tau(\mu)$ on how large $n$ needs to be.

For $k = 1$, \Cref{thm:multisum} reduces to \Cref{thm:samelimit}. For general $k$, \Cref{thm:multisum} states that once you hit the threshold set by $\tau(\cdot)$ for GUE edge statistics by summing, these statistics are stable under further sums.

\Cref{thm:multisum} demonstrates that addition has a regularizing effect: adding enough yields GUE edge statistics. In another direction, we can show GUE edge statistics for sums of two matrices under technical assumptions. The general result is \Cref{thm:general_twosum} (see also \Cref{thm:master}). While these assumptions are non-optimal and nontrivial to verify, we find a family of Jacobi measures with varying power law decay at the spectral edges which satisfy these assumptions, leading to the following special case of \Cref{thm:general_twosum}:

\begin{theorem} \label{thm:two_sum}
Under \Cref{assum:rmt_conv}, suppose $\vec{\ell} \in \cW_{\R}^N$ are the eigenvalues of
\[ X^{(1)}_N + X^{(2)}_N \]
and the density of $\bm^{(i)}$ is proportional to
\[ (x - \alpha_i)^{a_i}(\beta_i - x)^{b_i} \1_{(\alpha_i,\beta_i)} \, dx \]
for some $\alpha_i < \beta_i$, $a_i \ge -1/2$, $-1 < b_i \le -1/2$, for $i = 1,2$. Then there exist explicit $\fE_N \in \R$, $\fV_N > 0$ such that
\[ N^{2/3} \frac{\ell_i - \fE_N}{\fV_N} \quad i =1,\ldots,N \]
converges to the Airy point process as $N\to\infty$.
\end{theorem}

This result is closely related to those of \cite{BES18,BES20} who showed eigenvalue rigidity at the spectral edge of \eqref{rmt_model} for sums of $X_N^{(1)} + X_N^{(2)}$ where $\bm^{(1)}, \bm^{(2)}$ exhibit power law behavior with exponents between $-1$ and $1$ at the edge.

It would be interesting to identify the family of measure $\bm^{(1)},\bm^{(2)}$ for which the conclusion of \Cref{thm:two_sum} continues to hold. At our current stage, such a general result is out of reach.

Let us remark on several special cases and relatives of \eqref{rmt_model} studied previously. If $X_N$ is an $N\times N$ unitarily invariant random Hermitian matrix, then the eigenvalues of $X_N+G_N$ where $G_N$ is GUE form a determinantal point process with explicit correlation kernel \cite{BH96,Joh01}. In these cases, the local asymptotics are accessible via analysis of the correlation kernel, and have been extensively studied in a variety of regimes, see \cite{BK04,ABK05,Shc11,CW14,CP16}). Under regularity assumptions of $X_N$, \cite{LS15} go beyond the determinantal setting, showing convergence of the largest eigenvalues of $X_N + W_N$ to the Airy point process where $W_N$ is Wigner \cite{LS15} using Dyson Brownian motion. 

Beyond these particular cases, exact local asymptotic results at the spectral edge of additive Hermitian matrix models have not been accessed. We note, however, the work of \cite{CL19} who establish local bulk universality for $X_N^{(1)} + X_N^{(2)}$, where the summands are unitarily invariant with deterministic spectra, under mild assumptions on the spectral limits.

Our approach is based on formulas for the expected exponential moments, where we note the eigenvalues of \eqref{rmt_model} are no longer determinantal point processes. Using these formulas, we show convergence of the Laplace transforms of the $k$-point correlation functions. We defer to \Cref{ssec:outline} an outline of our argument and pointers to sections for proofs.

\subsection{Main Results: Decompositions of Representations of the Unitary Group} \label{intro:quantized}

We now discuss our main results on measures arising from the decomposition of representations of the unitary group. We study $\lambda \sim \rho^V$ where
\begin{align} \label{q_model}
V := V_N := V^{(1)}_N \otimes \cdots \otimes V^{(n)}_N
\end{align}
and $V^{(i)}_N$ is an irreducible representation of $\cU(N)$ for $1 \le i \le n$.

Suppose $\lambda^{(i)} := \lambda^{(i)}(N)$ is the signature associated to $V_N^{(i)}$. It is known \cite[Corollary 1.2 \& Corollary 1.4]{CS} (see also \cite{Bia95}) that if
\[ \frac{1}{N} \sum_{j=1}^N \delta_{\e(N)(\lambda^{(i)}_j + N - j)} \]
converges weakly to some $\bm^{(i)}$ and $\e(N) = o(1/N)$, then
\begin{align} \label{eq:lambda_conv}
\frac{1}{N} \sum_{j=1}^N \delta_{\e(N)(\lambda_j + N - j)} \to \bm^{(1)} \boxplus \cdots \boxplus \bm^{(n)}
\end{align}
weakly in probability, for $\lambda \sim \rho^V$ as in \eqref{q_model}. The idea here is that for $\e(N) = o(1/N)$, as $N$ tends to infinity, the semiclassical limit is approached quickly enough to recover random matrix asymptotics.

For $\e(N) = 1/N$, Bufetov-Gorin \cite{BuG15} found that the asymptotic behavior changes due to the persistence of quantum effects in the limit. In this regime, the right hand side of \eqref{eq:lambda_conv} is replaced by a measure $\bm^{(1)} \otimes \cdots \otimes \bm^{(n)}$, where $\otimes$ is a binary operation on measures with density bounded by $1$, known as the \emph{quantized free convolution}. The operations $\otimes$ are nontrivial twists of $\boxplus$, and can be defined in terms of the latter by a variant of the Markov-Krein correspondence, introduced by Kerov \cite[Chapter IV]{Ker03}, see also \cite[Theorem 1.10]{BuG15} and \Cref{thm:MK_correspondence} in this article. We note that Bufetov-Gorin \cite{BuG18} also determined the global fluctuations under this regime.

While the global asymptotics of \eqref{q_model} are understood, previous works do not address local edge asymptotics. Using our methods, we are able to access the local behavior at the edge, and prove an analogue of \Cref{thm:multisum} in the quantized setting. We now present our main result for the quantized model.

Let $\cM_1 \subset \cM$ be the set of Borel probability measures with density bounded by $1$.

\begin{assumption} \label{assum:m_conv_q}
For each positive integer $N$, let $V^{(1)}_N, V^{(2)}_N,\ldots$ be a sequence of irreducible representations of $\cU(N)$ corresponding to signatures $\lambda^{(1)} := \lambda^{(1)}(N),\lambda^{(2)} := \lambda^{(2)}(N),\ldots \in \cW_{\Z}^N$ respectively. Assume that there exists $\bm^{(i)} \in \cM_1$ such that
\[ \lim_{N\to\infty} \frac{1}{N} \sum_{j=1}^N \delta_{\frac{\lambda^{(i)}_j+N-j}{N}} = \bm^{(i)} \quad \mbox{weakly}, \quad \quad \lim_{N\to\infty} \sup_{1 \le j \le N} \dist\left(\frac{\lambda_j^{(i)} + N - j}{N}, \supp \bm^{(i)}\right) = 0 \]
for each $i \ge 1$.
\end{assumption}

For the next theorem, $\tau_q(\mu)$ is a constant depending only on $\mu \in \cM_1$, given by \eqref{eq:tau_q} in \Cref{sec:applications_q}.

\begin{theorem} \label{thm:multitensor}
Suppose $\mu_1,\ldots\mu_k \in \cM_1$, $n_1,\ldots,n_k$ are integers such that $n_i \ge \tau_q(\mu_i)$ for $1 \le i \le k$, and $n = n_1 + \cdots + n_k$. Under \Cref{assum:m_conv_q}, if $\lambda \sim \rho^V$ for
\[ V := V^{(1)}_N \otimes \cdots \otimes V^{(n)}_N \]
such that $\mu_1 = \bm^{(1)} = \cdots = \bm^{(n_1)}, \ldots, \mu_k = \bm^{(n_1 + \cdots + n_{k-1} + 1)} = \cdots = \bm^{(n)}$, then there exist explicit $\fE_N \in \R$, $\fV_N > 0$ such that
\[ N^{2/3} \frac{\frac{\lambda_i + N - i}{N} - \fE_N}{\fV_N} \quad i = 1,\ldots,N \]
converges to the Airy point process as $N\to\infty$.
\end{theorem}

\begin{remark}
We obtain a quantized analogue of \Cref{thm:samelimit} by taking $k = 1$.
\end{remark}

\begin{remark}
A quantized analogue of \Cref{thm:two_sum} can be obtained using ideas from the proof of \Cref{thm:multitensor}. In the interest of space and avoiding repetition, we do not state any formal results in this direction.
\end{remark}

\begin{remark}
Just as our methods can treat projections for random matrices (see \Cref{rmk:projection}), our methods can also treat restrictions of representations. We note that \cite{Pet14,DM20} studied local edge asymptotics of $\rho_V$ where $V$ is a restriction of an irreducible representation, and \cite[Corollary 1.6]{Gor17} established bulk universality of $\rho_V$ where $V$ is a restriction of tensor products of representations. We do not consider restrictions in this article in the interest of simplicity and due to the existence of local results in this direction.
\end{remark}

An interesting feature of our approach is that the proof of \Cref{thm:multitensor} runs almost parallel with that of \Cref{thm:multisum}. Thus the approach clearly indicates homologous features in the two settings at the level of methods, see \Cref{ssec:outline} for further details.

\subsection{Supersymmetric Lifts} \label{ssec:ssym}
Our main results on edge fluctuations are obtained through new formulas for expected exponential moments for the particle systems arising from \eqref{rmt_model} and \eqref{q_model}. The fundamental objects behind these formulas are the \emph{Schur generating functions} for discrete particle systems on $\Z$, and the \emph{multivariate Bessel generating functions} for particle systems on $\R$, which were used in \cite{GS,BuG15,BuG18,BuG19} to study global fluctuations of a generality of particle systems, including those studied in this article. We introduce these generating functions below, and state our moment formulas (\Cref{thm:moment_bgf,thm:moment_sgf}).

Given $\vec{\ell} \in \cW_{\R}^N$ and $\lambda \in \cW_{\Z}^N$, the multivariate Bessel function and rational Schur function are respectively defined by
\[ \cB_{\vec{\ell}}(z_1,\ldots,z_N) := \frac{\det \begin{pmatrix} e^{z_i \ell_j} \end{pmatrix}_{i,j=1}^N}{\prod_{1 \le i < j \le N} (z_i - z_j)}, \quad \quad \quad \quad  s_\lambda(z_1,\ldots,z_N) := \frac{\det\begin{pmatrix} z_i^{\lambda_j + N - j}\end{pmatrix}_{i,j=1}^N}{\prod_{1 \le i < j \le N} (z_i - z_j)}. \]
Given a random $\vec{\ell} \in \cW_{\R}^N$ and a random $\lambda \in \cW_{\Z}^N$, the multivariate Bessel generating function of $\vec{\ell}$ and the Schur generating function of $\lambda$ are functions in $(z_1,\ldots,z_N)$ defined by
\[ \E \left[ \frac{\cB_{\vec{\ell}}(z_1,\ldots,z_N)}{\cB_{\vec{\ell}}(0,\ldots,0)} \right], \quad \quad \quad \quad \E\left[ \frac{s_\lambda(e^{z_1},\ldots,e^{z_N})}{s_\lambda(1,\ldots,1)} \right]\]
respectively, given that the expectations are absolutely convergent, uniformly in a neighborhood of $(0,\ldots,0)$. We note that our definition of Schur generating function differs from that of Bufetov-Gorin in that theirs is a function of $x_i = e^{z_i}$. The reason for our convention is to facilitate parallel treatment of the Schur and multivariate Bessel generating functions.

\begin{definition}
Let $\Omega \subset \C$ be a neighborhood and $S$ be an analytic symmetric function on $\Omega^N$. We say that a family $\{\wt{S}_k\}_{k\ge 0}$ is a \emph{supersymmetric lift of $S$ on the domain $\Omega$} if $\wt{S}_k(x_1,\ldots,x_{N+k}/y_1,\ldots,y_k)$ is analytic and symmetric in $(x_1,\ldots,x_{N+k}) \in \Omega^{N+k}$ and in $(y_1,\ldots,y_k) \in \Omega^k$, and satisfies
\begin{align*}
\wt{S}_k(x_1,\ldots,x_{N+k}/y_1,\ldots,y_k) \Big|_{x_{N+k} = y_k} &= \wt{S}_{k-1}(x_1,\ldots,x_{N+k-1}/y_1,\ldots,y_{k-1}), \\
\wt{S}_0(x_1,\ldots,x_N) &= S(x_1,\ldots,x_N)
\end{align*}
We drop the subscript and say $\wt{S}$ is a supersymmetric lift of $S$.
\end{definition}

Below, we use the notation $(a^k)$ to denote the $k$-vector with every component equal to $a$.

\begin{theorem} \label{thm:moment_bgf}
Suppose that $\vec{\ell}$ is a random element of $\cW_{\R}^N$ with multivariate Bessel generating function $S:\Omega^N \to \C$ for some domain $\Omega \ni 0$. If $\wt{S}$ is a supersymmetric lift of $S$ on $\Omega$, then for $c_1,\ldots,c_k \in \C$ we have
\begin{align*}
\begin{multlined}
\E\left[ \prod_{i=1}^k \sum_{j=1}^N e^{c_i \ell_j} \right] = \frac{1}{(2\pi\bi)^k} \oint \cdots \oint \wt{S}(u_1 + c_1,\ldots,u_k + c_k,0^N/u_1,\ldots,u_k) \\
\times \prod_{1 \le i < j \le k} \frac{(u_j + c_j - u_i - c_i)(u_j - u_i)}{(u_j - u_i - c_i)(u_j + c_j - u_i)} \prod_{i=1}^k \left( \frac{u_i + c_i}{u_i} \right)^N \frac{du_i}{c_i}
\end{multlined}
\end{align*}
where both the $u_i$- and $u_i + c_i$-contours are contained in $\Omega$ and positively oriented around $0$ for $1 \le i \le k$ and the $u_j$-contour contains $u_i + c_i$ whenever $i < j$, assuming such a set of contours in $\Omega$ exists.
\end{theorem}

\begin{theorem} \label{thm:moment_sgf}
Suppose that $\lambda$ is a random element of $\cW_{\Z}^N$ with Schur generating function $S:\Omega^N \to \C$ for some domain $\Omega \ni 0$. If $\wt{S}$ is a supersymmetric lift of $S$ on $\Omega$, then for $c_1,\ldots,c_k \in \C$ we have
\begin{align*}
\begin{multlined}
\E\left[ \prod_{i=1}^k \sum_{j=1}^N e^{c_i (\lambda_j + N - j)} \right] =  \frac{1}{(2\pi \bi)^k} \oint \cdots \oint \wt{S}(u_1 + c_1,\ldots, u_k + c_k, 0^N/u_1,\ldots,u_k) \\
\times \prod_{1 \le i < j \le k} \frac{(e^{u_j + c_j} - e^{u_i + c_i})(e^{u_j} - e^{u_i})}{(e^{u_j} - e^{u_i + c_i})(e^{u_j + c_j} - e^{u_i})} \prod_{i=1}^k \left( \frac{e^{u_i + c_i} - 1}{e^{u_i} - 1} \right)^N \frac{du_i}{e^{c_i} - 1}
\end{multlined}
\end{align*}
where both the $u_i$- and $u_i + c_i$-contours is contained in $\Omega$ and positively oriented around $0$ for $1 \le i \le k$ and the $u_j$-contour contains $u_i + c_i$ whenever $i < j$, assuming such a set of contours in $\Omega$ exists.
\end{theorem}

The inspiration behind \Cref{thm:moment_bgf,thm:moment_sgf} comes from the work \cite{BC14} where Borodin-Corwin study Macdonald processes to obtain fluctuations of the free energy of the O'Connell-Yor polymer. In particular, they observed that the action of a distinguished family of difference operators, whose eigenfunctions are given by the Macdonald symmetric functions, yield contour integral formulas for observables of Macdonald processes.  We adapt this approach to our setting in which the systems are not Macdonald processes, where the main new ingredient is the formalism of supersymmetric lifts.

To analyze the particle systems arising from \eqref{rmt_model} and \eqref{q_model}, we seek analyzable formulas for supersymmetric lifts of their multivariate Bessel generating functions and Schur generating functions respectively. For the eigenvalues of \eqref{rmt_model}, the multivariate Bessel generating function is given by
\begin{align} \label{eq:model_mbgf}
\prod_{i=1}^n \frac{\cB_{\vec{\ell^{(i)}}}(z_1,\ldots,z_N)}{\cB_{\vec{\ell^{(i)}}}(0^N)}.
\end{align}
Here, the operation of summation corresponds to products of multivariate Bessel functions. We describe the connection between operations and generating functions in detail in \Cref{ssec:examples_lifts}. Similarly, for the particles arising from \eqref{q_model}, the Schur generating function is given by
\begin{align} \label{eq:model_sgf}
\prod_{i=1}^n \frac{s_{\lambda^{(i)}}(e^{z_1},\ldots,e^{z_N})}{s_{\lambda^{(i)}}(1^N)}.
\end{align}

It can be easily seen that supersymmetric lifts respect product (see \Cref{ssec:examples_lifts}). Thus to analyze the eigenvalues of \eqref{rmt_model} through \Cref{thm:moment_bgf}, we seek a supersymmetric lift of $\cB_{\vec{\ell}}$ with a form suitable for asymptotic analysis, and likewise for $s_\lambda$ to study \eqref{q_model}.

Fix $p \in \C$. For $m = N+k$, define
\begin{align*}
& \cB_{\vec{\ell},p}(u_1,\ldots,u_m/v_1,\ldots,v_k) \\
& \quad = \frac{(-1)^{Nk} \prod_{i=1}^m \prod_{j=1}^k (u_i - v_j)}{\prod_{1\le i < j \le m} (u_i - u_j) \prod_{1 \le i < j \le k} (v_j - v_i)}
\det \begin{pmatrix}
\frac{e^{p(u_1 - v_1)}}{u_1 - v_1} & \cdots & \frac{e^{p(u_1 - v_k)}}{u_1 - v_k} &  e^{u_1\ell_1} & \cdots & e^{u_1 \ell_N} \\
\vdots & \ddots & \vdots & \vdots & \ddots & \vdots \\
\frac{e^{p(u_m - v_1)}}{u_m - v_1} & \cdots & \frac{e^{p(u_m - v_k)}}{u_m - v_k} &  e^{u_m\ell_1} & \cdots & e^{u_m \ell_N}
\end{pmatrix}
\end{align*}
Then $\{\cB_{\vec{\ell},p}(u_1,\ldots,u_{N+k}/v_1,\ldots,v_k)\}_{k=0}^\infty$ defines a supersymmetric lift of $\cB_{\vec{\ell}}(u_1,\ldots,u_N)$, stated formally in \Cref{thm:ssym_lift}.

\begin{theorem} \label{cor:ssym_bessel}
Fix $p \in \C \setminus \{\ell_1,\ldots,\ell_N\}$. We have
\[ \frac{\cB_{\vec{\ell},p}(u_1,\ldots,u_k,0^N/v_1,\ldots,v_k)}{\cB_{\vec{\ell}}(0^N)} = \frac{\prod_{i,j=1}^k (u_i - v_j)}{\prod_{1 \le i < j \le k} (u_i - u_j)(v_j - v_i)} \det \left( \frac{1}{u_i - v_j} \frac{\cB_{\vec{\ell},p}(u_i,0^N/v_j)}{\cB_{\vec{\ell}}(0^N)} \right)_{i,j=1}^k. \]
If $u,v \in \C \setminus \{0\}$ such that $u \ne v$, then
\[ \frac{\cB_{\vec{\ell},p}(u,0^N/v)}{\cB_{\vec{\ell}}(0^N)} = (u - v) \left( \frac{v}{u} \right)^N \left( \frac{e^{p(u - v)}}{u - v} + \int_p^\infty \! dw \oint \frac{dz}{2\pi\bi} \cdot \frac{1}{w-z} \frac{e^{zu}}{e^{wv}} \prod_{i=1}^N \frac{w - \ell_i}{z - \ell_i} \right) \]
where the $z$-contour is positively oriented around $\ell_1,\ldots,\ell_N$, and the $w$-contour is a ray from $p$ to $\infty$ which is disjoint from the $z$-contour such that $\Re(w v) > 0$ for $|w|$ large.
\end{theorem}

Supersymmetric lifts are not unique, and here we have a supersymmetric lift for each $p \in \C$. These lifts were motivated by formulas for the supersymmetric Schur functions from \cite{MJ03} and are closely related to finite rank perturbations of normalized multivariate Bessel and Schur functions studied in \cite{GS,CG} where formulas similar to that of \Cref{cor:ssym_bessel} were obtained. We elaborate further on these connections and alternative choices of lifts in \Cref{appendix:supersymmetric}.

With \Cref{cor:ssym_bessel} we can asymptotically analyze these supersymmetric lifts via steepest descent. In the analysis, we must choose $p \in \C$ in such a way that the contours are steepest descent curves.

For the quantized model, we can similarly find that we require supersymmetric lifts of the rational Schur functions. We find that
\begin{align*}
\begin{split}
& \frac{s_{\lambda,p}(e^{u_1},\ldots,e^{u_k},1^N/e^{v_1},\ldots,e^{v_k})}{s_\lambda(1^N)} := \prod_{i=1}^k \frac{1}{e^{v_i}} \left( \frac{u_i}{e^{u_i} - 1} \cdot \frac{e^{v_i} - 1}{v_i} \right)^N \\
& \quad \times \prod_{i,j=1}^k \frac{e^{u_i} - e^{v_j}}{u_i - v_j} \prod_{1 \le i < j \le k} \frac{(u_i - u_j)(v_i- v_j)}{(e^{u_i} - e^{u_j})(e^{v_i} - e^{v_j})} \frac{\cB_{\lambda+\delta_N,p}(u_1,\ldots,u_k,0^N/v_1,\ldots,v_k)}{\cB_{\lambda+\delta_N,p}(0^N)}
\end{split}
\end{align*}
provide a family of supersymmetric lifts of the rational Schur functions, see \Cref{ssec:det_lifts}. It follows that a significant portion of the analysis in the quantized model is joint with that of the random matrix model.

\subsection{Outline} \label{ssec:outline}

Our starting point is the set of formulas \Cref{thm:moment_bgf,thm:moment_sgf} which link our models to supersymmetric lifts of their corresponding generating functions. Through these formulas, we establish \Cref{thm:master} which gives a set of asymptotic conditions under which the largest particles of a particle system converge to the Airy point process. These asymptotic conditions are general and abstract, and essentially say that the supersymmetric lifts asymptotically behave in a manner such that the observables converge to the Laplace transform of the $k$-point correlation functions of the Airy point process. The proofs of these results are contained in \Cref{sec:keystone}.

The remainder of the article is devoted to showing that supersymmetric lifts of the generating functions in our setting (which have the form \eqref{eq:model_mbgf} and \eqref{eq:model_sgf}) satisfy the hypotheses of \Cref{thm:master}. In \Cref{sec:ssym_lifts}, we show that $\cB_{\vec{\ell},p}$ and $s_{\lambda,p}$ (defined in \Cref{ssec:ssym}) give supersymmetric lifts of the multivariate Bessel and Schur functions, and prove \Cref{cor:ssym_bessel}. We study the asymptotics of $\cB_{\vec{\ell},p}$ by steepest descent under the limit where $\tfrac{1}{N} \sum_{i=1}^N \delta_{\ell_i}$ weakly converges to some $\bm \in \cM$. These asymptotics provide the bridge between \Cref{thm:master} and our setting. The core object in our analysis is the function
\[ \cS_u(z) = uz - \int \log(z - x) d\bm(x) - \log u. \]
The steepest descent analysis in \Cref{sec:ssym_lifts} assumes that $\cS_u(z)$ satisfies desired asymptotic properties. This leaves the subsequent sections to verifying that these conditions are met.

In \Cref{sec:submaster}, we show that the hypotheses of \Cref{thm:master} may be simplified in our setting. We combine the asymptotics from \Cref{sec:ssym_lifts} with \Cref{thm:master} so that our main results are reduced to establishing the existence of a critical point at a natural boundary point of the inverse Cauchy transform of the limiting measure, along with a technical inclusion property. An important ingredient in this section is analytic subordination. The subordination phenomenon for addition of random matrices was first observed in \cite{V93}, and has been an important tool in the study of these models, see e.g. \cite{Bel08,Bel14,BES18}. Our development of analytic subordination in the quantized setting is based on a bijective connection between the random matrix and quantized settings. This connection is given by a variant of the Markov-Krein correspondence which was first observed in \cite{BuG15}.

\Cref{sec:applications} is devoted to the proofs of our main results \Cref{thm:multisum,thm:two_sum}, where we check the simplified hypotheses from \Cref{sec:submaster}. The proof of \Cref{thm:multisum} relies on two major ingredients. First, we require monotonicity properties about the Cauchy transform and the fact that the level sets of $\cS_u(z)$ are easily understood for $|z|$ large. Here, $|z|$ large corresponds to the setting where $n$ (the number of summands or tensor factors) is large. Second, we require stability of the hypotheses of \Cref{sec:submaster} under sums. We note that the ideas behind this stability result were inspired by the techniques of \cite{BES18}. To prove \Cref{thm:two_sum}, we prove a more general result \Cref{thm:general_twosum} which follows from controlling the level sets of $\cS_u(z)$ with a monotonicity assumption on $\Re \cS_u(z)$.

We conclude with the proof of \Cref{thm:multitensor} in \Cref{sec:applications_q}. The proof is based on the importation of the stability result used in the proof of \Cref{thm:multisum} into the quantized setting via a variant of the Markov-Krein correspondence.

\begin{notation}
Given a function $f(w_1,\ldots,w_N)$ we write
\[ f(w_1,\ldots,w_k,a^{N-k}) := f(w_1,\ldots,w_k,a,\ldots,a), \]
rather than the usual indication where $a^{N-k}$ is exponentiation. Let $B_r(z) := \{w \in \C: |z - w| < r\}$, $\C^\pm := \{z \in \C: \pm \Im z > 0\}$, $\cW_{\R}^N := \{\vec{x} \in \R^N: x_1 \ge \cdots \ge x_N\}$, and $\cW_{\Z}^N := \{\lambda \in \Z^N: \lambda_1 \ge \cdots \ge \lambda_N\}$. We denote by $\cM$ the set of compactly supported Borel probability measures on $\R$ and $\cM_1$ the subset of $\cM$ of measures with density $\le 1$.
\end{notation}

\section*{Acknowledgements}

I am indebted to Vadim Gorin for his immense role in this work including his original suggestion of the problem of local extremal asymptotics for additive models and their quantized analogues, numerous stimulating discussions, encouragement, and the idea of using the Markov-Krein correspondence as a means to import analytic subordination to the quantized setting, which greatly simplified this work. I also thank Alexei Borodin, Leonid Petrov, and Yi Sun for helpful suggestions and conversations. The author
was partially supported by National Science Foundation Grant DMS-1664619.

\section{Conditions for Airy Fluctuations} \label{sec:keystone}

The main result in this section gives a set of asymptotic conditions on the multivariate Bessel or Schur generating functions, which we recall below, under which Airy fluctuations appear at the edge of the particle system. The asymptotic conditions are given in terms of supersymmetric lifts of these generating functions which we also recall below.

Given variables $\vec{z} = (z_1,\ldots,z_N)$, $\vec{\ell} \in \cW_{\R}^N$, and $\lambda \in \cW_{\Z}^N$, the \emph{multivariate Bessel function} and \emph{rational Schur function} are respectively defined by 
\begin{align} \label{eq:alternant}
\cB_{\vec{\ell}}(z_1,\ldots,z_N) := \frac{\det \begin{pmatrix} e^{z_i \ell_j} \end{pmatrix}_{i,j=1}^N}{\Delta(\vec{z})}, \quad \quad \quad \quad  s_\lambda(\vec{z}) := \frac{\det\begin{pmatrix} z_i^{\lambda_j + N - j}\end{pmatrix}_{i,j=1}^N}{\Delta(\vec{z})}
\end{align}
where $\Delta(\vec{z}) := \prod_{1 \le i < j \le N} (z_i - z_j)$ is the Vandermonde determinant. Both are symmetric functions in $\vec{z}$, and the latter is a homogeneous rational function of degree $\sum_{i=1}^N \lambda_i$ with possible poles at $x_i = 0$.

The \emph{multivariate Bessel generating function} of a random particle system $\vec{\ell} = (\ell_1,\ldots,\ell_N)$ on $\R$ is the map
\[ (z_1,\ldots,z_N) \mapsto \E \left[ \frac{\cB_{\vec{\ell}}(z_1,\ldots,z_N)}{\cB_{\vec{\ell}}(0,\ldots,0)} \right], \]
given that the expectation is absolutely convergent, uniformly in a neighborhood of $(0,\ldots,0)$. If $\lambda$ is a random element of $\cW_{\Z}^N$, then the \emph{Schur generating function} of $\lambda$ is the map
\[ (z_1,\ldots,z_N) \mapsto \E \left[ \frac{s_\lambda(e^{z_1},\ldots,e^{z_N})}{s_\lambda(1,\ldots,1)} \right] \]
given that the expectation is absolutely convergent, uniformly in a neighborhood of $(0,\ldots,0)$.

\begin{definition} \label{def:supersymmetric}
Let $\Omega \subset \C$ be a neighborhood and $S$ be an analytic symmetric function on $\Omega^N$. We say that a family $\{\wt{S}_k\}_{k\ge 0}$ is a \emph{supersymmetric lift of $S$ (on the domain $\Omega$)} if $\wt{S}_k(x_1,\ldots,x_{N+k}/y_1,\ldots,y_k)$ is analytic and symmetric in $(x_1,\ldots,x_{N+k}) \in \Omega^{N+k}$ and in $(y_1,\ldots,y_k) \in \Omega^k$, and satisfies
\begin{align} \label{eq:cancellation}
\begin{split}
\wt{S}_k(x_1,\ldots,x_{N+k}/y_1,\ldots,y_k) \Big|_{x_{N+k} = y_k} &= \wt{S}_{k-1}(x_1,\ldots,x_{N+k-1}/y_1,\ldots,y_{k-1}), \\
\wt{S}_0(x_1,\ldots,x_N) &= S(x_1,\ldots,x_N)
\end{split}
\end{align}
In this case, we drop the subscript and say $\wt{S}$ is a supersymmetric lift of $S$.
\end{definition}

These lifts are not unique, and this freedom of choice will be relevant for us later.

\begin{definition}
Given a domain $\Omega \subset \C$ containing $0$ and a point real point $\fz \in \Omega \setminus \{0\}$, let $\cH_{\fz}(\Omega)$ denote the set of meromorphic functions $\cA$ on $\Omega$ with a unique pole at $0$ such that $\cA'(\fz) = 0$, $\cA''(\fz) > 0$, and $\overline{\cA(z)} = \cA(\overline{z})$.
\end{definition}

\begin{definition}
Given $\cA \in \cH_{\fz}(\Omega)$, let $\Gamma_{\cA}(\Omega,\fz)$ denote the set of all simple closed curves $\gamma$ in $\Omega$ such that $\fz \in \gamma$, $\gamma$ is positively oriented around $0$, and
\[ \Re \cA(\fz) > \Re \cA(z), \]
for every $z \in \gamma \setminus \{\fz\}$.
\end{definition}

\begin{figure}[ht]
    \centering
    \includegraphics[width=0.5\linewidth]{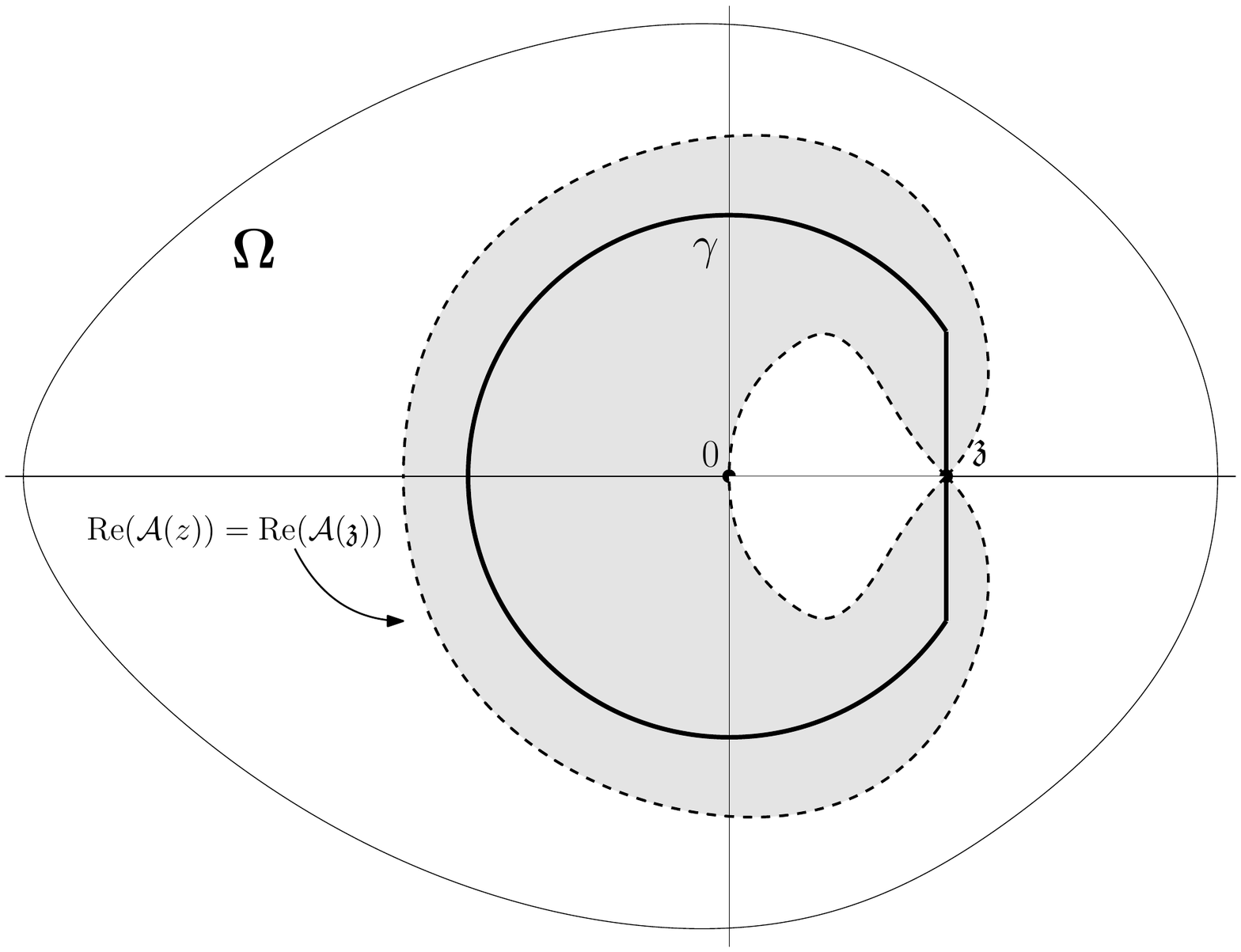}
    \caption{Illustration of an element $\gamma \in \Gamma_{\cA}(\Omega,\fz)$. Shaded region corresponds to $\{ \Re \cA(z) < \Re \cA(\fz)\}$.}
    \label{fig:AEapprop}
\end{figure}

\begin{definition} \label{def:airy_edge}
A sequence of multivariate Bessel generating functions $S_N$ in $N$ is \emph{Airy edge appropriate} if there exists $\cA \in \cH_{\fz}(\Omega)$ for some $\Omega \subset \C$ containing $0$ and a real point $\fz \in \Omega \setminus \{0\}$, a sequence $A_N \in \cH_{\fz}(\Omega)$, and supersymmetric lifts $\wt{S}_N$ of $S_N$ on $N\Omega$ such that
\begin{enumerate}[(i)]
    \item $A_N \to \cA$ uniformly on compact subsets of $\Omega \setminus \{0\}$,
    \item $\Gamma_\cA(\Omega,\fz)$ is nonempty,
    \item for each $k \in \Z_+$ and $c_1,\ldots,c_k > 0$, we have
    \begin{align} \label{ssym_mbgf_limit}
    \begin{split}
    & \wt{S}_N(N u_1 + N^{2/3} c_1,\ldots,N u_k + N^{2/3} c_k, 0^N/ N u_1,\ldots, N u_k) \prod_{i=1}^k \left( \frac{u_i + N^{-1/3} c_i}{u_i} \right)^N \\
    & \quad \quad = F_N(u_1,\ldots,u_k) \exp\left( N \sum_{i=1}^k \int_{u_i}^{u_i + N^{-1/3} c_i} A_N(z) \, dz \right)
    \end{split}
    \end{align}
    where $F_N$ is uniformly bounded over compact subsets of $(\Omega \setminus\{0\})^k$, and if $z_1(N),\ldots,z_k(N) \to \fz$ as $N\to\infty$, then
    \begin{align} \label{FN_approx}
    \lim_{N\to\infty} F_N\left(z_1(N), \ldots, z_k(N) \right) = 1.
    \end{align}
\end{enumerate}
\end{definition}

\begin{definition}
A sequence of Schur generating functions $S_N$ in $N$ is \emph{Airy edge appropriate} if there exists $\cA \in \cH_{\fz}(\Omega)$ for some $\Omega \subset \C$ containing $0$ and a real point $\fz \in \Omega \setminus \{0\}$, a sequence $A_N \in \cH_{\fz}(\Omega)$, and supersymmetric lifts $\wt{S}_N$ of $S_N$ on $\Omega$ satisying (i), (ii) in \Cref{def:airy_edge}, and
\begin{itemize} 
    \item[(iii')] $S_N$ is defined on $\Omega^k$, and for each $k \in \Z_+$ and $c_1,\ldots,c_k > 0$, we have
    \begin{align} \label{schur_version_lim}
    \begin{split}
    & \wt{S}_N(u_1 + N^{-1/3} c_1,\ldots,u_k + N^{-1/3} c_k, 0^N/ u_1,\ldots, u_k) \prod_{i=1}^k \left( \frac{e^{u_i + N^{-1/3} c_i} - 1}{e^{u_i} - 1} \right)^N \\
    & \quad \quad = F_N(u_1,\ldots,u_k) \exp\left( N \sum_{i=1}^k \int_{u_i}^{u_i + N^{-1/3} c_i} A_N(z) \, dz \right)
    \end{split}
    \end{align}
    where $F_N$ is uniformly bounded over compact subsets of $(\Omega \setminus\{0\})^k$, and \eqref{FN_approx} holds.
\end{itemize}
\end{definition}

\begin{theorem} \label{thm:master}
Assume that either
\begin{enumerate}[(i)]
    \item $\vec{\ell} \in \cW_{\R}^N$ is a random particle system varying in $N$ with an Airy edge appropriate sequence of multivariate Bessel generating functions, or
    \item $\vec{\ell} := \tfrac{1}{N}(\lambda_1 + N - 1,\ldots,\lambda_N)$ where $\lambda$ is a random element of $\cW_{\Z}^N$ varying in $N$ with an Airy edge appropriate sequence of Schur generating functions.
\end{enumerate}
Then there exists a sequence $\fz_N$ of critical points of $A_N$ converging to $\fz$ such that
\[  N^{2/3} \frac{\ell_i - A_N(\fz_N)}{(A_N''(\fz_N)/2)^{1/3}} \quad i = 1,\ldots,N\]
converges to the Airy point process as $N\to\infty$.
\end{theorem}

\begin{remark}
Note that our discrete particle systems on $\Z$ must have supports growing of order $N$, whereas our continuous particle systems do not have this restriction. The difference in scaling between Airy appropriateness for the multivariate Bessel and Schur generating functions reflects the fact that under our scaling limits, we preserve the size of the support for continuous particle systems and consider supports growing $O(N)$ for discrete particle systems.
\end{remark}

\subsection{Examples of Supersymmetric Lifts} \label{ssec:examples_lifts}

\begin{example} \label{ex:multiplicative}
For multiplicative functions, there is a natural multiplicative supersymmetric lift. If
\[ F(x_1,\ldots,x_N) = \prod_{i=1}^N f(x_i) \]
for some analytic $f$ on $\Omega$, then
\[ \wt{F}(x_1,\ldots,x_{N+k}/y_1,\ldots,y_k) = \prod_{i=1}^{N+k} f(x_i) \prod_{j=1}^k \frac{1}{f(y_j)} \]
is a supersymmetric lift of $F$ on $\Omega \setminus \{z:f(z) = 0\}$.
\end{example}

\begin{example} \label{ex:algebra}
If $F$ and $G$ are symmetric functions with supersymmetric lifts $\wt{F}$ and $\wt{G}$ on $\Omega$ respectively, then $\wt{F}\cdot \wt{G}$ is a supersymmmetric lift of $F\cdot G$ on $\Omega$.
\end{example}

The key fact which links \Cref{thm:master} with our random matrix and quantized models is the coherency of multivariate Bessel generating functions with the operations of matrix products.

\begin{lemma} \label{bessel_gf}
Let $X$ and $Y$ be independent $N\times N$ unitarily invariant, Hermitian random matrices. If the multivariate Bessel generating functions $S_X,S_Y$ of the eigenvalues of $X,Y$ (resp.) are defined in some neighborhood of $\vec{0}$, then the multivariate Bessel generating function of the eigenvalues of $X + Y$ is given by
\[ S_X(z_1,\ldots,z_N) S_Y(z_1,\ldots,z_N). \]
\end{lemma}

\begin{proof}
It suffices to check these properties for deterministic $X,Y$ with deterministic eigenvalues. This is an immediate consequence of the Harish-Chandra-Itzykson-Zuber integral formula \cite{HC57,IZ80} which states
\[ \frac{\cB_{\vec{\ell}}(z_1,\ldots,z_N)}{\cB_{\vec{\ell}}(0,\ldots,0)} = \prod_{i=1}^{N-1} i! \cdot \frac{\det\left( e^{z_i \ell_j} \right)_{i,j=1}^N}{\Delta(\vec{z}) \Delta(\vec{\ell})} = \int_{\cU(N)} e^{\tr \diag(z_1,\ldots,z_N) U \diag(\ell_1,\ldots,\ell_N) U^*} \, dU. \]
\end{proof}

The rational Schur functions are identified with the irreducible characters of $\cU(N)$. Observe that any class function on $\cU(N)$ is a symmetric function of the eigenvalues of $U \in \cU(N)$. More precisely, let $u_1,\ldots,u_N$ be the eigenvalues of $U$, then a class function on $\cU(N)$ is of the form
\[ U \mapsto f(u_1,\ldots,u_N) \]
for some function $f$ which is symmetric in its variables.

\begin{proposition}[{\cite{Wey97}}] \label{thm:irreducible}
The irreducible character of $\cU(N)$ associated to $\lambda \in \cW_{\Z}^N$ is the symmetric function $s_\lambda(u_1,\ldots,u_N)$ on the eigenvalues of $U \in \cU(N)$.
\end{proposition}

\begin{lemma} \label{schur_gf}
Let $V$ and $W$ be finite dimensional representations. If the Schur generating functions $S_V,S_W$ of $\rho_V,\rho_W$ (resp.) are defined in some neighborhood of $\vec{0}$, then the Schur generating function of $\rho^{V \otimes W}$ is given by
\[ S_V(z_1,\ldots,z_N) S_W(z_1,\ldots,z_N). \]
\end{lemma}

\begin{proof}
For a general finite dimensional representation $V$, the Schur generating function $S_V$ of $\rho_V$ is the normalized character of the representation $V$. Then $S_V\cdot S_W$ is the normalized character of $V\otimes W$, hence the Schur generating of $V\otimes W$.
\end{proof}

The remainder of this section is devoted to the proof of \Cref{thm:master}. We first prove \Cref{thm:moment_bgf,thm:moment_sgf} which we recall below as \Cref{thm:contour_mbgf,thm:contour_sgf}. The idea is to use operators which have the Schur and multivariate Bessel functions as eigenfunctions, resulting in formulas in terms of supersymmetric lifts of the corresponding generating function. The proof of \Cref{thm:master} is then a steepest descent analysis of this formula, under the assumption of Airy appropriateness.

\subsection{Operators and Supersymmetric Lifts}

In this subsection, we prove \Cref{thm:moment_bgf,thm:moment_sgf}.

For $t \in \C$, define the operators $D_c^N$, $\fD_c^N$ on functions in $(z_1,\ldots,z_N)$ by
\[ D_c^N := \sum_{i=1}^N \prod_{j \ne i} \frac{e^{z_i + c} - e^{z_j}}{e^{z_i} - e^{z_j}} \cT_{c,x_i}, \quad \quad \fD_c^N := \sum_{i=1}^N \prod_{j \ne i} \frac{z_i + c - z_j}{z_i - z_j} \cT_{c,z_i} \]
where $\cT_{c,x_i}$ maps $g(x_1,\ldots,x_i,\ldots,x_N)$ to $g(x_1,\ldots,x_i + c,\ldots,x_N)$. By direct computation, we have the eigenrelations
\begin{align*}
D_t^N s_\lambda(e^{z_1},\ldots,e^{z_N}) &= \left( \sum_{i=1}^N e^{c(\lambda_i + N - i)} \right) s_\lambda(e^{z_1},\ldots,e^{z_N}), \quad \lambda \in \cW_{\Z}^N \\
\fD_c^N \cB_{\vec{\ell}}(z_1,\ldots,z_N) &= \left( \sum_{i=1}^N e^{c\ell_i} \right) \cB_{\vec{\ell}}(z_1,\ldots,z_N), \quad \vec{\ell} \in \cW_{\R}^N
\end{align*}
Thus, if $S(z_1,\ldots,z_N)$ is the Schur generating function of a random $\lambda \in \cW_{\Z}^N$, then
\begin{align} \label{eq:obs_sgf}
D_{c_1}^N \cdots D_{c_k}^N S(z_1,\ldots,z_N) \Big|_{z_i = 0} = \E\left[ \prod_{i=1}^k \prod_{j=1}^N e^{c_i(\lambda_j + N - j)} \right]
\end{align}
and if $S(z_1,\ldots,z_N)$ is the multivariate Bessel generating function of a particle system $\{\ell_i\}_{i=1}^N$, then
\begin{align} \label{eq:obs_mbgf}
\fD_{c_1}^N \cdots \fD_{c_k}^N S(x_1,\ldots,x_N) \Big|_{z_i = 0} = \E \left[ \prod_{i=1}^k \sum_{j=1}^N e^{c_i \ell_j} \right]
\end{align}
We require $c_1,\ldots,c_k$ to be sufficiently small, based on the domain of analyticity $\Omega$ for the function above. It suffices to take $|c_1| + \cdots + |c_k| < r$ where $B_r(0) \subset \Omega$.

The action of $D_c^N$ on a function can be described in terms of a supersymmetric lift.

\begin{proposition} \label{thm:contour_sgf}
Let $a_1,\ldots,a_N \in \C$ and $\Omega \subset \C$ be a neighborhood of $\{a_1,\ldots,a_N\}$. If $S: \Omega^N \to \C$ is an analytic symmetric function, $\wt{S}$ is a supersymmetric lift of $S$ on $\Omega$, and $c_1,\ldots,c_k \in \C$ such that $B_{|c_1| + \cdots + |c_k|}(a_i) \subset \Omega$ for $1 \le i \le k$, then
\begin{align} \label{eq:contour_sgf}
\begin{multlined}
D_{c_1}^N \cdots D_{c_k}^N S(a_1,\ldots,a_N) =  \frac{1}{(2\pi \bi)^k} \oint \cdots \oint \wt{S}(a_1,\ldots,a_N, u_1 + c_1,\ldots, u_k + c_k/u_1,\ldots,u_k) \\
\times \prod_{1 \le i < j \le k} \frac{(e^{u_j + c_j} - e^{u_i + c_i})(e^{u_j} - e^{u_i})}{(e^{u_j} - e^{u_i + c_i})(e^{u_j + c_j} - e^{u_i})} \prod_{i=1}^k \left( \prod_{j=1}^N \frac{e^{u_i + c_i} - e^{a_j}}{e^{u_i} - e^{a_j}} \right) \frac{du_i}{e^{c_i} - 1}
\end{multlined}
\end{align}
where the $u_i$-contour is contained in $\Omega$ and positively oriented around the poles $a_1,\ldots,a_N$ for $1 \le i \le k$, and the $u_j$-contour contains $u_i + c_i$ whenever $i < j$.
\end{proposition}
\begin{proof}
By continuity, we may assume $a_1,\ldots,a_N$ are distinct. For $k = 1$, we have
\[ D_c^N S (a_1,\ldots,a_N) = \frac{1}{2\pi \bi} \oint \wt{S}(a_1,\ldots,a_N,u + c/u)\prod_{j=1}^N \frac{e^{u+c} - e^{a_j}}{e^u - e^{a_j}} \, \frac{du}{e^c - 1} \]
by the residue theorem and the cancellation property \eqref{eq:cancellation}. The general case follows by induction and \Cref{ex:multiplicative,ex:algebra}.
\end{proof}

\begin{remark}
The specialization of \Cref{thm:contour_sgf} to the case of multiplicative functions as in \Cref{ex:multiplicative} was first observed in \cite[Proposition 2.2.11]{BC14} in the general context of Macdonald difference operators. The application of the operators $D_c^N$ to general symmetric functions has not been considered before in the literature.
\end{remark}

Analogously, we have a description for the action of $\fD_c^N$.

\begin{proposition} \label{thm:contour_mbgf}
Let $a_1,\ldots,a_N \in \C$ and $\Omega \subset \C$ be a neighborhood of $\{a_1,\ldots,a_N\}$. If $S: \Omega^N \to \C$ is an analytic symmetric function, $\wt{S}$ is a supersymmetric lift of $S$ on $\Omega$, and $c_1,\ldots,c_k \in \C$ such that $B_{|c_1| + \cdots + |c_k|}(a_i) \subset \Omega$ for $1 \le i \le k$, then
\begin{align} \label{eq:contour_mbgf}
\begin{multlined}
\fD_{c_1}^N \cdots \fD_{c_k}^N S(a_1,\ldots,a_N) =  \frac{1}{(2\pi \bi)^k} \oint \cdots \oint \wt{S}(a_1,\ldots,a_N,u_1 + c_1,\ldots,u_k + c_k/u_1,\ldots,u_k) \\
\times \prod_{1 \le i < j \le k} \frac{(u_j + c_j - u_i - c_i)(u_j - u_i)}{(u_j - u_i - c_i)(u_j + c_j - u_i)} \prod_{i=1}^k \left( \prod_{j=1}^N \frac{u_i + c_i - a_j}{u_i - a_j} \right) \frac{du_i}{c_i}
\end{multlined}
\end{align}
where the $u_i$-contour is contained in $\Omega$ and positively oriented around the poles $a_1,\ldots,a_N$ for $1 \le i \le k$. Moreover, the $u_j$-contour contains $u_i + c_i$ whenever $i < j$. 
\end{proposition}

We therefore obtain \Cref{thm:moment_bgf,thm:moment_sgf}.

\subsection{Proof of Theorem \ref{thm:master}}

We first prove the theorem for multivariate Bessel generating functions, then point out the minor modifications required to adapt the argument for Schur generating functions. We begin with a lemma which will be useful for the proof and later parts of this paper. Given $\vec{x} = (x_1,\ldots,x_m)$ and $\vec{y} = (y_1,\ldots,y_n)$, let
\[ D(\vec{x},\vec{y}) = \prod_{i=1}^m \prod_{j=1}^n (x_i + y_j). \]

\begin{lemma} \label{cauchy_det_bd}
Let $\vec{\fc} := (\fc_1,\ldots,\fc_k) \in \R^k$. Given $\xi \ge 0$ and $c > 0$, there exist constants $C_2 > C_1 > 0$ such that
\begin{align} 
\begin{gathered}
C_1 N^{-k\xi} \le \left| \frac{D\left(\vec{u} + N^{-\xi}\vec{\fc}; -\vec{u}\right)}{\Delta\left(\vec{u} + N^{-\xi} \vec{\fc}\right) \Delta(-\vec{u})} \right| \le C_2 N^{-k\xi}
\end{gathered}
\end{align}
for $\vec{u} \in \C^k$ such that
\begin{align} \label{cv_v_dist}
\min_{i \ne j} \left( |u_i - u_j|, |u_i + N^{-\xi} \fc_i - u_j - N^{-\xi} \fc_j|, |u_i + N^{-\xi} \fc_i - u_j| \right) \ge c N^{-\xi}
\end{align}
\end{lemma}
\begin{proof}
Note that
\[ \frac{D\left(\vec{u} + N^{-\xi} \vec{\fc};-\vec{u}\right)}{\Delta\left(\vec{u} + N^{-\xi} \vec{\fc}\right) \Delta(-\vec{u})} = N^{-k\xi} \prod_{i=1}^k \fc_i \cdot \prod_{1 \le i < j \le k} \frac{u_i + N^{-\xi} \fc_i - u_j}{u_i + N^{-\xi} \fc_i - u_j - N^{-\xi} \fc_j} \frac{u_i - u_j - N^{-\xi} \fc_j}{u_i - u_j}. \]
By \eqref{cv_v_dist}, we have constants $C_2' > C_1' > 0$ such that
\begin{align*}
C_1' < \left| \frac{u_i + N^{-\xi} \fc_i - u_j}{u_i + N^{-\xi} \fc_i - u_j - N^{-\xi} \fc_j} \frac{u_i - u_j - N^{-\xi} \fc_j}{u_i - u_j} \right| < C_2'
\end{align*}
for every $1 \le i < j \le k$. The lemma now follows.
\end{proof}

There exists a sequence of critical points $\fz_N$ of $A_N$ such that $\fz_N \to \fz$, by the uniform convergence $A_N \to \cA$ on compact subsets of $\Omega \setminus \{0\}$. We show that for any $\fc_1,\ldots,\fc_k > 0$,
\begin{align} \label{Airy_limit}
\begin{multlined}
\E \left[ \prod_{i=1}^k \sum_{j=1}^N e^{N^{2/3} \fc_i \ell_j} \right] = \Bigg[ e^{N^{2/3} A_N(\fz_N) \sum \fc_i} \cdot \frac{e^{\frac{1}{2} A_N''(\fz_N) \sum \fc_i^3/12}}{(2\pi\bi)^k} \oint_{\Re z_1 = \fv_1} \frac{dz_1}{\fc_1} \cdots \oint_{\Re z_k = \fv_k} \frac{dz_k}{\fc_k} \\
\times \prod_{1 \le i < j \le k} \frac{(z_j + \frac{\fc_j}{2} - z_i - \frac{\fc_i}{2})(z_j - \frac{\fc_j}{2} - z_i + \frac{\fc_i}{2})}{(z_j - \frac{\fc_j}{2} - z_i - \frac{\fc_i}{2})(z_j + \frac{\fc_j}{2} - z_i + \frac{\fc_i}{2})} \cdot \prod_{i=1}^k \exp\left( \frac{\fc_i}{2} A_N''(\fz_N) z_i^2 \right)\Bigg] \cdot (1 + o(1))
\end{multlined}
\end{align}
where the contours are oriented so that the imaginary parts increase and $\fv_1,\ldots,\fv_k \in \R$ satisfy
\begin{align} \label{thm_contour_dist}
\fv_i + \frac{\fc_i}{2} < \fv_j - \frac{\fc_j}{2}, \quad 1 \le i < j \le k.
\end{align}
Indeed, by changing variables $w_i = [A_N''(\fz_N)/2]^{-1/3}z_i$, \Cref{airy_moments} implies that the right hand side of \eqref{Airy_limit} is
\[ e^{N^{2/3} A_N(\fz_N) \sum \fc_i} \E \left[ \prod_{i=1}^k \sum_{j=1}^\infty e^{\fc_i (A_N''(\fz_N)/2)^{1/3} \fa_i} \right] (1 + o(1)) \] 
where $\fa_1,\fa_2,\ldots$ denotes the Airy point process. Thus, \eqref{Airy_limit} implies \Cref{thm:master}. Throughout this proof, we use $C$ as a positive constant independent of $N$ which may vary from line to line.

Our starting point is an integral formula for the left hand side of \eqref{Airy_limit} given by \eqref{eq:obs_mbgf} and \Cref{thm:contour_mbgf}. After changing variables so that $u_i$ is replaced by $Nu_i$, we have
\begin{align*}
& \E \left[ \prod_{i=1}^k \sum_{j=1}^N e^{N^{2/3} \fc_i \ell_j} \right] \\
& \quad \quad = \frac{1}{(2\pi\bi)^k} \oint \cdots \oint \frac{\Delta(\vec{u} + N^{-1/3} \vec{\fc}) \Delta(-\vec{u})}{D(\vec{u} + N^{-1/3} \vec{\fc}; - \vec{u})} \wt{S}_N(N \vec{u} + N^{2/3} \vec{\fc},0^N/N \vec{u}) \prod_{i=1}^k \left( \frac{u_i + N^{-1/3} \fc_i}{u_i} \right)^N du_i.
\end{align*}

\textbf{Step 1.} In this step, we construct the contours for our steepest descent analysis. By assumption, we have the existence of some $\gamma^\infty \in \Gamma_\cA(\Omega,\fz)$. We can choose $\gamma^\infty$ to have bounded arc length and to be a linear segment through $\fz$ in the positive imaginary direction near the critical point $\fz$ since $\cA''(\fz) > 0$. By the uniform convergence $A_N \to \cA$ on compact subsets of $\Omega$, there exists a sequence of curves $\gamma := \gamma^N$ converging to $\gamma^\infty$ such that $\gamma \in \Gamma_{A_N}(\Omega,\fz_N)$ for $N$ large with uniformly bounded arc length in $N$. Furthermore, we may assume that $\gamma$ is a line segment through $\fz_N$ in the positive imaginary direction in a constant order neighborhood of $\fz_N$.

The contours we use are microscopic variations of $\gamma$. Fix $\ft_1,\ldots,\ft_k \in \R$ so that
\begin{align} \label{contour_dist}
\ft_i + \fc_i < \ft_j, \quad 1 \le i < j \le k.
\end{align} 
By perturbing $\gamma$ on the order of $N^{-1/3}$, we can choose contours $\gamma_1 := \gamma_1^N,\ldots,\gamma_k := \gamma_k^N$ so that
\begin{itemize}
    \item $\gamma_i$ is a linear segment in the positive imaginary direction through $\fz_N + N^{-1/3}\ft_i$ in a neighborhood of $\fz_N + N^{-1/3} \ft_i$;
    \item for each $1 \le i < j \le k$, $\gamma_i$ is encircled by the $\gamma_j$ contour such that
    \[ \dist(\gamma_i,\gamma_j) \ge \max(\fc_i,\fc_j) N^{-1/3}, \]
    \item $\gamma_i$ has uniformly bounded arc length in $N$;
    \item $\gamma_i \to \gamma^\infty$ as $N\to\infty$ for $1 \le i \le k$.
\end{itemize}

Choose $\delta > 0$ so that
\begin{align} \label{delta_jm}
N^{-1/3 + \e} \le \delta \le N^{-\e} \end{align}
for some small $\e > 0$. We decompose the contours $\gamma_i$ as
\[ \gamma_i = \gamma_{i,1} \cup \gamma_{i,2} \]
where
\begin{align} \label{gamma_i}
\gamma_{i,1} = B_\delta(\fz + N^{-1/3} \ft_i) \cap \gamma_i, \quad \gamma_{i,2} = \gamma_i \setminus \gamma_{i,1}.
\end{align}
By \eqref{delta_jm} and the fact that $\gamma \in \Gamma_{A_N}(\Omega)$, we have for sufficiently large $N$ 
\begin{align} \label{gammai_bd}
\Re[ A_N(\fz) - A_N(u) ] &> \frac{1}{4} |A_N''(\fz)| \delta^2, \quad \quad u \in \bigcup_{i=1}^k \gamma_{i,2}.
\end{align}

\textbf{Step 2.} In this step, we rewrite our expectation in a form suitable for steepest descent. By \eqref{ssym_mbgf_limit} in our Airy appropriate assumption, we have
\begin{align} \label{moment_proof}
\begin{split}
& \E \left[ \prod_{i=1}^k \sum_{j=1}^N e^{N^{2/3} \fc_i s_j} \right] \\
& \quad \quad = \frac{1}{(2\pi\bi)^k} \oint \cdots \oint F_N(\vec{u}) \frac{\Delta(\vec{u} + N^{-1/3} \vec{\fc}) \Delta(-\vec{u})}{D(\vec{u} + N^{-1/3} \vec{\fc}; - \vec{u})} \exp\left( N \sum_{i=1}^k \int_{u_i}^{u_i + N^{-1/3} \fc_i} A_N(z) \, dz \right) \, du_i
\end{split}
\end{align}
where the $u_i$-contour is $\gamma_i$. By Taylor expanding, the above becomes
\begin{align} \label{A_moment}
\frac{1}{(2\pi\bi)^k} \oint \cdots \oint \wt{F}_N(\vec{u}) \frac{\Delta(\vec{u} + N^{-1/3} \vec{\fc}) \Delta(-\vec{u})}{D(\vec{u} + N^{-1/3} \vec{\fc}; - \vec{u})} \prod_{i=1}^k \exp\left( N^{2/3} \fc_i A_N(u_i) + N^{1/3} \frac{\fc_i^2}{2} A_N'(u_i) + \frac{\fc_i^3}{6} A_N''(u_i) \right)  \, du_i
\end{align}
where $\wt{F}_N$ is a product of $F_N$ and the $e^{o(1)}$ from the Taylor expansion. By Airy appropriateness, $\wt{F}_N$ is uniformly bounded on $\gamma_1\times \cdots \times \gamma_k$ and
\[ \lim_{N\to\infty} \wt{F}_N(\fz_N + N^{-1/3}v_1,\ldots,\fz_N + N^{-1/3}v_k) = 1, \]
for $v_1,\ldots,v_k \in \C$.
Recalling the decomposition $\gamma_i = \gamma_{i,1} \cup \gamma_{i,2}$, write
\begin{align} \label{expand_by_I}
\E \left[ \prod_{i=1}^k \sum_{j=1}^N e^{N^{2/3} \fc_i \ell_j} \right] = \sum_{\Upsilon \in \{1,2\}^k} I_\Upsilon
\end{align}
where $I_\Upsilon$ is as in \eqref{A_moment} except with the $u_i$-contour taken to be $\gamma_{i,\Upsilon_i}$ for $\Upsilon = (\Upsilon_1,\ldots,\Upsilon_k)$.

\textbf{Step 3.} We analyze each $I_\Upsilon$ in \eqref{expand_by_I}. We show that $I_{(1^k)}$ is the dominating term and exhibits the desired asymptotic behavior. For $u_i \in \gamma_{i,1}$, set
\[ u_i = \fz_N + N^{-1/3} \left( z_i - \frac{\fc_i}{2} \right), \quad \quad \Re z_i = \ft_i + \frac{\fc_i}{2},~ \Im z_i \in( -N^{1/3} \delta, N^{1/3} \delta) \]
Let $\gamma_{z_i}$ denote the $z_i$-contour obtained from this transformation on $\gamma_{i,1}$.

Using the fact that $A_N'(\fz_N) = 0$ and \eqref{delta_jm}, we have
\begin{align} \label{A_expanded}
\begin{split}
& N^{2/3} \fc_i A_N(u_i) + \frac{1}{2} N^{1/3} \fc_i^2 A_N'(u_i) + \frac{1}{6} \fc_i^3 A_N''(u_i) \\
& \quad = N^{2/3} \fc_i \left( A_N(\fz_N) + \frac{N^{-2/3}}{2} A_N''(\fz_N) (z_i - \fc_i/2)^2 \right) + \frac{1}{2} \fc_i^2 A_N''(\fz_N) (z_i - \fc_i/2)  + \frac{1}{6} \fc_i^3 A_N''(\fz_N) + o(1) \\
& \quad = N^{2/3} \fc_i A_N(\fz_N) + \frac{1}{2} A_N''(\fz_N) \cdot \frac{\fc_i^3}{12} + \frac{1}{2} \fc_i A_N''(\fz_N) z_i^2 + o(1), \quad \quad \mbox{for}~u_i \in \gamma_{i,1}.
\end{split}
\end{align}

\underline{$I_{(1^k)}$.} Set $\sL_i$ to be the infinite vertical contour $\Re z_i = \fv_i := \ft_i + \frac{\fc_i}{2}$ oriented upwards. Note that \eqref{contour_dist} implies that $\fv_1,\ldots,\fv_k$ satisfy \eqref{thm_contour_dist}. We have
\begin{align} \label{I_leading}
\begin{split}
I_{(1^k)} &= \Biggl[ e^{N^{2/3} A_N(\fz_N) \sum \fc_i} \frac{e^{\frac{1}{2} A_N''(\fz_N) \sum \fc_i^3/12}}{(2\pi\bi)^k} \oint_{\sL_1} dz_1 \cdots \oint_{\sL_k} dz_k \\
& \quad \quad \times \frac{\Delta(\vec{z}+\vec{\fc}/2) \Delta(-\vec{z}+\vec{\fc}/2)}{D(\vec{z}+\vec{\fc}/2;-\vec{z}+\vec{\fc}/2)} \cdot \prod_{i=1}^k \exp\left( \frac{1}{2} \fc_i A_N''(\fz_N) z_i^2 \right)\Biggr] \cdot (1 + o(1))
\end{split}
\end{align}
Indeed, by \eqref{A_expanded},
\begin{align*}
\begin{split}
I_{(1^k)} &= e^{N^{2/3} A_N(\fz_N) \sum \fc_i} \frac{e^{\frac{1}{2} A_N''(\fz_N) \sum \fc_i^3/12}}{(2\pi\bi)^k} \oint_{\gamma_{z_1}} dz_1 \cdots \oint_{\gamma_{z_k}} dz_k \\
& \quad \quad \times \cF_N(\vec{z}) \frac{\Delta(\vec{z}+\vec{\fc}/2) \Delta(-\vec{z}+\vec{\fc}/2)}{D(\vec{z}+\vec{\fc}/2;-\vec{z}+\vec{\fc}/2)} \prod_{i=1}^k \exp\left( \frac{1}{2} \fc_i A_N''(\fz_N) z_i^2 \right) (1 + o(1))
\end{split}
\end{align*}
where we use $\wt{F}_N$ is $1 + o(1)$ on $\gamma_{1,1} \times \cdots \gamma_{k,1}$. Due to the exponential decay coming from $e^{\frac{1}{2}\fc_i A_N''(\fz_N) z_i^2}$, we may replace $\gamma_{z_i}$ with $\sL_i$ in the integral at the cost of a relative $o(1)$ error.

\underline{$I_\Upsilon$, $\Upsilon \ne (1^k)$.} We show that
\begin{align} \label{I_Upsilon_asymp}
I_\Upsilon = I_{(1^k)} \cdot o(1), \quad \quad \Upsilon \ne (1^k).
\end{align}

We have the bound
\[ I_{(1^k)} = C e^{N^{2/3} A_N(\fz_N) \sum \fc_i} \]
for some $C > 0$. We have $\wt{F}_N$ and $N^{-k/3} \Delta(\vec{u}+\vec{\fc})\Delta(\vec{u})/D(\vec{u}+\vec{\fc};-\vec{u})$ are uniformly bounded for $\vec{u} \in \gamma_1 \times \cdots \times \gamma_k$ and in $N$; the latter follows from \Cref{cauchy_det_bd}. This implies that
\[ |I_\Upsilon| \le  C \cdot \frac{N^{k/3}}{(2\pi\bi)^k} \oint_{\gamma_{1,\Upsilon_1}} d|u_1| \cdots \oint_{\gamma_{k,\Upsilon_k}} d|u_k| \prod_{i=1}^k \left| \exp\left( N^{2/3} \fc_i A_N(u_i) + \frac{1}{2} N^{1/3} \fc_i^2 A_N'(u_i) + \frac{1}{6} \fc_i^3 A_N''(u_i)  \right) \right|. \]
For $u_i \in \gamma_{i,\Upsilon_i}$,
\begin{align*}
\left| \exp\left( N^{2/3} \fc_i A_N(u_i) + \frac{1}{2} N^{1/3} \fc_i^2 A_N'(u_i) + \frac{1}{6} \fc_i^3 A_N''(u_i)  \right) \right| \le \left\{  \begin{array}{cc}
    \displaystyle C \cdot e^{N^{2/3} A_N(\fz_N) \fc_i} & \mbox{if $u_i \in \gamma_{i,1}$} \\
    C_1 e^{N^{2/3}(A_N(\fz_N) - C_2 \delta^2)  }  & \mbox{if $u_i \in \gamma_{i,2}$},
\end{array} \right.
\end{align*}
where the bound on $\gamma_{i,2}$ comes from \eqref{gammai_bd}. By \eqref{delta_jm}, the latter bound for $u_i \in \gamma_{i,2}$ is
\[ C_1 e^{N^{2/3} A_N(\fz_N) - C_2 N^{2\e}}. \]
Since $\Upsilon_i = 2$ for some $i$, we obtain
\[ |I_\Upsilon| \le C_1 N^{k/3} e^{N^{2/3} A_N(\fz_N) \sum \fc_i} \cdot e^{-C_2 N^{2\e}} = I_{(1^k)} \cdot o(1). \]

By \eqref{expand_by_I}, we obtain
\[ \E \left[ \prod_{i=1}^k \sum_{j=1}^h e^{N^{2/3} \fc_i s_j} \right] = I_{(1^k)} \cdot (1 + o(1)). \]
From \eqref{I_leading}, we obtain \eqref{Airy_limit}, thus proving the theorem for multivariate Bessel generating functions.

For the case of Schur generating functions, we have by \eqref{eq:contour_sgf} and \Cref{thm:contour_sgf}
\begin{align*}
& \E \left[ \prod_{i=1}^k \sum_{j=1}^N e^{N^{2/3} \fc_i \ell_j} \right] \\
& \quad \quad = \frac{1}{(2\pi\bi)^k} \oint \cdots \oint \frac{\Delta(e^{\vec{u} + N^{-1/3} \vec{\fc}}) \Delta(-e^{\vec{u}})}{D(e^{\vec{u} + N^{-1/3} \vec{\fc}}; - e^{\vec{u}})} \wt{S}_N(\vec{u} + N^{-1/3} \vec{\fc},0^N/ \vec{u}) \prod_{i=1}^k \left( \frac{e^{u_i + N^{-1/3} \fc_i} - 1 }{e^{u_i} - 1} \right)^N du_i.
\end{align*}
Notice that choosing the contours as above, we can apply Airy appropriateness of the Schur generating functions to obtain \eqref{moment_proof} by replacing \[ \frac{\Delta(\vec{u} + N^{-1/3} \vec{\fc}) \Delta(-\vec{u})}{D(\vec{u} + N^{-1/3} \vec{\fc}; - \vec{u})} \quad \quad \mbox{with} \quad \quad \frac{\Delta(e^{\vec{u} + N^{-1/3} \vec{\fc}}) \Delta(-e^{\vec{u}})}{D(e^{\vec{u} + N^{-1/3} \vec{\fc}}; - e^{\vec{u}})}. \]
By an analogue of \Cref{cauchy_det_bd} for the latter term, the remainder of the proof is identical to the multivariate Bessel case.

\section{Supersymmetric Lifts} \label{sec:ssym_lifts}

The central objects in this section are a family of supersymmetric lifts for the Schur and multivariate Bessel functions. The main results of this section are contour integral formulas (\Cref{thm:ssym_bessel} and \eqref{eq:normalized_bessel_schur}) for these lifts, and asymptotics of these lifts (\Cref{bessel_asymp}). In particular, we show that \Cref{thm:ssym_bessel} implies \Cref{cor:ssym_bessel}.

\subsection{A Determinantal Family of Lifts} \label{ssec:det_lifts}

Let $\vec{x} = (x_1,\ldots,x_m)$, $\vec{y} = (y_1,\ldots,y_n)$, $p \in \C$,
\[ D(\vec{x};\vec{y}) := \prod_{i=1}^m \prod_{j=1}^n (x_i + y_j), \quad E_p(\vec{x};\vec{y}) := \begin{pmatrix} \displaystyle \frac{e^{p(x_i-y_j)}}{x_i - y_j} \end{pmatrix}_{\substack{1 \le i \le m \\ 1 \le j \le n}}. \]
Let $\delta_m = (m-1,m-2,\ldots,0)$, and for any $\vec{\ell} = (\ell_1,\ldots,\ell_m) \in \C^m$, let
\[ A_{\vec{\ell}}(\vec{x}) := \begin{pmatrix} e^{x_i \ell_j} \end{pmatrix}_{1 \le i,j \le m}. \]

Given integers $N,k \ge 0$, $p \in \C$, set $\vec{u} = (u_1,\ldots,u_{N+k})$, $\vec{v} = (v_1,\ldots,v_k)$, and define
\begin{align} \label{def:ssym_bessel}
\cB_{\vec{\ell},p}(\vec{u}/\vec{v}) := (-1)^{Nk} \frac{D(\vec{u};-\vec{v})}{\Delta(\vec{u}) \Delta(-\vec{v})} \det \begin{pmatrix} E_p(\vec{u};\vec{v}) & A_{\vec{\ell}}(\vec{u}) \end{pmatrix}
\end{align}
which is valid by extension as an analytic function in $\ell_1,\ldots,\ell_N$, $u_1,\ldots,u_{N+k}$, $v_1,\ldots,v_k \in \C$. Note that this is a doubly symmetric function in $\vec{u} \in \C^{N+k}$ and $\vec{v} \in \C^k$. Given $\lambda \in \cW_{\Z}^N$, define
\begin{align} \label{def:ssym_schur}
s_{\lambda,p}(e^{\vec{u}}/e^{\vec{v}}) := (-1)^{Nk} \prod_{i=1}^k e^{-v_i} \cdot \frac{D(e^{\vec{u}};-e^{\vec{v}})}{\Delta(e^{\vec{u}}) \Delta(-e^{\vec{v}})} \det \begin{pmatrix} E_p(\vec{u};\vec{v}) & A_{\lambda + \delta_N}(\vec{u}) \end{pmatrix}
\end{align}
which is a doubly symmetric function in $\vec{u} \in \C^{N+k}$ and $\vec{v} \in \C^k$.

\begin{remark}
Note that $s_{\lambda,p}(\vec{x}/\vec{y})$ is ill-defined for $\vec{x} = (x_1,\ldots,x_{N+k})$ and $\vec{y} = (y_1,\ldots,y_k)$ without specifying branches of $\log x_i$, $\log y_j$. In this paper, we always write $s_{\lambda,p}(e^{\vec{u}}/e^{\vec{v}})$ or otherwise specify the branch.
\end{remark}

\begin{theorem} \label{thm:ssym_lift}
For fixed $p \in \C$, $\cB_{\vec{\ell},p}(u_1,\ldots,u_{N+k}/v_1,\ldots,v_k)$ and $s_{\lambda,p}(e^{u_1},\ldots,e^{u_{N+k}}/e^{v_1},\ldots,e^{v_k})$ are respective supersymmetric lifts of $\cB_{\vec{\ell}}(u_1,\ldots,u_N)$ and $s_\lambda(e^{u_1},\ldots,e^{u_N})$ on $\C$ for $\vec{\ell} \in \cW_{\R}^N$ and $\lambda \in \cW_{\Z}^N$.
\end{theorem}

\begin{proof}[Proof of \Cref{thm:ssym_lift}]
We prove that $\cB_{\vec{\ell},p}$ is a supersymmetric lift of $\cB_{\vec{\ell}}$ for $\ell_1,\ldots,\ell_N$ distinct. Continuity in $\vec{\ell}$ implies the general statement. The proof for $s_\lambda$ is similar. For $k = 0$,
\[ \cB_{\vec{\ell},p}(u_1,\ldots,u_N) = \cB_{\vec{\ell}}(u_1,\ldots,u_N). \]
It remains to check the cancellation property
\[ \cB_{\vec{\ell},p}(u_1,\ldots,u_{N+k}/v_1,\ldots,v_k) \Big|_{u_1 = v_1} = \cB_{\vec{\ell},p}(u_2,\ldots,u_{N+k}/v_2,\ldots,v_k). \]
In \eqref{def:ssym_bessel}, bring the term $u_1 - v_1$ from $D(\vec{u};-\vec{v})$ into the first column of the matrix in the determinant. Sending $u_1 \to v_1$ and observing
\[ \lim_{u_1 \to v_1} (u_1 - v_1) \frac{e^{p(u_1 - v_1)}}{u_1 - v_1} = 1 \]
proves the cancellation property.
\end{proof}

\begin{theorem} \label{thm:ssym_bessel}
Let $p,\alpha \in \C$, $\vec{\ell} \in \cW_{\R}^N$, $\vec{u} = (u_1,\ldots,u_k)$, $\vec{v} = (v_1,\ldots,v_k) \in \C^k$ where $u_1,\ldots,u_k,v_1,\ldots,v_k$ are distinct. Then
\begin{align} \label{eq:ssym_bessel}
\frac{\cB_{\vec{\ell},p}(\vec{u},0,\alpha,\ldots,(N-1)\alpha/\vec{v})}{\cB_{\vec{\ell}}(0,\alpha,\ldots,(N-1)\alpha)} = \frac{D(\vec{u};-\vec{v})}{\Delta(\vec{u})\Delta(-\vec{v})} \det \left( \frac{1}{u_i - v_j} \frac{\cB_{\vec{\ell},p}(u_i,0,\alpha,\ldots,(N-1)\alpha
/v_j)}{\cB_{\vec{\ell}}(0,\alpha,\ldots,(N-1)\alpha
)} \right)_{i,j=1}^k.
\end{align}
If $\alpha \ne 0$, $p \notin \{\ell_1,\ldots,\ell_N\}$ and $u,v \in \C \setminus\{0,\alpha,\ldots,(N-1)\alpha\}$ such that $u \ne v$, then
\begin{align} \label{eq:ssym_bessel_k=1}
\begin{multlined}
\frac{\cB_{\vec{\ell},p}(u,0,\alpha,\ldots,(N-1)\alpha
/v)}{\cB_{\vec{\ell}}(0,\alpha,\ldots,(N-1)\alpha
)} \\
= (u-v) \prod_{i=1}^N \frac{v - (i-1)\alpha}{u - (i-1)\alpha} \left( \frac{e^{p(u - v)}}{u - v} +  \int_p^\infty \! dw \oint \frac{dz}{2\pi\bi} \cdot \frac{\alpha}{e^{(w-z)\alpha} - 1} \frac{e^{zu}}{e^{wv}} \prod_{i=1}^N \frac{e^{w\alpha} - e^{\ell_i\alpha}}{e^{z\alpha} - e^{\ell_i\alpha}} \right)
\end{multlined}
\end{align}
where the $z$-contour is positively oriented around $\ell_1,\ldots,\ell_N$, and the $w$-contour is a ray from $p$ to $\infty$ which is disjoint from the $z$-contour such that
\[ \min_{0 \le i < N} \Re \Big( w \big( v - (N-i)\alpha \big) \Big) > 0 \]
for $|w|$ large.
\end{theorem}

The proof of \Cref{thm:ssym_bessel} is contained in \Cref{ssec:ssym_contour}. The proof of \Cref{cor:ssym_bessel} is now trivial:

\begin{proof}[Proof of \Cref{cor:ssym_bessel}]
By taking $\alpha \to 0$ in \Cref{thm:ssym_bessel}, we obtain \Cref{cor:ssym_bessel}.
\end{proof}

It turns out that understanding these lifts of multivariate Bessel functions is sufficient in understanding the corresponding lifts of Schur functions. To see why this is so, observe that
\[ s_\lambda(e^{u_1},\ldots,e^{u_N}) = \frac{\Delta(u_1,\ldots,u_N)}{\Delta(e^{u_1},\ldots,e^{u_N})} \cB_{\lambda + \delta_N} (u_1,\ldots,u_N) \]
and
\[ s_{\lambda,p}(e^{\vec{u}}/e^{\vec{v}}) = \prod_{i=1}^k e^{-v_i} \cdot \frac{D(e^{\vec{u}};-e^{\vec{v}})}{D(\vec{u};-\vec{v})} \frac{\Delta(\vec{u}) \Delta(\vec{v})}{\Delta(e^{\vec{u}}) \Delta(e^{\vec{v}})} \cB_{\lambda + \delta_N,p}(\vec{u}/\vec{v};N,\alpha). \]
for $\vec{u} = (u_1,\ldots,u_{N+k})$ and $\vec{v} = (v_1,\ldots,v_k)$. For our particular application, the above implies
\begin{align} \label{eq:normalized_bessel_schur}
\begin{split}
& \frac{s_{\lambda,p}(e^{u_1},\ldots,e^{u_k},1^N/e^{v_1},\ldots,e^{v_k})}{s_\lambda(1^N)} = \prod_{i=1}^k \frac{1}{e^{v_i}} \left( \frac{u_i}{e^{u_i} - 1} \cdot \frac{e^{v_i} - 1}{v_i} \right)^N \\
& \quad \times \frac{D(e^{u_1},\ldots,e^{u_k};-e^{v_1},\ldots,-e^{v_k})}{D(u_1,\ldots,u_k;-v_1,\ldots,-v_k)} \frac{\Delta(u_1,\ldots,u_k) \Delta(v_1,\ldots,v_k)}{\Delta(e^{u_1},\ldots,e^{u_k}) \Delta(e^{v_1},\ldots,e^{v_k})} \frac{\cB_{\lambda+\delta_N,p}(u_1,\ldots,u_k,0^N/v_1,\ldots,v_k)}{\cB_{\lambda+\delta_N,p}(0^N)}
\end{split}
\end{align}
where the branches of $1$ in $1^N$ are chosen to have argument $0$.

\subsection{Asymptotics of Lifts}

The goal of this subsection is to obtain asymptotics for the supersymmetric lifts of multivariate Bessel functions introduced earlier. These asymptotics rely on steepest descent analysis. We begin with several definitions and preliminary propositions. For the most part, these definitions are conditions for which the steepest descent analysis is made possible.

The \emph{Cauchy transform} of a finite Borel measure $\bm$ is the map $G_\bm: \C \setminus \supp \bm \to \C$ defined by
\[ G_\bm(z) = \int_{\R} \! \frac{d\bm(x)}{z - x}. \]
Fixing $\bm \in \cM$, define
\[ \cS_u(z) := zu - \int \! \log(z - x) d\bm(x) - \log u. \]
We have the relation
\[ \frac{\partial}{\partial z} \cS_u(z) = u - G_\bm(z), \quad \quad z \in \C \setminus \supp \bm. \]
For $z_0 \in \C \setminus \supp \bm$, let
\begin{align*}
\cD_{u,z_0}^- & := \{ z \in \C \setminus \supp \bm: \Re \cS_u(z) < \Re \cS_u(z_0)\} \\
\cD_{u,z_0}^+ & := \{ z \in \C \setminus \supp \bm: \Re \cS_u(z) > \Re \cS_u(z_0)\}
\end{align*}
The regions above are where our steepest descent contours will lie, where $z_0$ will be given by a suitable saddle point. The next definition provides the conditions we would like for this saddle point (which we will denote by $z_u$) to satisfy.

\begin{definition} \label{def:suitable}
Let $\fO$ be the set of $u \in \C$ such that there exists $z_u \in \C \setminus \supp \bm$ satisfying
\begin{itemize}
    \item[(i)] $u = G_\bm(z_u)$ and $G_\bm'(z_u) \ne 0$;
    \item[(ii)] there exists a simple closed curve $\gamma$ which is positively oriented around $\supp \bm$ and contains $z_u$ such that
    \[ \gamma \setminus \{z_u\} \subset \cD_{u,z_u}^-. \]
\end{itemize}
Given $\fp \in \C \setminus \supp \bm$, let $\fO_\fp$ be the set of $u \in \fO$ such that
\begin{itemize}
    \item[(ii')] there exists a simple closed curve $\gamma$ which is positive oriented around $\supp \bm \cup \{\fp\}$ and contains $z_u$ such that
    \[ \gamma \setminus \{z_u\} \subset \cD_{u,z_u}^-. \]
    \item[(iii)] $\fp$ is in a component of $\cD_{u,z_u}^+$ whose boundary contains $z_u$.
\end{itemize}
Clearly, $\fO := \bigcup_{\fp \in \C \setminus \supp \bm} \fO_\fp$ and both $\fO_\fp$ and $\fO$ are open subsets of $\C$.
\end{definition}

The distinguished point $\fp$ in the definition of $\fO_{\fp}$ arises in the steepest descent analysis as a base point of a contour which is not closed. Indeed, we recall that the $w$-contour from \Cref{cor:ssym_bessel} is not a closed contour, and control of this base point is the motivation for the definition of $\fO_{\fp}$.

\begin{figure}[ht]
    \centering
    \includegraphics[width=0.5\linewidth]{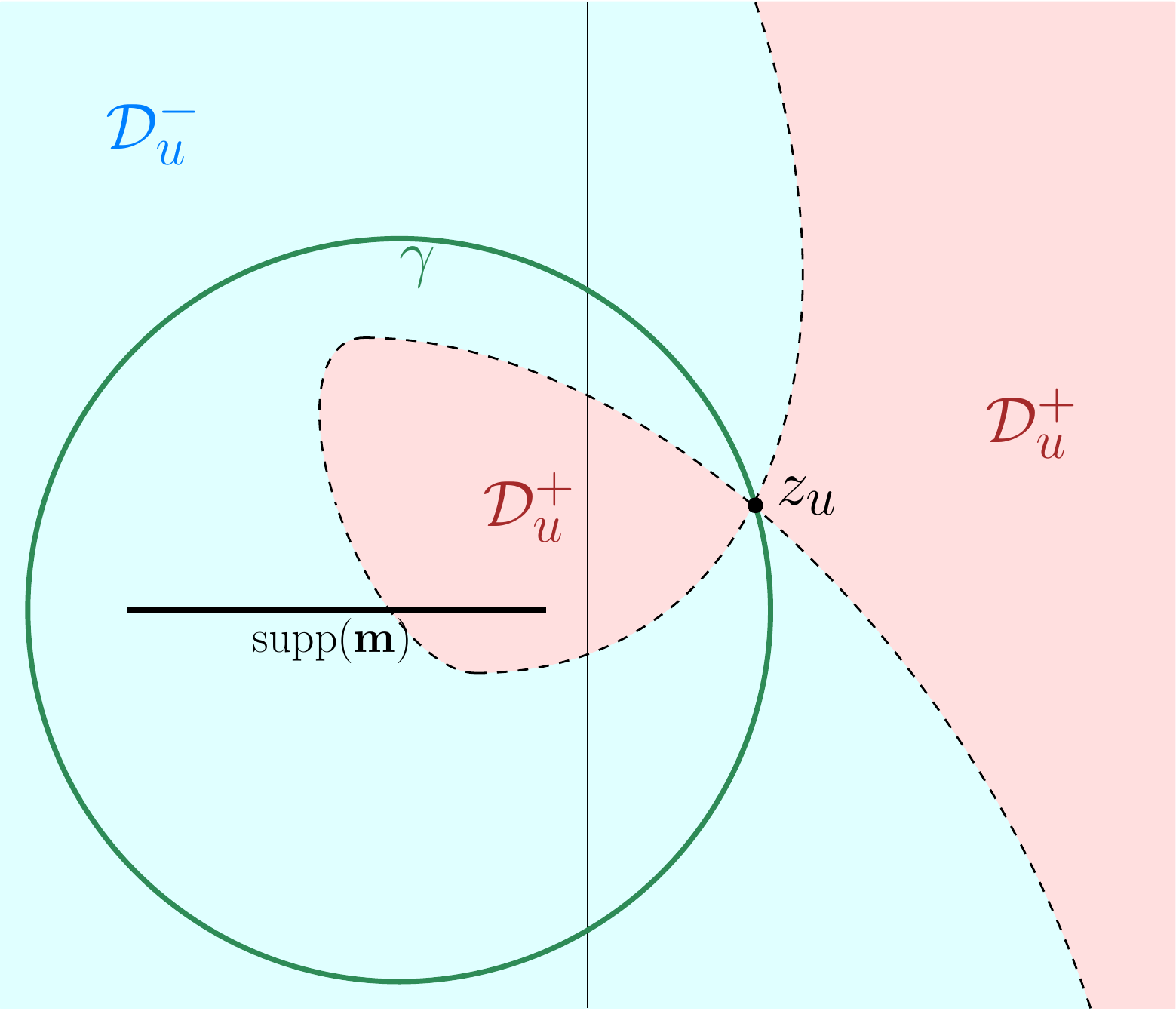}
    \caption{Level lines (dashed) of $\cS_u$ for $u \in \fO$}
    \label{fig:level_lines}
\end{figure}

\begin{proposition} \label{thm:component}
Let $z_0 \in \C \setminus \supp \bm$, $u:= G(z_0)$, and $c \in \R$. Then there exists a unique unbounded connected component of $\{z \in \C \setminus \supp \bm: \Re \cS_u(z) > c\}$ and likewise for $\{z \in \C \setminus \supp \bm: \Re \cS_u(z) < c\}$.
\end{proposition}

\begin{proposition} \label{thm:level_lines}
If $u \in \C$ such that there exists $z_u \in \C \setminus \supp \bm$ satisfying
\begin{itemize}
    \item $u = G(z_u)$ and $G'(z_u) \ne 0$;
    \item there exists a simple closed curve $\gamma$ positively oriented around $\supp \bm$ and through $z_u$ such that
    \[ \gamma \setminus \{z_u\} \subset \cD_{u,z_u}^- \]
\end{itemize}
then $z_u$ is the unique such point and $u \in \fO$. Furthermore, $z_u$ is contained in the boundaries of exactly three connected components of $\cD_{u,z_u}^- \cup \cD_{u,z_u}^+$. These components are the unbounded components of $\cD_{u,z_u}^-$, $\cD_{u,z_u}^+$, and a bounded component of $\cD_{u,z_u}^+$ (see \Cref{fig:level_lines}).
\end{proposition}

\Cref{thm:component,thm:level_lines} are proved in \Cref{ssec:component}. We note that in particular, \Cref{thm:component} implies the existence of a unique unbounded component of $\cD_u^\pm$.

Since the existence of $z_u$ implies its uniqueness, we may write $G_\bm^{[-1]}(u) := z_u$ and $\cD_u^\pm := \cD_{u,z_u}^\pm$ for $u \in \fO$. Then $G_\bm^{[-1]}$ is a right inverse of $G_\bm$, but not necessarily a left inverse. We emphasize that we use the boxed superscript $[-1]$ to indicate that this inverse is defined only on $\fO$. However, we will be able to work strictly with this choice of inverse for the remainder of the paper. If there is ambiguity in the measure $\bm$, we include an additional $\bm$ in the subscript: $\cS_{\bm,u}$, $\cD_{\bm,u,z_0}^\pm$, $\cD_{\bm,u}^\pm$, $\fO_{\bm,\fp}$, $\fO_\bm$.

\begin{assumption} \label{assum:m_conv}
Suppose $\vec{\ell} \in\cW_{\R}^N$ and $\bm \in \cM$ such that
\[ \lim_{N\to\infty} \frac{1}{N} \sum_{j=1}^N \delta_{\ell_j} = \bm \quad \mbox{weakly}, \quad \quad \lim_{N\to\infty} \sup_{1 \le i \le N} \dist(\ell_i, \supp \bm) = 0. \]
\end{assumption}

Under \Cref{assum:m_conv} with $\mathrm{m}_N := \tfrac{1}{N} \sum_{j=1}^N \delta_{\ell_j}$, we have $G_{\mathrm{m}_N}, G_{\mathrm{m}_N}'$ converge to $G_\bm, G_\bm'$ uniformly on compact subsets of $\fO_\bm$. Thus
\[ \lim_{N\to\infty} G_{\mathrm{m}_N}^{[-1]}(u) = G_\bm^{[-1]}(u) \]
uniformly over compact subsets of $\fO_\bm$.

\begin{theorem} \label{bessel_asymp}
Under \Cref{assum:m_conv} with $\mathrm{m}_N := \tfrac{1}{N} \sum_{j=1}^N \delta_{\ell_j}$, if $\fp \in \C \setminus \supp \bm$ and $\cK \subset \fO_{\bm,\fp}$ compact, then
\begin{align} \label{eq:bessel_asymp}
\begin{split}
& \frac{\cB_{\vec{\ell},\fp}(N\vec{u},0^N/N\vec{v})}{\cB_{\vec{\ell}}(0^N)} = \frac{D(\vec{u};-\vec{v})}{\Delta(\vec{u}) \Delta(-\vec{v})} \frac{\Delta(G_{\mathrm{m}_N}^{[-1]}(\vec{u})) \Delta(-G_{\mathrm{m}_N}^{[-1]}(\vec{v}))}{D(G_{\mathrm{m}_N}^{[-1]}(\vec{u});-G_{\mathrm{m}_N}^{[-1]}(\vec{v}))} \\
& \quad \quad \times \prod_{i=1}^k \frac{1}{\sqrt{G_{\mathrm{m}_N}'(G_{\mathrm{m}_N}^{[-1]}(u_i)) G_{\mathrm{m}_N}'(G_{\mathrm{m}_N}^{[-1]}(v_i))}} \exp\Bigg( N\Big(\cS_{\mathrm{m}_N,u_i}(G_{\mathrm{m}_N}^{[-1]}(u_i)) - \cS_{\mathrm{m}_N,v_i}(G_{\mathrm{m}_N}^{[-1]}(v_i))\Big) \Bigg)(1 + o(1))
\end{split}
\end{align}
uniformly over $(\vec{u},\vec{v}) = (u_1,\ldots,u_k,v_1,\ldots,v_k) \in \cK^{2k}$. The branches are chosen so that $G_{\mathrm{m}_N}'\left(G_{\mathrm{m}_N}^{[-1]}(u)\right)^{-1/2}$ points locally from $G_{\mathrm{m}_N}^{[-1]}(u)$ in the direction of the unbounded component of $\cD_{\mathrm{m}_N,u}^+$.
\end{theorem}

By \eqref{eq:normalized_bessel_schur}, \Cref{bessel_asymp} implies asymptotics for $s_{\lambda,p}$.

We divide the remainder of this section into three parts. In \Cref{ssec:ssym_contour}, we derive a contour integral formula for the normalized lift $\cB_{\vec{\ell},p}$. In \Cref{ssec:component}, we provide the proofs of the results \Cref{thm:component,thm:level_lines} about $\cS_u$. In \Cref{ssec:bessel_asymp}, we use these formulas to prove \Cref{single_bessel}.

\subsection{Contour Integral Formula} \label{ssec:ssym_contour}

We prove \Cref{thm:ssym_bessel} through a formula for more general normalizations of $\cB_{\vec{\ell},p}$.

\begin{proposition} \label{thm:general_ssym_bessel}
Let $\vec{\ell} \in \cW_{\R}^N$, $\vec{\xi} := (\xi_1,\ldots,\xi_N) \in \C^N$, $\vec{u} := (u_1,\ldots,u_k)$, $\vec{v} := (v_1,\ldots,v_k) \in (\C \setminus \{\xi_1,\ldots,\xi_N\})^k$ such that $u_1,\ldots,u_k,$ $v_1,\ldots,v_k$ are distinct. Then
\begin{align} \label{eq:general_formula}
\frac{\cB_{\vec{\ell},p}(\vec{u},\vec{\xi}/\vec{v})}{\cB_{\vec{\ell}}(\vec{\xi})} = \frac{D(\vec{u};-\vec{v})}{\Delta(\vec{u}) \Delta(-\vec{v})} \det \left( \frac{1}{u_i - v_j} \frac{\cB_{\vec{\ell},p}(u_i,\vec{\xi}/v_j)}{\cB_{\vec{\ell}}(\vec{\xi})} \right)_{i,j=1}^k.
\end{align}
Define the vectors
\[ \mathrm{a}(u) = (e^{\ell_1 u},\ldots,e^{\ell_N u}), \quad \mathrm{b}(v) = \left( \frac{e^{p(\xi_1 - v)}}{v - \xi_1}, \frac{e^{p(\xi_2 - v)}}{v - \xi_2},\ldots, \frac{e^{p(\xi_N - v)}}{v - \xi_N} \right) \]
which we view as column vectors, and denote the transpose of a vector $\mathrm{u}$ by $\mathrm{u}'$. If $\xi_1,\ldots,\xi_N$ are distinct and $u,v \in \C \setminus \{\xi_1,\ldots,\xi_N\}$ such that $u \ne v$, then
\begin{align} \label{eq:general_formula_k=1}
\frac{\cB_{\vec{\ell},p}(u,\vec{\xi}/v)}{\cB_{\vec{\ell}}(\vec{\xi})} = (u - v) \prod_{i=1}^N \frac{v - \xi_i}{u - \xi_i} \left( \frac{e^{p(u - v)}}{u - v} + \mathrm{a}(u)' A_{\vec{\ell}}(\vec{\xi})^{-1} \mathrm{b}(v) \right).
\end{align}
\end{proposition}

\begin{proof}
We may assume $\ell_1,\ldots,\ell_N$ are distinct by continuity. Suppose $u_1,\ldots,u_k$, $v_1,\ldots,v_k$, $\xi_1,\ldots,\xi_N$ are distinct. By \eqref{def:ssym_bessel}, we have
\[ \cB_{\vec{\ell},p}(\vec{u},\vec{\xi}/\vec{v}) = \frac{D(\vec{v};-\vec{\xi})}{D(\vec{u};-\vec{\xi})} \frac{D(\vec{u};-\vec{v})}{\Delta(\vec{u}) \Delta(-\vec{v})} \frac{1}{\Delta(\vec{\xi})} \det \begin{pmatrix} E_p(\vec{u};\vec{v}) & A_{\vec{\ell}}(\vec{u}) \\ E_p(\vec{\xi};\vec{v}) & A_{\vec{\ell}}(\vec{\xi}) \end{pmatrix}. \]
Using the block matrix determinant formula
\[ \det \begin{pmatrix} M_{11} & M_{12} \\ M_{21} & M_{22} \end{pmatrix} = \det\begin{pmatrix} M_{22} \end{pmatrix} \det \begin{pmatrix} M_{11} - M_{12}M_{22}^{-1}M_{21} \end{pmatrix} \]
where $M_{11},M_{22}$ are square matrices and $M_{22}$ is invertible, we obtain
\[ \cB_{\vec{\ell},p}(\vec{u},\vec{\xi}/\vec{v}) = \frac{D(\vec{v};-\vec{\xi})}{D(\vec{u};-\vec{\xi})} \frac{D(\vec{u};-\vec{v})}{\Delta(\vec{u}) \Delta(-\vec{v})} \cB_{\vec{\ell}}(\vec{\xi}) \det \begin{pmatrix} E_p(\vec{u};\vec{v}) - A_{\vec{\ell}}(\vec{u}) A_{\vec{\ell}}(\vec{\xi})^{-1} E_p(\vec{\xi};\vec{v}) \end{pmatrix}  \]
where we used the alternant formula for multivariate Bessel functions \eqref{eq:alternant}. Then
\begin{align} \label{eq:ssym_schur_proof}
\frac{\cB_{\vec{\ell},p}(\vec{u},\vec{\xi}/\vec{v})}{\cB_{\vec{\ell}}(\vec{\xi})} = \frac{D(\vec{u};-\vec{v})}{\Delta(\vec{u}) \Delta(-\vec{v})} \det \left( \prod_{l=1}^N \frac{v_j - \xi_l}{u_i - \xi_l} \left( \frac{e^{p(u_i - v_j)}}{u_i - v_j} + \mathrm{a}(u_i)' A_{\vec{\ell}}(\vec{\xi})^{-1} \mathrm{b}(v_j) \right) \right)_{i,j=1}^k.
\end{align}
When $k = 1$, this gives us \eqref{eq:general_formula_k=1}. For general $k$, observe that the $(i,j)$ entry of the matrix in \eqref{eq:ssym_schur_proof} is
\[ \frac{1}{u_i - v_j} \frac{\cB_{\vec{\ell},p}(u_i,\vec{\xi}/v_j)}{\cB_{\vec{\ell}}(\vec{\xi})}. \]
This proves \eqref{eq:general_formula}, where the general case for $\xi$ with possibly indistinct components follows from continuity. 
\end{proof}

\begin{proof}[Proof of \Cref{thm:ssym_bessel}]
We have \eqref{eq:ssym_bessel} as a direct consequence of \eqref{eq:general_formula} with $\Xi = ((N-1)\alpha,(N-2)\alpha,\ldots,0)$. We now prove \eqref{eq:ssym_bessel_k=1}. Using \eqref{eq:general_formula_k=1}, we have
\begin{align} \label{eq:geossym_schur_proof}
\begin{split}
&\frac{\cB_{\vec{\ell},p}(u,0,\alpha,\ldots,(N-1)\alpha
/v)}{\cB_{\vec{\ell}}(0,\alpha,\ldots,(N-1)\alpha
)} \\
& \quad = (u - v) \prod_{i=1}^N \frac{v - (i-1)\alpha}{u - (i-1)\alpha} \left( \frac{e^{p(u - v)}}{u - v} + \sum_{a,b=1}^N e^{\ell_a u} \left[A_{\vec{\ell}}((N-1)\alpha,\ldots,0)^{-1}\right]_{a,b} \frac{e^{p(N-b)\alpha - pv}}{v - (N-b)\alpha} \right)
\end{split}
\end{align}
and note that
\[ A_{\vec{\ell}}((N-1)\alpha,(N-2)\alpha,\ldots,0) = A_{\delta_N}(\ell_1 \alpha,\ell_2 \alpha,\ldots,\ell_N \alpha)'. \]
This is a Vandermonde matrix in $e^{\ell_1 \alpha},\ldots,e^{\ell_N \alpha}$ whose inverse is given by
\[ \left[A_{\vec{\ell}}((N-1)\alpha,\ldots,0)^{-1}\right]_{a,b} = (-1)^{b-1} e_{b-1}(e^{\ell_1 \alpha},\ldots,\widehat{e^{\ell_a \alpha}},\ldots,e^{\ell_N \alpha}) \prod_{i \ne a} \frac{1}{e^{\ell_a \alpha} - e^{\ell_i \alpha}} \]
where the $e_b$'s denote the elementary symmetric polynomials. The summation over $a,b$ in \eqref{eq:geossym_schur_proof} becomes
\[ \sum_{a=1}^N e^{\ell_a u} \prod_{i\ne a} \frac{1}{e^{\ell_a \alpha} - e^{\ell_i \alpha}} \sum_{b=1}^N (-1)^{b-1} e_{b-1}(e^{\ell_1 \alpha},\ldots,\widehat{e^{\ell_a \alpha}},\ldots,e^{\ell_N \alpha}) \int_p^\infty \frac{e^{w(N-b)\alpha}}{e^{wv}} \, dw \]
where the hat indicates omission of the indicated term and where the line integral is a ray from $p$ to $\infty$ such that
\[ \Re\Big(w\big((N-b)\alpha - v \big) \Big) < 0 \]
for every $b = 1,\ldots,N$ and $|w|$ large. By the generating series for the elementary symmetric polynomials, this becomes
\[ \sum_{a=1}^N \int_p^\infty \! \frac{e^{\ell_a u}}{e^{w v}} \prod_{i \ne a} \frac{e^{w \alpha} - e^{\ell_i \alpha}}{e^{\ell_a \alpha} - e^{\ell_i \alpha}} \, dw = \int_p^\infty \! dw \oint \frac{dz}{2\pi\bi} \cdot \frac{\alpha}{e^{(w-z)\alpha} - 1} \frac{e^{zu}}{e^{wv}} \prod_{i=1}^N \frac{e^{w\alpha} - e^{\ell_i\alpha}}{e^{z\alpha} - e^{\ell_i\alpha}} \]
where the $z$-contour is positively oriented around $\ell_1,\ldots,\ell_N$ and does not intersect the $w$-contour. Inputting the double contour integral formula into \eqref{eq:geossym_schur_proof} yields \eqref{eq:ssym_bessel_k=1}.
\end{proof}

\subsection{Components and Level Lines} \label{ssec:component}

In this section we prove \Cref{thm:component,thm:level_lines}.

\begin{proof}[Proof of \Cref{thm:component}]
Observe that for any $\e > 0$, there exists $r_0 > 0$, which can be chosen to depend continuously in $c$, $\inf \supp \bm$, $\sup \supp \bm$, large enough so that
\begin{gather} \label{eq:arg_contain}
\begin{split}
\{ z \in \C \setminus \supp \bm: |z| > r_0, \Re(zu) < -\e|z| \} \subset \{z: \Re \cS_u(z) < c \}, \\
\{ z \in \C \setminus \supp \bm: |z| > r_0, \Re(zu) > \e|z| \} \subset \{z: \Re \cS_u(z) > c\}.
\end{split}
\end{gather}
which proves the existence of unbounded connected components of $\{\Re \cS_u(z) > c\}$ and $\{\Re \cS_u(z) < c\}$. Indeed, \eqref{eq:arg_contain} is implied by the domination of the linear term in
\begin{align*}
\Re \cS_u(z) & = \Re(zu) - \int \log \left| z - x \right| d\bm(x) - \Re \log u
\end{align*}
as $|z| \to \infty$. 

To prove the uniqueness of unbounded components, we show that, $\Re \cS_u(z) = \Re \cS_u(|z|e^{\bi \theta})$ as a function of $\theta$ with $|z| = r \ge r_0$ fixed is monotone over each of the two arcs in $\{|\Re(zu)| \le \varepsilon r: |z| = r\}$ for $r$ sufficiently large. Indeed, this monotonicity implies that the boundary between $\{\Re \cS_u(z) > c\}$ and $\{\Re \cS_u(z) < c\}$ is crossed exactly once in each of these arcs for all $r \ge r_0$, which can only be the case if the unbounded components are unique. If $z = x + \bi y$ and $f$ is analytic in $z$, then
\[ f'(z) = (\Re f)_x(z) - \bi (\Re f)_y(z) \]
by Cauchy-Riemann. This implies
\[ \Re\left[ \frac{d}{dt} f(\gamma(t)) \right] = \Re[ f'(\gamma(t)) \gamma'(t)] = (\Re f)_x(\gamma(t)) \Re \gamma'(t) + (\Re f)_y(\gamma(t)) \Im \gamma'(t) = \frac{d}{dt} \Re f(\gamma(t)) \]
where in the latter we view $\gamma(t)$ as a curve in $\R^2$. Using this fact and setting $z = |z| e^{\bi \theta}$, we have
\[ \frac{d}{d\theta} \Re \cS_u(z) = \Re\left[ \frac{d}{d\theta} \cS_u(z) \right] = \Re\left(\bi z u - \int \frac{\bi z \, d\bm(x)}{z - x} \right) = \Re(\bi z u) + O(1) \]
as $|z| \to \infty$, where the $O(1)$ term can be bounded by a constant order term depending continuously in $c, \inf \supp \bm, \sup \supp \bm$. If $r$ is large and $\e$ is small, $\Re(\bi z u)/|z|$ is bounded away from zero on each of the arcs of $\{|\Re(zu)| \le \varepsilon r: |z| = r\}$. Thus we obtain the desired monotonicity statement.
\end{proof}

\begin{proof}[Proof of \Cref{thm:level_lines}]
We show that $z_u$ is contained in the boundaries of exactly three components of $\cD_{u,z_u}^- \cup \cD_{u,z_u}^+$: the unbounded components of $\cD_{u,z_u}^-$ and $\cD_{u,z_u}^+$ and a bounded component of $\cD_{u,z_u}^+$. We prove uniqueness afterwards.

Since $z_u$ is a critical point of $\cS_u$ such that $\cS_u''(z_u) = G'(z_u) \ne 0$, we have $\cS_u$ is locally quadratic at $z_u$. Thus, for $\e$ small, the set $B_\e(z_0) \setminus \{\Re \cS_u(z) = \Re \cS_u(z_u)\}$ consists of four disjoint sets, two belonging to $\cD_{u,z_u}^+$ and the other two to $\cD_{u,z_u}^-$. Each piece belongs to a component of $\cD_{u,z_u}^+$ or $\cD_{u,z_u}^-$. We denote these three components by $\mathrm{D}_1^+,\mathrm{D}_2^+ \subset \cD_{u,z_u}^+$ and $\mathrm{D}_1^-,\mathrm{D}_2^- \subset \cD_{u,z_u}^-$.

By \Cref{def:suitable}, there exists a curve $\gamma$ positively oriented around $\supp \bm$ and containing $G^{[-1]}(u)$ such that $\gamma \setminus \{G^{[-1]}(u)\} \subset \cD_{u,z_u}^-$. This implies that $\mathrm{D}^- := \mathrm{D}_1^- = \mathrm{D}_2^-$, one of $\mathrm{D}_1^+,\mathrm{D}_2^+$ (say $\mathrm{D}_1^+$) is a bounded set, and $\mathrm{D}_2^+$ does not intersect $\supp \bm$. By harmonicity of $\Re \cS_u(z)$ and the maximum principle, we must have that $\mathrm{D}_2^+$ is unbounded.

To see that $\mathrm{D}^-$ is unbounded, first suppose $\Re u \ge 0$. Then $\Re \cS_u(t)$ increases on $\{t < \inf \supp \bm\}$. Since $\gamma \subset \mathrm{D}^-$ encircles $\supp \bm$, there exists $a \in \mathrm{D}^-$ such that $a < \inf \supp \bm$. Then $(-\infty,a] \subset \mathrm{D}^-$ by the noted monotonicity. Thus $\mathrm{D}^-$ is the unbounded component of $\cD_u^-$. The argument for $\Re u \le 0$ is similar.

It remains to show that $z_u$ is the unique point in $\C \setminus \supp \bm$ satisfying (i) and (ii) in \Cref{def:suitable}. Suppose for contradiction that we have two such points $z_1,z_2$. Assume that $\Re \cS_u(z_1) \ge \Re \cS_u(z_2)$. Then $z_2$ is contained in the boundary of a bounded component of $\cD_{u,z_u}^+$ and the unbounded component of $\cD_{u,z_2}^+$ denoted $\mathrm{D}_1,\mathrm{D}_2$ as above. By harmonicity and the maximum principle, $\mathrm{cl}(\mathrm{D}_1)$ intersects $\supp \bm$. Then $\mathrm{cl}(\mathrm{D}_1 \cup \mathrm{D}_2)$ is a connected, unbounded subset of $\C$ intersecting $\supp \bm$. Thus any closed curve around $\supp \bm$ must intersect $\mathrm{cl}(\mathrm{D}_1 \cup \mathrm{D}_2) \subset \cD_{u,z_2}^+$. However, this is incompatible with the existence of closed curve $\gamma$ around $\supp \bm$ satisfying $\gamma \setminus \{z_2\} \subset \cD_{u,z_2}^-$ by the fact that $\Re \cS_u(z_1) \ge \Re \cS_u(z_2)$ and $z_1 \ne z_2$ --- a contradiction.
\end{proof}

\subsection{Proof of Theorem \ref{bessel_asymp}} \label{ssec:bessel_asymp}

In \cite{GS}, Gorin-Sun study similar asymptotics arising from normalized multivariate Bessel functions. Due to common features in the asymptotics, we adapt the organization in \cite[\S 3.4]{GS} to our setting.

We first prove a special case of \Cref{bessel_asymp} where we take $k = 1$ and impose separation between our variables.

\begin{proposition} \label{single_bessel}
Fix $\xi > 0$ arbitrarily small. Under \Cref{assum:m_conv} with $\mathrm{m}_N := \tfrac{1}{N} \sum_{j=1}^N \delta_{\ell_j}$, if $\fp \in \C \setminus \supp \bm$ and $\cK \subset \fO_{\bm,\fp}$ is compact, then
\begin{align}
\begin{split}
& \frac{\cB_{\vec{\ell},\fp}(Nu,0^N/Nv)}{\cB_{\vec{\ell}}(0^N)} = \frac{u - v}{G_{\mathrm{m}_N}^{[-1]}(u) - G_{\mathrm{m}_N}^{[-1]}(v)} \frac{1}{\sqrt{G_{\mathrm{m}_N}'(G_{\mathrm{m}_N}^{[-1]}(u)) G_{\mathrm{m}_N}'(G_{\mathrm{m}_N}^{[-1]}(v))}} \\
& \quad \quad \times \exp\Bigg( N\Big(\cS_{\mathrm{m}_N,u}(G_{\mathrm{m}_N}^{[-1]}(u)) - \cS_{\mathrm{m}_N,v}(G_{\mathrm{m}_N}^{[-1]}(v))\Big) \Bigg)(1 + o(1))
\end{split}
\end{align}
uniformly over $(\cK \times \cK) \cap \{|u - v| \ge \xi \}$. The branch is chosen so that $G_{\mathrm{m}_N}'(G_{\mathrm{m}_N}^{[-1]}(u))^{-1/2}$ points locally from $G_{\mathrm{m}_N}^{[-1]}(u)$ in the direction of the unbounded component of $\cD_{\mathrm{m}_N,u}^+$.
\end{proposition}

\begin{proof}
We prove this proposition by steepest descent.

Let $K := K_N := (\cK \times \cK) \cap \{|u - v| \ge \xi\}$. Throughout this proof, we use $C$ and $C'$ to denote positive constants which are independent of $N$ and $(u,v) \in K$; it may depend on $\xi$ and vary from line to line. Let $S_u := \cS_{\mathrm{m}_N,u}$, $\cS_u := \cS_{\bm,u}$, and $\cD_u := \cD_{\bm,u}$.

By \Cref{cor:ssym_bessel},
\begin{align*}
\frac{\cB_{\vec{\ell},\fp}(Nu,0^N/Nv)}{\cB_{\vec{\ell}}(0^N)} &= N(u - v) \left( \frac{e^{N(u - v)\fp}}{N(u - v)} \left( \frac{v}{u} \right)^N +  \int_\fp^\infty \! dw \oint \! \frac{dz}{2\pi\bi} \cdot \frac{1}{w - z} \exp\Bigl( N (S_u(z) - S_v(w)) \Bigr) \right)
\end{align*}
where the $z$-contour contains $\ell_1,\ldots,\ell_N$ and the $w$-contour is an infinite ray from $\fp$ to $\infty$ which is disjoint from the $z$-contour and such that
\[ \Re(w \cdot v) > 0 \]
for $|w|$ large.

We divide the proof of \Cref{single_bessel} into three steps. The first step is to deform the $z$- and $w$- contours to steepest descent contours whereas the second and third steps carry out the steepest descent analysis. More specifically, after deforming the contours in the first step, we split $\cB_{\vec{\ell},\fp}(Nu,0^N/Nv)/\cB_{\vec{\ell}}(0^N)$ into two main parts. We demonstrate that one of these parts gives the leading order term in the second step and demonstrate that the other vanishes relative to the leading order in the third step.

\begin{figure}[ht]
    \centering
\begin{subfigure}{.5\textwidth}
  \centering
  \includegraphics[width=0.9\linewidth]{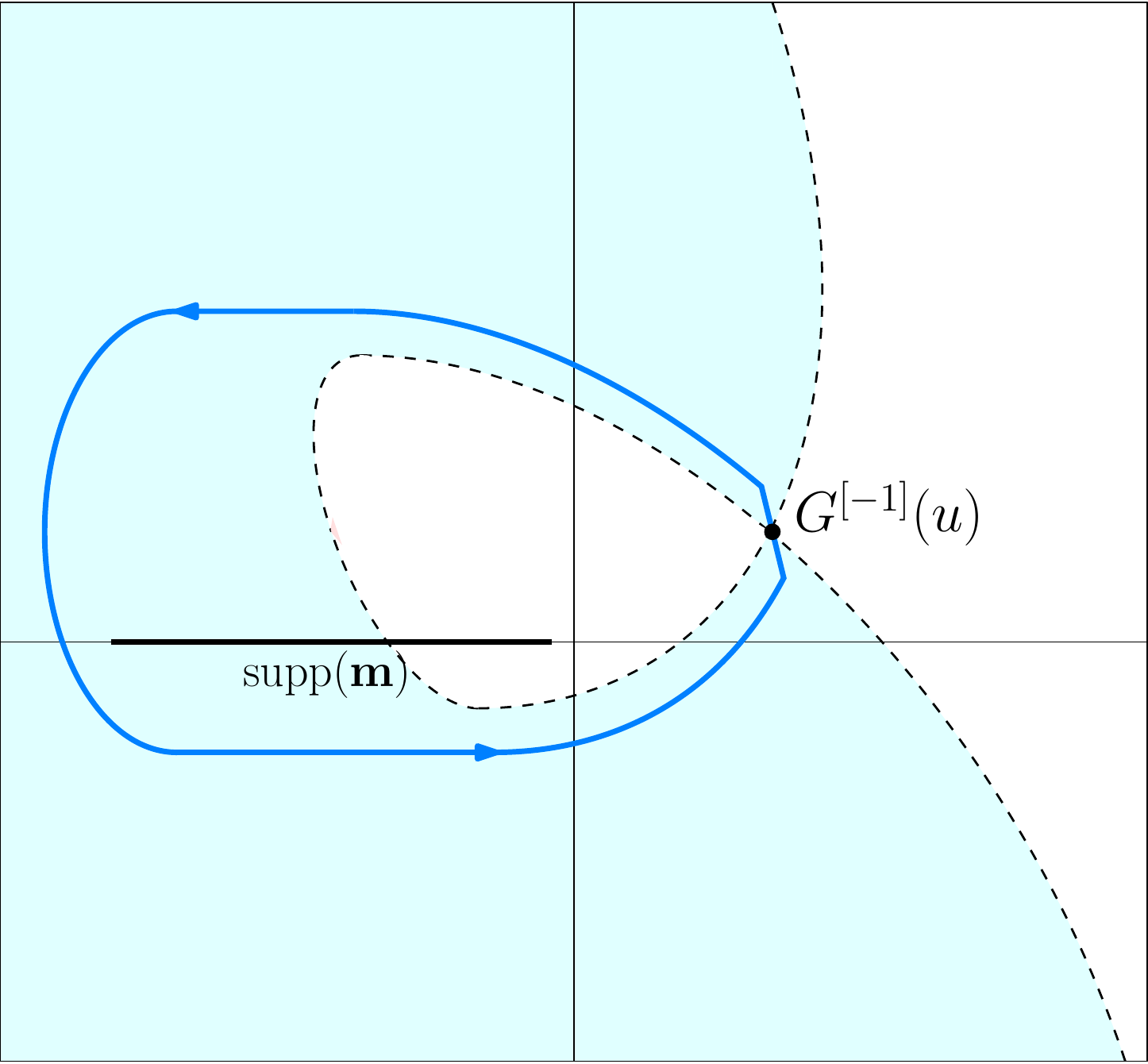}
\end{subfigure}%
\begin{subfigure}{.5\textwidth}
  \centering
  \includegraphics[width=0.9\linewidth]{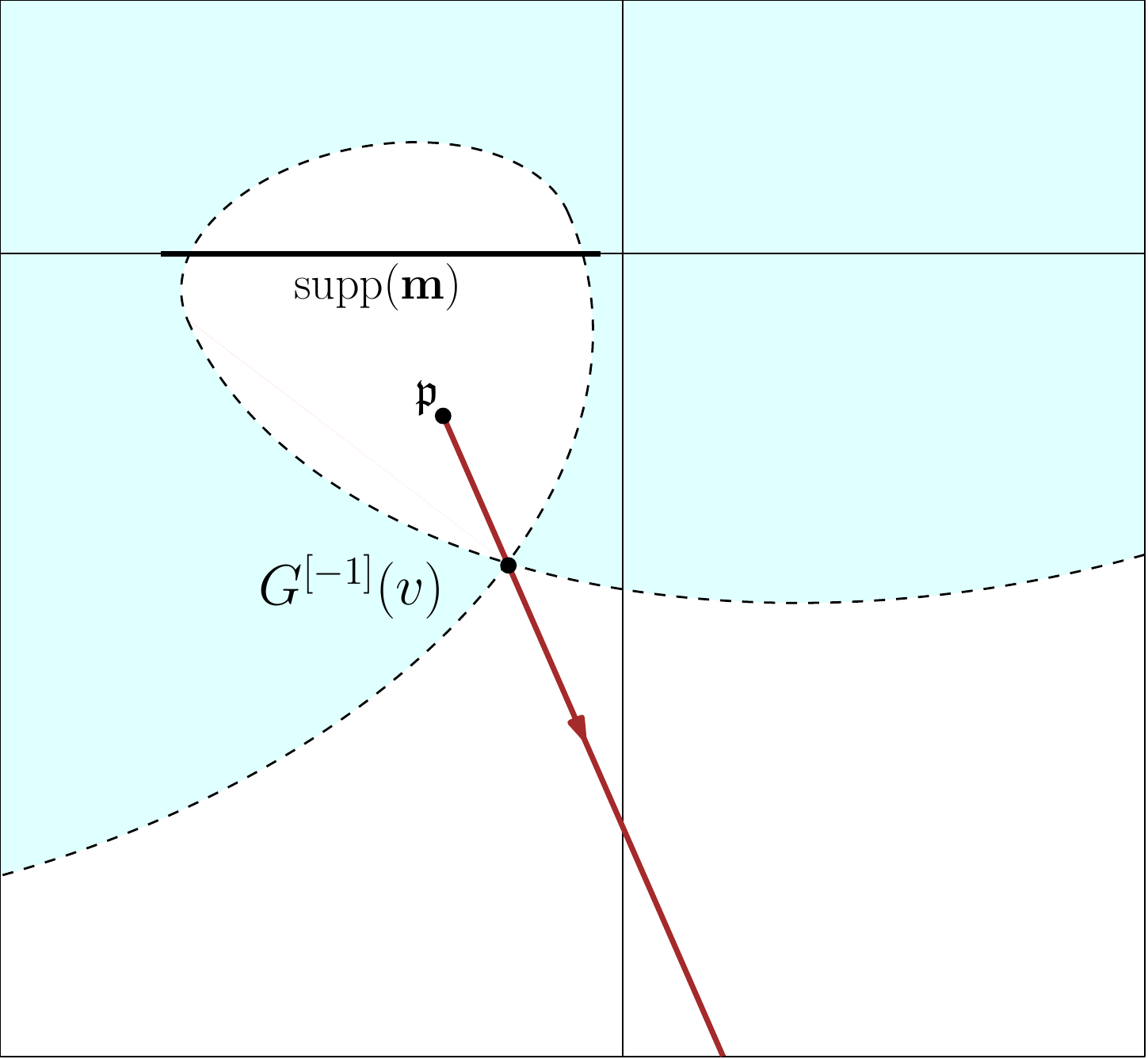}
\end{subfigure}
\caption{Steepest descent contours may be viewed as perturbations of the schematically represented contours above. Shaded regions correspond to $\cD_u^-$ and $\cD_v^-$.}
\label{fig:contour}
\end{figure}

\textbf{Step 1.} By the weak convergence $\mathrm{m}_N \to \bm$, we have that $G_{\mathrm{m}_N}(z) \to G_\bm(z)$, $S_u(z) \to \cS_u(z)$ uniformly over $u \in \cK$ and $z$ in compact subsets of $\C \setminus \supp \bm$. Then
\[ y_u := G_{\mathrm{m}_N}^{[-1]}(u) \quad \mbox{converges to} \quad G_\bm^{[-1]}(u) \]
uniformly for $u \in \cK$. Furthermore, since $\cK$ is compact, $\cK \subset \fO_{\mathrm{m}_N,\fp}$ for $N$ sufficiently large. By \Cref{def:suitable} and \Cref{thm:level_lines}, we may choose our steepest descent contours $\gamma_z := \gamma_{z,N}$ and $\gamma_w := \gamma_{w,N}$ such that
\begin{enumerate}
    \item $\gamma_z$ passes through $y_u$;
    \item in a constant order neighborhood of $y_u$, $\gamma_z$ is a line segment in the direction $[-S_u''(y_u)]^{-1/2}$;
    \item $\gamma_z$ is positively oriented around $\{\ell_i\}_{i=1}^N \cup \{\fp\}$;
\end{enumerate}
where the branch of $[-S_u''(y_u)]^{-1/2}$ is taken to be consistent with the orientation of $\gamma_z$, and
\begin{enumerate}
    \item $\gamma_w$ starts from $\fp$ and passes through $y_v$;
    \item in a constant order neighborhood of $y_v$, $\gamma_w$ is a line segment in the direction $[S_v''(y_v)]^{-1/2}$;
    \item there exists some large $r > 0$, uniform over $v \in \cK$ and $N$ sufficiently large, such that $\gamma_w \setminus B_r(0)$ is a ray in the direction $w_0$ given by
    \begin{align} \label{eq:gammaw_ray}
    \Re(w_0 \cdot v) > c
    \end{align}
    for some $c > 0$;
\end{enumerate}
where the branch of $[S_u''(y_v)]^{-1/2}$ is taken to point from $y_v$ locally in the direction of the unbounded connected component of $\{\Re S_v(w) > \Re S_v(y_v)\}$. Furthermore,
\begin{enumerate}
    \item the contours $\gamma_z,\gamma_w$ can be chosen to vary continuously in $(u,v) \in K$;
    \item since $|u - v| \ge \xi$, $\gamma_z,\gamma_w$ may be chosen so that
    \begin{align} \label{dist_int}
    \min\big(\dist(\gamma_w, y_u), \dist(\gamma_z, y_v ) \big) \ge C > 0
    \end{align}
    for $(u,v) \in K$;
    \item the curve $\gamma_{z,w}$, defined as the part of $\gamma_w$ enclosed by $\gamma_z$, is a union of a uniformly bounded (for $(u,v) \in K$ and $N$ sufficiently large) number of connected, bounded curves.
\end{enumerate}

These contours may be viewed as perturbations of corresponding contours for $\bm$, see \Cref{fig:contour}.

By deforming the $z$- and $w$-contours to $\gamma_z$,$\gamma_w$ respectively, we obtain
\begin{align} \label{B=I+R}
\frac{\cB_{\vec{\ell},\fp}(Nu,0^N/Nv)}{\cB_{\vec{\ell}}(0^N)} = N(u - v) \Bigl( I(u,v) + L(u,v) \Bigr)
\end{align}
where
\begin{align*}
I(u,v) &:= \int_{\gamma_w} \! dw \oint_{\gamma_z} \! \frac{dz}{2\pi\bi} \cdot \frac{1}{w - z} \exp\Bigl( N (S_u(z) - S_v(w)) \Bigr), \\
L(u,v) &:= \int_{\gamma_{z,w}} \exp\Bigl( N (S_u(w) - S_v(w)) \Bigr) \, dw + \left( \frac{v}{u} \right)^N \frac{e^{N(u - v)\fp}}{N(u - v)}.
\end{align*}
The term $L(u,v)$ is continuous in $u,v$, and $\gamma_{z,w}$ is nonempty because $u,v \in \fO_{\bm,\fp}$.

Choose $\delta := \delta(N) > 0$ small so that
\begin{align} \label{eq:delta}
N^{-1/2 + \e} \le \delta \le N^{-1/3 - \e}
\end{align}
for some small $\e > 0$. We decompose the contours $\gamma_z,\gamma_w$ as
\[ \gamma_z = \gamma_z^1 \cup \gamma_z^2, \quad \gamma_w = \gamma_w^1 \cup \gamma_w^2 \]
where
\begin{gather*}
\gamma_z^1 = B_\delta(y_u) \cap \gamma_z, \quad \gamma_z^2 = \gamma_z \setminus \gamma_z^1 \\
\gamma_w^1 = B_\delta(y_v) \cap \gamma_w, \quad \gamma_w^2 = \gamma_w \setminus \gamma_w^1.
\end{gather*}
By the choice of contours, for sufficiently large $N$
\begin{align} \label{eq:gammaz_max}
\Re S_u(z) < \Re S_u(y_u) & \quad \quad \mbox{for $z \in \gamma_z \setminus \{y_u\}$}, \\ \label{eq:gammaw_min}
\Re S_v(w) > \Re S_v(y_v) & \quad \quad \mbox{for $w \in \gamma_w \setminus \{y_v\}$}
\end{align}
and
\begin{align} \label{eq:gammaz2_bd}
\Re[ S_u(y_u) - S_u(z) ] > \frac{1}{4} |S_u''(y_u)| \delta^2 & \quad \mbox{for $z \in \gamma_z^2$}, \\ \label{eq:gammaw2_bd}
\Re[ S_v(w) - S_v(y_v) ] > \frac{1}{4} |S_v''(y_v)| \delta^2 & \quad \mbox{for $w \in \gamma_w^2$}.
\end{align}

\textbf{Step 2.} We establish that
\begin{align} \label{I_asymp}
I(u,v) = \frac{1}{y_v - y_u} \frac{-\bi}{N \sqrt{-S_u''(y_u)}\sqrt{S_v''(y_v)}} e^{N(S_u(y_u)-S_v(y_v))} (1 + o(1)).
\end{align}
We have
\[ I(u,v) = I_1(u,v) + I_2(u,v) + I_3(u,v) + I_4(u,v) \]
where
\begin{align*}
I_1(u,v) &:= \frac{1}{y_v - y_u} \left[ \int_{\gamma_z^1} \exp( N S_u(z) ) \, \frac{dz}{2\pi\bi} \right] \left[ \int_{\gamma_w^1} \exp( -N S_v(w)) \, dw \right], \\
I_2(u,v) &:= \int_{\gamma_z^2} \frac{dz}{2\pi\bi} \int_{\gamma_w} dw \cdot \frac{1}{w - z} \exp\Bigl( N (S_u(z) - S_v(w)) \Bigr), \\
I_3(u,v) &:= \int_{\gamma_z^1} \frac{dz}{2\pi\bi} \int_{\gamma_w^2} dw \cdot \frac{1}{w - z} \exp\Bigl( N (S_u(z) - S_v(w)) \Bigr), \\
I_4(u,v) &:= \int_{\gamma_z^1} \frac{dz}{2\pi\bi} \int_{\gamma_w^1} dw \left[ \frac{1}{w - z} - \frac{1}{y_v - y_u} \right] \exp\Bigl( N (S_u(z) - S_v(w)) \Bigr).
\end{align*}
Note that $y_u \ne y_v$ for $(u,v) \in K$, since 
\[ G_{\mathrm{m}_N}(y_u) = u \ne v = G_{\mathrm{m}_N}(y_v). \]

\underline{$I_1(u,v)$.} We show that
\begin{align} \label{I1_asymp}
I_1(u,v) = \frac{1}{y_v - y_u} \frac{-\bi}{N \sqrt{-S_u''(y_u)}\sqrt{S_v''(y_v)}} e^{N(S_u(y_u)-S_v(y_v))} (1 + o(1)).
\end{align}
For $|z - y_u| < \delta$, we have
\begin{align} \label{S_taylor}
S_u(z) = S_u(y_u) + \frac{1}{2} S_u''(y_u) (z - y_u)^2 + \cE_u(z)
\end{align}
where
\begin{align} \label{S_error}
|\cE_u(z)| < C|z - y_u|^3.
\end{align}

Therefore,
\[ \int_{\gamma_z^1} \exp( N S_u(z) ) \, \frac{dz}{2\pi\bi} = e^{N S_u(y_u)} \int_{\sL_u} \exp\Bigl( \frac{N}{2} S_u''(y_u) (z - y_u)^2 \Bigr) \frac{dz}{2\pi\bi} (1 + o(1)) \]
where $\cL_u$ is the line passing through $y_u$ in the direction $[-S_u''(y_u)]^{-1/2}$. The latter can be evaluated to obtain
\[ \int_{\gamma_z^1} \exp( N S_u(z) ) \, \frac{dz}{2\pi\bi} = e^{N S_u(y_u)} [-S_u''(y_u)]^{-1/2} \int_{-\infty}^\infty e^{-\frac{N}{2}t^2} \frac{dt}{2\pi\bi} (1 + o(1)) = \frac{-\bi}{\sqrt{-2\pi N S_u''(y_u)}} e^{N S_u(y_u)} (1 + o(1)) \]
Likewise,
\[ \int_{\gamma_w^1} \exp( -N S_v(w)) \, dw = e^{-NS_v(y_v)} \int_{\sL_v} \exp\Bigl( -\frac{N}{2} S_v''(y_v) (w - y_v)^2 \Bigr) \, dw (1 + o(1))\]
where $\sL_v$ is the line passing through $y_v$ in the direction $[S_v''(y_v)]^{-1/2}$. Evaluating the latter, we obtain
\[ \int_{\gamma_w^1} \exp( -N S_v(w)) \, dw = e^{-N S_v(y_v)} [S_v''(y_v)]
^{-1/2} \int_{-\infty}^\infty \! e^{-\frac{N}{2} t^2} \, dt (1 + o(1)) = \frac{\sqrt{2\pi}}{\sqrt{N S_v''(y_v)}} e^{-NS_v(y_v)} (1 + o(1)). \]
Combining these asymptotics, \eqref{I1_asymp} follows.

\underline{$I_2(u,v)$ and $I_3(u,v)$.} By compactness of $K$, we have that $|S_u''(y_u)|$ and $|S_v''(y_v)|$ are bounded away from $0$ for $(u,v) \in K$. By \eqref{eq:delta} and \eqref{eq:gammaz2_bd}, for $z\in \gamma_z^2$ we have
\begin{align} \label{I23_gammaz2}
\left| \exp\Bigl( N ( S_u(z) - S_u(y_u) ) \Bigr) \right| < \exp\left( - \frac{1}{4} S_u''(y_u) \delta^2 N \right) = O(e^{-N^\e})
\end{align}
Similarly, by \eqref{eq:delta} and \eqref{eq:gammaw2_bd}, for $w \in \gamma_w^2$ we have
\begin{align} \label{I23_gammaw2}
\left| \exp\Bigl( N ( S_v(y_v) - S_v(w) ) \Bigr) \right| < \exp\left( - \frac{1}{4} S_v''(y_v) \delta^2 N \right) = O(e^{-N^\e}).
\end{align}
By \eqref{eq:gammaw_ray}, for $w \in \gamma_w \setminus B_r(0)$, we have
\begin{align*}
\begin{split}
\left| \exp\Bigl( - N S_v(w) \Bigr) \right| &= \left| \exp\Bigl( -N(wv - \int \log(w - x) d\mathrm{m}_N(x) - \log u) \Bigr) \right|\\
&< \exp\Bigl( -N c'|w| \Bigr)
\end{split}
\end{align*}
for some $c' > 0$ and $r$ large enough, independent of $N, (u,v) \in K$. This implies that
\begin{align}  \label{I23_gammaw_unbdd}
\int_{\gamma_w \setminus B_r(0)} \left| \exp\Bigl( N(S_v(y_v) -  S_v(w)) \Bigr) \right| \, d|w| \le O(e^{-CN})
\end{align}

By \eqref{eq:gammaw_min} and \eqref{I23_gammaz2}, we have
\[ \left| \int_{\gamma_z^2} \frac{dz}{2\pi\bi} \int_{\gamma_w \cap B_r(0)} dw \cdot \frac{1}{w - z} \exp\Bigl( N ( S_u(z) - S_v(w)) \Bigr) \right| \le O(e^{-N^\e}) \exp\Bigl( N( S_u(y_u) - S_v(y_v)) \Bigr). \]
We note that the singularity $1/(w - z)$ at any intersection point between $\gamma_z^2$ and $\gamma_w \cap B_r(0)$ is logarithmic and therefore integrable. By \eqref{I23_gammaz2}, \eqref{I23_gammaw_unbdd}, and the bounded arc length of $\gamma_z$, we have
\[ \left| \int_{\gamma_z^2} \frac{dz}{2\pi\bi} \int_{\gamma_w \setminus B_r(0)} dw \cdot \frac{1}{w - z} \exp\Bigl( N ( S_u(z) - S_v(w)) \Bigr) \right| \le O(e^{-CN}) \exp\Bigl( N(S_u(y_u) - S_v(y_v)) \Bigr). \]
These two bounds imply that
\[ |I_2(u,v)| \le O(e^{-N^\e}) \exp\Bigl( N(S_u(y_u) - S_v(y_v)) \Bigr) = |I_1(u,v)| \cdot o(1). \]

Similarly, to bound $I_3(u,v)$, we use
\begin{align*}
\left| \int_{\gamma_z^1} \frac{dz}{2\pi\bi} \int_{\gamma_w^2 \cap B_r(0)} dw \cdot \frac{1}{w - z} \exp\Bigl( N ( S_u(z) - S_v(w)) \Bigr) \right| &\le O(e^{-N^\e}) \exp\Bigl( N( S_u(y_u) - S_v(y_v)) \Bigr) \\
\left| \int_{\gamma_z^1} \frac{dz}{2\pi\bi} \int_{\gamma_w^2 \setminus B_r(0)} dw \cdot \frac{1}{w - z} \exp\Bigl( N ( S_u(z) - S_v(w)) \Bigr) \right| &\le O(e^{-CN}) \exp\Bigl( N(S_u(y_u) - S_v(y_v)) \Bigr).
\end{align*}
The first inequality follows from \eqref{eq:gammaz_max} and \eqref{I23_gammaw2}. The second inequality follows from \eqref{eq:gammaz_max} and \eqref{I23_gammaw_unbdd}. Thus
\[ |I_3(u,v)| \le O(e^{-N^\e}) \exp\Bigl( N (S_u(y_u) - S_v(y_v)) \Bigr) = |I_1(u,v)| \cdot o(1). \]

\underline{$I_4(u,v)$.} By \eqref{S_taylor} and
\[ \frac{1}{w - z} - \frac{1}{y_v - y_u} = \frac{1}{y_v - y_u} \left( \frac{(z - y_u) - (w - y_v)}{w - z} \right), \]
we have
\begin{align*}
I_4(u,v) &= \frac{e^{N(S_u(y_u) - S_v(y_v))}}{y_v - y_u} \int_{\gamma_z^1} \frac{dz}{2\pi\bi} \int_{\gamma_w^1} dw \\
& \quad \times \frac{(z - y_u) - (w - y_v)}{w - z} \exp\left(\frac{N}{2}\Big(S_u''(y_u)(z - y_u)^2 - S_v''(y_v)(w - y_v)^2 + \cE_u(z) - \cE_v(w)\Big) \right).
\end{align*}
By compactness of $K$, $|S_u''(y_u)|$ and $|S_v''(y_v)|$ are bounded away from $0$ for $(u,v) \in K$. By \eqref{dist_int} and \eqref{eq:delta}, we have
\[ |w - z| \ge C, \quad z \in \gamma_z^1, \quad w \in \gamma_w^1. \]
Furthermore, we have
\[ |z - y_u| \le \delta, \quad |w - y_v| \le \delta, \quad z \in \gamma_z^1, \quad w \in \gamma_w^1. \]
These considerations imply
\begin{align*}
|I_4(u,v)| & \le C\delta \left| \frac{e^{N(S_u(y_u) - S_v(y_v))}}{y_v - y_u} \right| \int_{-\delta}^\delta ds \int_{-\delta}^\delta dt \, e^{-NC'(s^2 + t^2)} \\
& \le C \delta \left| \frac{e^{N(S_u(y_u) - S_v(y_v))}}{N(y_v - y_u)} \right| = |I_1(u,v)|\cdot o(1)
\end{align*}
where we use \eqref{S_error} in the first inequality and the equality uses \eqref{eq:delta}.

We conclude \eqref{I_asymp} from the analyses for $I_1(u,v)$, $I_2(u,v)$, $I_3(u,v)$, $I_4(u,v)$.

\textbf{Step 3.} We now analyze the term $L(u,v)$. Recall that the curve $\gamma_{z,w}$ is a nonempty union of connected curves $\gamma_{z,w}^0, \gamma_{z,w}^1,\ldots,\gamma_{z,w}^n$ where $\gamma_{z,w}^0$ is a curve from $a_0 := \fp$ to $b_0 \in \gamma_w \cap \gamma_z$ and $\gamma_{z,w}^i$ has endpoints $a_i,b_i \in \gamma_w \cap \gamma_z$ for $1 \le i \le n$, with $n > 0$. Moreover, $a_i,b_i$ vary continuously in $u,v$ for $0 \le i \le n$. Since
\[ S_u(z) - S_v(z) = (u - v)z - \log(u/v), \]
we have
\begin{align*}
L(u,v) &= \left( \frac{v}{u} \right)^N \frac{e^{N(u - v)\fp}}{N(u - v)} + \sum_{i=0}^n \int_{\gamma_{z,w}^i} \exp\Bigl( N(S_u(w) - S_v(w)) \Bigr) \, dw \\
&= \frac{e^{N(S_u(\fp) - S_v(\fp))}}{N(u - v)} + \sum_{i=0}^n \frac{1}{N(u - v)} \left( e^{N(S_u(b_i) - S_v(b_i))} - e^{N(S_u(a_i) - S_v(a_i))} \right) \\
&= \sum_{i=0}^n \frac{1}{N(u - v)} e^{N(S_u(b_i) - S_v(b_i))} - \sum_{i=1}^n \frac{1}{N(u - v)} e^{N(S_u(a_i) - S_v(a_i))} 
\end{align*}
where the last equality uses $a_0 = \fp$. In particular, we end up with a finite sum of terms of the form
\[ \pm\frac{1}{N(u - v)} e^{N(S_u(a) - S_v(a))}, \]
where $a \in \gamma_w \cap \gamma_z$.

By \eqref{dist_int} and \eqref{eq:delta}, we have $a \in \gamma_z^2 \cap \gamma_w^2$. Then \eqref{eq:gammaz2_bd} and \eqref{eq:gammaw2_bd} imply
\[ \Re[ S_u(a) - S_v(a)] < \Re[ S_u(y_u) - S_v(y_v)] - \frac{1}{4} |S_u''(y_u)| \delta^2. \]
It follows that
\[ \left| \frac{1}{N(u - v)} e^{N(S_u(a) - S_v(a))} \right| < C e^{-\frac{1}{4} |S_u''(y_u)| \delta^2 N} \frac{1}{N(u - v)} |e^{N(S_u(y_u) - S_v(y_v))}| < |I(u,v)| \cdot o(1) \]
where we use the fact that $|S_u''(y_u)|$ is bounded away from $0$ for all $u$ such that $(u,v) \in K$, and the fact that
\[ C^{-1} |v - u| < |y_v - y_u| < C|v - u|, \quad \quad (u,v) \in K. \]
The latter inequality follows from observing that
\[ y_v - y_u = G_{\mathrm{m}_N}^{[-1]}(v) - G_{\mathrm{m}_N}^{[-1]}(u) = \frac{1}{G_{\mathrm{m}_N}'(y_u)}(v - u) + O(|v-u|^2)= -\frac{1}{S_u''(y_u)} (v - u) + O(|v-u|^2) \]
for $|v - u|$ small. Thus
\[ |L(u,v)| = |I(u,v)| \cdot o(1). \]
By \eqref{B=I+R}, \eqref{I_asymp} and our bound for $|L(u,v)|$, we have
\[ \frac{\cB_{\vec{\ell},\fp}(Nu,0^N/Nv)}{\cB_{\vec{\ell}}(0^N)} = \frac{u - v}{y_u - y_v} \frac{\bi}{\sqrt{-S_u''(y_u)}\sqrt{S_v''(y_v)}} e^{N(S_u(y_u)-S_v(y_v))} (1 + o(1)). \]
To complete the proof of the theorem, we show that
\[ S_u(y_u) = H_{\mathrm{m}_N}(u), \]
which is defined on $\fO$, and
\[ \frac{\bi}{\sqrt{-S_u''(y_u)}} = \frac{1}{\sqrt{S_u''(y_u)}} \]
where the latter branch is chosen so that $[S_u''(y_u)]^{-1/2}$ is in the direction of the unbounded component of $\{ \Re S_u(z) > \Re S_u(y_u) \}$. The first equality is obvious by definition of $H_{\mathrm{m}_N}$. For the second equality, recall that $\sqrt{-S_u''(y_u)}$ is tangent to $\gamma_z$ at $y_u$ and $\gamma_z$ is positively oriented around $\supp \mathrm{m}_N$. Thus clockwise rotation by $\pi/2$ points in the direction of the unbounded component of $\{ \Re S_u(z) > \Re S_u(y_u) \}$, which is exactly multiplication by $-\bi$. This completes the proof of \Cref{single_bessel}.
\end{proof}

We now prove \Cref{bessel_asymp}.

\begin{proof}[Proof of Theorem \ref{bessel_asymp}]
Since $\fO_{\bm,\fp}$ is open, we can choose $\e > 0$ so that $\cK_\e := \{z \in \C: \mathrm{dist}(z,\cK) \le \e \} \subset \fO_{\bm,\fp}$. We may assume that the boundaries $\cK_\delta \subset \cK_\e$ for $0 < \delta < \e$ may be parametrized by finitely many smooth curves of bounded arc length, for example we may replace $\cK$ with a cover formed by a union of finitely many balls in $\fO_{\bm,\fp}$. Let
\[ 0 < \delta_1 < \cdots < \delta_{2k} \le \e \]
and suppose $\gamma_1,\ldots,\gamma_{2k}$ are simple contours with images $\partial \cK_{\delta_1},\ldots,\partial \cK_{\delta_{2k}}$ and are positively oriented around $\cK_{\delta_1},\ldots,\cK_{\delta_{2k}}$ respectively. Then $\gamma_1,\ldots,\gamma_{2k}$ are disjoint and contained in $\fO_{\bm,\fp}$. We have
\begin{align*}
& \frac{\cB_{\vec{\ell},\fp}(N\vec{u},0^N/N\vec{v})}{\cB_{\vec{\ell}}(0^N)} = \frac{D(\vec{u};-\vec{v})}{\Delta(\vec{u}) \Delta(-\vec{v})} \det \left( \frac{1}{u_i - v_j} \frac{\cB_{\vec{\ell},\fp}(Nu_i/Nv_j)}{\cB_{\vec{\ell}}(0^N)} \right)_{i,j=1}^k \\
& \quad = \frac{D(\vec{u};-\vec{v})}{\Delta(\vec{u}) \Delta(-\vec{v})} \det \left( \frac{1 + o(1)}{G_{\mathrm{m}_N}^{[-1]}(u_i) - G_{\mathrm{m}_N}^{[-1]}(v_j)} \right) \prod_{i=1}^k \frac{\exp\Bigg( N\Big(\cS_{\mathrm{m}_N,u}(G_{\mathrm{m}_N}^{[-1]}(u_i)) - \cS_{\mathrm{m}_N,v}(G_{\mathrm{m}_N}^{[-1]}(v_i))\Big) \Bigg)}{\sqrt{G_{\mathrm{m}_N}'(G_{\mathrm{m}_N}^{[-1]}(u_i)) G_{\mathrm{m}_N}'(G_{\mathrm{m}_N}^{[-1]}(v_i))}} \\
& \quad = \frac{D(\vec{u};-\vec{v})}{\Delta(\vec{u}) \Delta(-\vec{v})} \frac{\Delta(G_{\mathrm{m}_N}^{[-1]}(\vec{u})) \Delta(-G_{\mathrm{m}_N}^{[-1]}(\vec{v}))}{D(G_{\mathrm{m}_N}^{[-1]}(\vec{u}); -G_{\mathrm{m}_N}^{[-1]}(\vec{v}))} \prod_{i=1}^k \frac{\exp\Bigg( N\Big(\cS_{\mathrm{m}_N,u}(G_{\mathrm{m}_N}^{[-1]}(u_i)) - \cS_{\mathrm{m}_N,v}(G_{\mathrm{m}_N}^{[-1]}(v_i))\Big) \Bigg)}{\sqrt{G_{\mathrm{m}_N}'(G_{\mathrm{m}_N}^{[-1]}(u_i)) G_{\mathrm{m}_N}'(G_{\mathrm{m}_N}^{[-1]}(v_i))}} (1 + o(1))
\end{align*}
uniformly over $(\vec{u},\vec{v}) \in \gamma_1 \times \cdots \times \gamma_k$, where the first equality is \Cref{thm:ssym_bessel}, the second follows from \Cref{single_bessel}, and the third from the Cauchy determinant formula. Setting
\[ \fB_N(\vec{u},\vec{v}) := \frac{D(\vec{u};-\vec{v})}{\Delta(\vec{u}) \Delta(-\vec{v})} \frac{\Delta(G_{\mathrm{m}_N}^{[-1]}(\vec{u})) \Delta(-G_{\mathrm{m}_N}^{[-1]}(\vec{v}))}{D(G_{\mathrm{m}_N}^{[-1]}(\vec{u}); -G_{\mathrm{m}_N}^{[-1]}(\vec{v}))} \prod_{i=1}^k \frac{\exp\Bigg( N\Big(\cS_{\mathrm{m}_N,u}(G_{\mathrm{m}_N}^{[-1]}(u_i)) - \cS_{\mathrm{m}_N,v}(G_{\mathrm{m}_N}^{[-1]}(v_i))\Big) \Bigg)}{\sqrt{G_{\mathrm{m}_N}'(G_{\mathrm{m}_N}^{[-1]}(u_i)) G_{\mathrm{m}_N}'(G_{\mathrm{m}_N}^{[-1]}(v_i))}}, \]
we have that both $\fB_N(\vec{u},\vec{v})$ and $\cB_{\vec{\ell},\fp}(N\vec{u},0^N/N\vec{v})/B_{\vec{\ell}}(0^N)$ are analytic in $\cK_\e$. Since the former function is nonzero on $\cK_\e$, so is the latter for $N$ large. Then for $(\vec{u},\vec{v}) \in \cK^{2k}$, we have
\begin{align*}
& \frac{\cB_{\vec{\ell},\fp}(N\vec{u},0^N/N\vec{v})}{\cB_{\vec{\ell}}(0^N)} \frac{1}{\fB_N(\vec{u},\vec{v})} = \frac{1}{(2\pi\bi)^{2k}} \oint_{\gamma_1} \frac{dz_1}{z_1 - u_1} \cdots \oint_{\gamma_k} \frac{dz_k}{z_k - u_k} \oint_{\gamma_{k+1}} \frac{dw_1}{w_1 - v_1} \cdots \oint_{\gamma_{2k}} \frac{dw_k}{w_k - v_k} \\
& \quad \quad \times \frac{\cB_{\vec{\ell},\fp}(Nz_1,\ldots,Nz_k,0^N/Nw_1,\ldots,Nw_k)}{\cB_{\vec{\ell}}(0^N)} \frac{1}{\fB_N(z_1,\ldots,z_k,w_1,\ldots,w_k)} = 1 + o(1)
\end{align*}
for $(\vec{u},\vec{v}) \in \cK^{2k}$ by Cauchy's integral formula. This completes our proof.
\end{proof}

\section{Simplification of Hypotheses} \label{sec:submaster}

The main results of this section (\Cref{thm:submaster} and \Cref{thm:submaster_q}) provide a set of sufficient conditions for Airy-appropriateness, at the expense of some generality. \Cref{thm:submaster} gives conditions in the random matrix setting and \Cref{thm:submaster_q} gives analogous conditions in the quantized setting. The advantage of these conditions is that checking Airy appropriateness is reduced to checking certain properties about the limiting measures. The key ingredients in the proofs are \Cref{thm:master} and analytic subordination for free additive convolutions, and its quantized analogues. Our main results (\Cref{thm:multisum,thm:two_sum,thm:multitensor}) are proved in \Cref{sec:applications,sec:applications_q} using the developments from this section.

We begin by defining the free additive convolution via analytic subordination. Assume throughout this section that $\bm,\bm^{(1)},\bm^{(2)},\ldots \in \cM$.

\begin{theorem}[{\cite[Theorem 4.1]{BB07}, \cite[Theorem 6]{Bel14}}] \label{thm:subord_add}
There exists a unique compactly supported measure $\bm^{(1)} \boxplus \bm^{(2)}$ and unique analytic functions $\omega_1, \omega_2: \C^+ \to \C^+$ such that
\begin{gather} \nonumber
\lim_{t \nearrow \infty} \frac{\omega_i(\bi t)}{\bi t} = 1 \\ \label{subord_imz}
\Im \omega_1(z), \Im \omega_2(z) \ge \Im z \\ \label{subord_cauchy}
G_{\bm^{(1)} \boxplus \bm^{(2)}}(z) = G_{\bm^{(1)}}(\omega_1(z)) = G_{\bm^{(2)}}(\omega_2(z)) \\ \label{subord_relation}
z = \omega_1(z) + \omega_2(z) - \frac{1}{G_{\bm^{(1)} \boxplus \bm^{(2)}}(z)}.
\end{gather}
for all $z \in \C^+$. Furthermore, if neither $\bm^{(1)}$ nor $\bm^{(2)}$ are a single atom, then $\omega_1,\omega_2$ extend continuously to $\C^+ \cup \R$ with values in $\C \cup \{\infty\}$ and the inequalities \eqref{subord_imz} are strict.
\end{theorem}

\begin{remark}
The statement that $\Im \omega_i(z) > \Im z$ for $i = 1,2$ if $\bm^{(1)},\bm^{(2)}$ are not single atoms is not mentioned in \cite{BB07,Bel14}. It can be seen from the following argument. If $\Im \omega_1(w) = \Im w$ at some $w \in \C^+$, then $\omega_1(z) - z$ is an analytic map from $\C^+$ to $\C^+\cup \R$ achieving a real value at $w$. Then $\omega_1(z) - z \equiv a$ for some $a \in \R$. Using \eqref{subord_cauchy}, \eqref{subord_relation}, we find that the Cauchy transform of $G_{\bm^{(1)}}(z)$ is that of a single atom.
\end{remark}

The measure $\bm^{(1)} \boxplus \bm^{(2)}$ is known as the \emph{free additive convolution of $\bm^{(1)}$ and $\bm^{(2)}$}, which first appeared in the work of \cite{V86} in terms of the $R$-transform, a generating function for certain quantities known as free cumulants. The existence of subordination functions was established in \cite{V93} under some mild restrictions, see also \cite{BB07} for an elementary complex analytic proof using Denjoy-Wolff fixed points. We take the alternative approach of defining the free additive convolution in terms of this subordination. The continuity of $\omega_1,\omega_2$ on $\C^+ \cup \R$ was established by \cite{Bel14}.

It can be checked that $\boxplus$ is a commutative and associative binary operation. Although we are in the setting where our measures are compactly supported, we note that $\boxplus$ can be extended to arbitrary probability measures on $\R$, see e.g. \cite{MS17} for details.

Fix $\bm \in \cM$, and let
\[ E_- := \inf \supp \bm, \quad E_+ := \sup \supp \bm \]
We write $E_\pm(\bm)$ when there is ambiguity in the measure $\bm$. Let
\[ G(E_-) := \lim_{x \nearrow E_-} G(x), \quad G(E_+) := \lim_{x \searrow E_+} G(x), \]
the existence of these limits in $[-\infty,\infty]$ are guaranteed by monotonicity of $G$. Observe that
\begin{align} \label{eq:G_monotone}
G~\mbox{is strictly decreasing}: \quad (E_+,\infty) \overset{1:1}{\longrightarrow} (0, G(E_+) ) \quad \mbox{and} \quad (-\infty,E_-) \overset{1:1}{\longrightarrow} (G(E_-),0).
\end{align}

\begin{theorem} \label{thm:submaster}
Under \Cref{assum:rmt_conv}, let $\vec{\ell} \in \cW_{\R}^N$ denote the eigenvalues of
\[ X_N^{(1)} + \cdots + X_N^{(n)}, \]
$\mathrm{m}^{(i)} := \tfrac{1}{N} \sum_{j=1}^{N} \delta_{\ell_j^{(i)}}$, and
\[ \bm := \bm^{(1)} \boxplus \cdots \boxplus \bm^{(n)}. \]
Suppose that
\[ \cA(u) = G_{\bm^{(1)}}^{[-1]}(u) + \cdots + G_{\bm^{(n)}}^{[-1]}(u) + \frac{1 - n}{u} \]
has a minimal real positive critical point $\fz \in (0,\min_{1 \le i \le n} G_{\bm^{(i)}}(E_+(\bm^{(i)}))) \subset \bigcap_{i=1}^n \fO_{\bm^{(i)}}$ which satisfies $\cA''(\fz) > 0$.
\begin{enumerate}[(i)]
    \item Then $\fz = G_\bm(E_+(\bm))$.
    \item If there exists $\fp_i \in \C \setminus \supp \bm$ such that
    \begin{align} \label{eq:curve_exists}
    \Gamma_\cA\left(\bigcap_{i=1}^n \fO_{\bm^{(i)},\fp_i}, \fz \right) \ne \emptyset
    \end{align}
    for $1 \le i \le n$, then the multivariate Bessel generating functions of $\vec{\ell}(N)$ are Airy edge appropriate where
    \[ A_N(u) = G_{\mathrm{m}^{(1)}}^{[-1]}(u) + \cdots + G_{\mathrm{m}^{(n)}}^{[-1]}(u) + \frac{1 - n}{u}. \]
    
    \item If
    \begin{align} \label{eq:set_inclusion}
    G_{\bm^{(i)}}(\{z \in \C: \Re z \ge G_{\bm^{(i)}}^{[-1]}(\fz) \}) \subset \fO_{\bm^{(i)},\fp_i}
    \end{align}
    for some $\fp_i \in \C \setminus \supp \bm$, for each $1 \le i \le n$, then \eqref{eq:curve_exists} is satisfied.
\end{enumerate}
\end{theorem}

\begin{remark}
The assumption \eqref{eq:set_inclusion} is purely technical, and we believe that the critical point and second derivative condition on $\cA$ should be sufficient for Airy appropriateness.
\end{remark}

In words, \Cref{thm:submaster} gives a set of conditions for which the hypothesis of \Cref{thm:master} is satisfied in the setting of sums of unitarily invariant random matrices with deterministic spectra. \Cref{thm:submaster} (ii) states that we must check that $\cA(u)$ has a minimal real positive critical point with $\cA''(\fz) > 0$ and that \eqref{eq:curve_exists} holds. \Cref{thm:submaster} (iii) allows us to check \eqref{eq:set_inclusion} in place of \eqref{eq:curve_exists}.

We present an analogous result for Schur generating functions in the quantized setting. Prior to stating this analogue, we state a crucial connection between the random matrix and quantized models is given by a variant of the Markov-Krein correspondence. Recall that $\cM_1$ is the subset of $\cM$ consisting of measures with density $\le 1$.

\begin{theorem}[{\cite[Theorem 1.10]{BuG15}}] \label{thm:MK_correspondence}
There exists a bijection $\cQ:\cM_1 \to \cM$ such that
\begin{align} \label{eq:MK_correspondence}
\exp\left( - G_{\bm}(z) \right) = 1 - G_{\cQ\bm}(z)
\end{align}
for $\bm \in \cM_1$.
\end{theorem}

The original Markov-Krein correspondence is directly related to the bijection above. It is obtained by conjugating $\cQ$ by the map on measures which reflects about the origin, see \cite[\S 1.5]{BuG15} and references therein for details.

The \emph{quantized free convolution of $\bm^{(1)},\bm^{(2)} \in \cM_1$} is defined by
\[ \bm^{(1)} \otimes \bm^{(2)} := \cQ^{-1}( \cQ \bm^{(1)} \boxplus \cQ \bm^{(2)} ). \]

\begin{theorem} \label{thm:submaster_q}
Under \Cref{assum:m_conv_q}, let $\lambda \in \cW_{\Z}^N$ be distributed as $\rho^V$ where
\[ V := V^{(1)}_N \otimes \cdots \otimes V^{(n)}_N, \]
$\mathrm{m}^{(i)} := \tfrac{1}{N} \sum_{j=1}^N \delta_{\frac{\lambda^{(i)}_j + N - j}{N}}$, and
\[ \bm := \bm^{(1)} \otimes \cdots \otimes \bm^{(n)}. \]
Suppose that
\[ \cA(u) = G_{\bm^{(1)}}^{[-1]}(u) + \cdots + G_{\bm^{(n)}}^{[-1]}(u) + \frac{1 - n}{1 - e^{-u}} \]
has a minimal real positive critical point $\fz \in (0,\min_{1 \le i \le n} G_{\bm^{(i)}}(E_+(\bm^{(i)}))) \subset \bigcap_{i=1}^n \fO_{\bm^{(i)}}$ which satisfies $\cA''(\fz) > 0$.
\begin{enumerate}[(i)]
    \item Then $\fz = G_\bm(E_+(\bm))$.
    \item If there exists $\fp_i \in \C \setminus \supp \bm$ such that
    \begin{align} \label{eq:curve_exists_q}
    \Gamma_\cA\left(\bigcap_{i=1}^n \fO_{\bm^{(i)},\fp_i}, \fz \right) \ne \emptyset
    \end{align}
    for $1 \le i \le n$, then the Schur generating functions of $\lambda(N)$ are Airy edge appropriate where
    \[ A_N(u) = G_{\mathrm{m}^{(1)}}^{[-1]}(u) + \cdots + G_{\mathrm{m}^{(n)}}^{[-1]}(u) + \frac{1 - n}{1 - e^{-u}}. \]
    
    \item If
    \begin{align} \label{eq:set_inclusion_q}
    G_{\bm^{(i)}}(\{z \in \C: \Re z \ge G_{\bm^{(i)}}^{[-1]}(\fz) \}) \subset \fO_{\bm^{(i)},\fp_i}
    \end{align}
    for some $\fp_i \in \C \setminus \supp \bm$, for each $1 \le i \le n$, then \eqref{eq:curve_exists_q} is satisfied.
\end{enumerate}
\end{theorem}

\subsection{Domain of the Inverse Cauchy Transform}

In preparation for the proofs of \Cref{thm:submaster,thm:submaster_q}, we obtain conditions under which $G^{[-1]}(u) = z_u$ is a left inverse of $G$. Note that we already know it is always a right inverse.

Throughout this subsection, fix $\bm \in \cM$ and let $G := G_\bm$. By the expansion
\[ G(z) = \frac{1}{z} + O(1) \]
we know that $G$ is an invertible meromorphic map taking a neighborhood of $\infty$ to a neighborhood of $0$. We have a candidate inverse $G^{[-1]}$ defined on $\fO$. Below, we show that $\fO$ contains a punctured neighborhood of $0$ and that $G^{[-1]}$ is the inverse of $G$ on this neighborhood. Given that $\fp$ is close enough to the support, we further glean that this neighborhood is contained in $\fO_\fp$.

\begin{proposition} \label{thm:domain}
Let $\fp \in \C \setminus \supp \bm$ and for $r > 0$, set $\mathrm{D}_r := \left\{z\in \C: \mathrm{dist}(z,\supp \bm)  \ge r \right\}$.
\begin{enumerate}[(i)]
    \item We have $(G(E_-),G(E_+)) \setminus \{0\}$ is the maximal punctured real neighborhood of $0$ contained in $\fO$, and $G^{[-1]} \circ G$ is the identity on $\R \setminus [E_-,E_+]$.
    \item If $r \ge 4(E_+ - E_-)$ and $\fp \in \C \setminus \supp \bm$ such that $\sup\{ |\fp - x|:x \in \supp \bm\} \le r/4$ , then $G(\mathrm{D}_r) \subset \fO_\fp$ and $G^{[-1]} \circ G$ is the identity on $\mathrm{D}_r$.
\end{enumerate}
\end{proposition}

\begin{proof}
We use the following claim. For $z_0 \in \C \setminus \supp \bm$, if $u = G(z_0)$, $G'(z_0) \ne 0$, and $\cD_{u,z_0}^-$ is connected, then $G(z_0) \in \fO$ and $G^{[-1]}(G(z_0)) = z_0$. To see this fact, observe $\cS_u$ is locally quadratic at the critical point $z_0$. Then the connectedness of $\cD_{u,z_0}^-$ implies that $z_0$ belongs to the boundary of $\cD_{u,z_0}^-$ and two distinct components $\mathrm{D}_1,\mathrm{D}_2$ of $\cD_{u,z_0}^+$. At least one of $\mathrm{D}_1,\mathrm{D}_2$, say $\mathrm{D}_1$, must be bounded because the unbounded component of $\cD_{u,z_0}^+$ is unique by \Cref{thm:component}. Since $\Re \cS_u$ is harmonic, the maximum principle implies $\mathrm{cl}(\mathrm{D}_1)$ intersects $\supp \bm$. We also cannot have $\mathrm{cl}(\mathrm{D}_2)$ intersect $\supp \bm$ without violating the connectedness of $\cD_{u,z_0}^- \subset \C \setminus \supp \bm$. Thus $\mathrm{D}_2$ is unbounded. By the connectedness of $\cD_{u,z_0}^-$, we can find a simple closed curve $\gamma$ through $z_0$ such that $\gamma \setminus \{z_0\} \subset \fD_{u,z_0}$ and $\gamma$ is positively oriented around $\supp \bm$, see \Cref{fig:level_lines}.

We prove (i). If $u \in (G(E_-),G(E_+)) \setminus \{0\}$, there is a unique $z_u \in \R \setminus [E_-,E_+]$ such that $u = G(z_u)$ by \eqref{eq:G_monotone}, and we have $G'(z_u) \ne 0$. Observe that $\cD_{u,z_u}^-$ has no bounded components because
\[ \Re \cS_u(z) = u \Re z - \int \log|z - x| d\bm(x) - \log|u| \]
is decreasing to $-\infty$ as $|\Im z|$ increases to $\infty$ for fixed $\Re z$. By \Cref{thm:component}, $\cD_{u,z_u}^-$ consists of a single unbounded component and is therefore connected. By our claim, $u \in \fO$ and $G^{[-1]}(u) = z_u$. The maximality of the punctured neighborhood follows from the fact that $G^{[-1]}$ is continuous on $\fO$ so that if $G(E_+) \in \fO$ or $G(E_-) \in \fO$, then $G^{[-1]}(G(E_+)) = E_+$ or $G^{[-1]}(G(E_-)) = E_-$, both of which are contradictions since $G^{[-1]}$ maps into $\C \setminus \supp \bm$ (see \Cref{def:suitable}).

We now prove (ii). Fix $z_0 \in \mathrm{D}_r$ and set $u = G(z_0)$. For any fixed $x_0 \in \supp \bm$, we have
\begin{align*}
\Re\left( \frac{(z_0 - x_0)^2}{(z_0 - x)^2} \right) = \Re\left( 1 + 2 \cdot \frac{x - x_0}{z_0 - x} + \frac{(x - x_0)^2}{(z_0 - x)^2} \right) > 1 - 2\cdot \frac{1}{4} - \frac{1}{16} > 0
\end{align*}
for every $x\in \supp \bm$, since $|x - x_0| < E_+ - E_-$ and $r \ge 4(E_+ - E_-)$. This shows that $1/(z_0 - x)^2$ is contained in the half-plane $\{w:\Re(w(z_0 - x_0)^2) > 0\}$ for each $x \in \supp \bm$ and therefore
\[ G'(z_0) = -\int_\R \frac{d\bm(x)}{(z_0 - x)^2} \in \{w: \Re(w(z - x_0)^2) < 0\}. \]
In particular, $G'(z_0) \ne 0$. Observe that if $z_1 \in \C \setminus \supp \bm$ such that
\begin{align*}
|z_1 - x| \le r/4 \quad \quad \mbox{for $x \in \supp \bm$},
\end{align*}
then
\[ \Re \cS_u(z_1) > \Re \cS_u(z_0). \]
Indeed, we have
\[ \left| \frac{z_1 - x}{z_0 - x} \right| < \frac{1}{4} \]
for $x \in \supp \bm$, so that
\begin{align*}
\Re \Big( \cS_u(z_1) - \cS_u(z_0) \Big) &= \Re \Big( G(z_0)(z_1 - z_0) \Big) - \int_\R \! \log\left| \frac{z_1 - x}{z_0 - x} \right| d\bm(x) \\
&= \Re \left( -1 + \int_\R \frac{z_1 - x}{z_0 - x} d\bm(x) \right) - \int_\R \log \left| \frac{z_1 - x}{z_0 - x} \right| d\bm(x) > \log 4 - \frac{5}{4} > 0.
\end{align*}
This means that $\mathrm{cl}(\cD_{u,z_0}^-)$ does not intersect $\supp \bm$. Thus, $\cD_{u,z_0}^-$ does not contain any bounded components, as such a component must intersect $\supp \bm$ by harmonicity of $\Re \cS_u$ and the maximum principle. By \Cref{thm:component}, $\cD_{u,z_0}^-$ consists of a single unbounded component and is therefore connected, see \Cref{fig:domain}. By our claim, $u \in \fO$ and $G^{[-1]}(u) = z_0$. Furthermore, we see that there is only one bounded component of $\cD_{u,z_u}^+$ and that this component contains $\{\fp \in \C \setminus \supp \bm: |\fp - x| \le E_+ - E_-\}$.
\end{proof}

\begin{figure}[ht]
    \centering
    \includegraphics[width=0.5\linewidth]{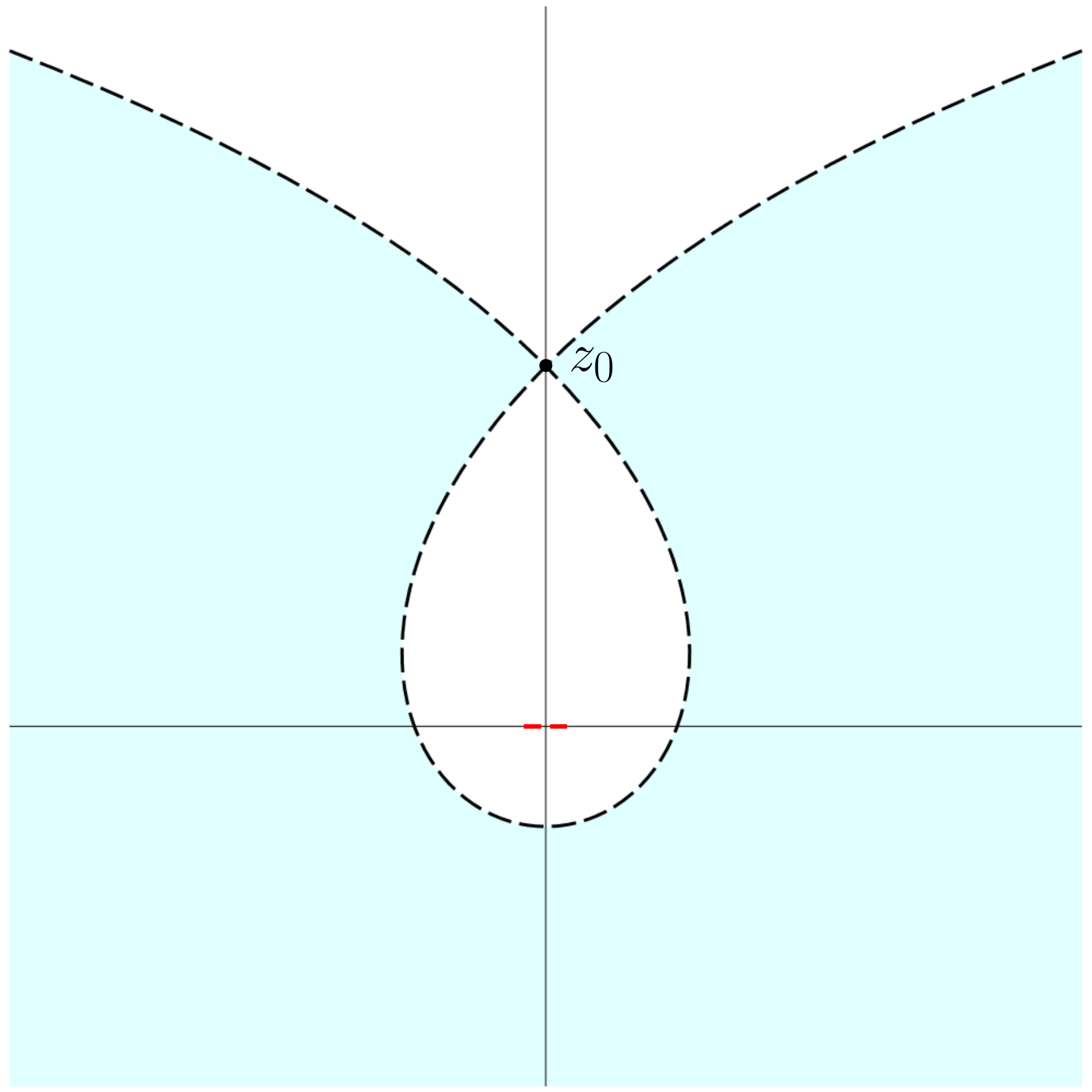}
    \caption{A representative diagram illustrating the level curve $\{\Re \cS_u(z) = \Re \cS_u(z_0)\}$ for $z_0$ far from the support. The shaded region is contained in $\cD_{u,z_0}^-$ and the red segments indicate the support.}
    \label{fig:domain}
\end{figure}

\subsection{Free Additive Convolution}

For this subsection, fix $\bm^{(1)},\ldots,\bm^{(n)} \in \cM$, let $\bm := \bm^{(1)} \boxplus \cdots \boxplus \bm^{(n)}$, and denote by $G$ the Cauchy transform of $\bm$.

By iterating \Cref{thm:subord_add}, we obtain

\begin{proposition} \label{thm:subord}
There exist analytic functions $\omega_i: \C^+ \to \C^+$ for $1 \le i \le n$ such that
\begin{gather} \label{omega:infty}
\lim_{t \nearrow \infty} \frac{\omega_i(\bi t)}{\bi t} = 1; \\ \label{omega:im_bd}
\Im \omega_i(z) \ge \Im z \\ \label{omega:subord}
G(z) = G_{\bm^{(i)}}(\omega_i(z)) \\ \label{omega:relation}
z = \omega_1(z) + \cdots + \omega_n(z) + \frac{1 - n}{G(z)}
\end{gather}
for all $z \in \C^+$, where the inequality \eqref{omega:im_bd} is strict if none of the measures $\bm^{(1)},\ldots,\bm^{(n)}$ are single atoms. Furthermore, if $\bm^{(i)}$ is not a single atom for each $1 \le i \le n$, then $\omega_1,\omega_2$ extend continuously to $\C^+ \cup \R$ with values in $\C \cup \{\infty\}$.
\end{proposition}

\begin{proof}
Let $G_j$ denote the Cauchy transform of $\bm^{(1)} \boxplus \cdots \boxplus \bm^{(j)}$. By \Cref{thm:subord_add}, for each $j = 1,\ldots,n-1$, we have analytic $\wt{\omega}_1^{(j)},\wt{\omega}_2^{(j)}: \C^+ \to \C^+$ such that
\begin{gather} \nonumber
G_{j+1}(z) = G_j(\wt{\omega}_1^{(j)}(z)) = G_{\bm^{(j+1)}}(\wt{\omega}_2^{(j)}(z)) \\ \label{eq:subord_j}
z = \wt{\omega}_1^{(j)}(z) + \wt{\omega}_2^{(j)}(z) - \frac{1}{G_{j+1}(z)}
\end{gather}
Set $\omega_1^{(1)} = \wt{\omega}_1^{(1)}$ and define iteratively in increasing $j = 1,\ldots,n-1$
\[ \omega_{j+1}^{(j+1)} = \wt{\omega}_2^{(j)}, \quad \quad \omega_i^{(j+1)} = \omega_i^{(j)} \circ \wt{\omega}_1^{(j)} \quad \quad (1 \le i \le j).  \]
Inducting in $j$, we have
\begin{gather*}
G_{j+1}(z) = G_{\bm^{(i)}}(\omega_i^{(j+1)}), \quad \quad 1 \le i \le j+1 \\
z = \omega_1^{(j+1)}(z) + \cdots + \omega_{j+1}^{(j+1)}(z) - \frac{j}{G_{j+1}(z)}.
\end{gather*}
The base case $j = 1$ is given by \eqref{eq:subord_j} and for $j > 1$ we use the previous induction step (replacing $j$ with $j - 1$ above), replace $z$ with $\wt{\omega}_1^{(j)}(z)$, and combine it with \eqref{eq:subord_j}. In particular, for $j = n-1$, this is exactly \eqref{omega:subord} and \eqref{omega:relation} where $\omega_i = \omega_i^{(n)}$ for $1 \le i \le n$.

Properties \eqref{omega:infty}, \eqref{omega:im_bd}, uniqueness, and continuity on $\C^+ \cup \R$ are inherited from $\wt{\omega}_1^{(j)},\wt{\omega}_2^{(j)}$.

If none of $\bm^{(1)},\ldots,\bm^{(n)}$ are single atoms, then \Cref{thm:subord_add} implies $\Im \wt{\omega}_i^j(z) \ge \Im z$ for $z \in \C^+$, $i = 1,2$, and $j = 1,\ldots,n$. Thus $\Im \omega_i(z) > \Im z$ for $z \in \C^+$.
\end{proof}

Let $\omega_1,\ldots,\omega_n$ be as in \Cref{thm:subord}. Then \eqref{omega:subord} and \Cref{thm:domain} imply that
\begin{align} \label{omega:G_rep}
\omega_i(z) = G_{\bm^{(i)}}^{[-1]} (G(z))
\end{align}
for $z$ in a neighborhood of $\infty$.

Our next task is to obtain a useful representation for the subordination functions in terms of a Cauchy transform of a measure. For this, we need the following representation theorem which is a converse of the fact that the Cauchy transform of a finite positive Borel measure on $\R$ is a map from $\C^+ \to \C^-$ and satisfies $1/z + o(1/|z|)$ as $z \to \infty$.

\begin{proposition} \label{thm:nevanlinna}
Suppose $\cG:\C^+ \to \C^-$ is analytic and $\limsup_{y \to \infty} y |\cG(\bi y)| = c < \infty$. Then there is a unique positive Borel measure $\nu$ on $\R$ such that
\[ \cG(z) = \int_\R \frac{1}{z - t} d\nu(t) \quad \quad \mbox{and} \quad \quad \nu(\R) = c. \]
\end{proposition}

This is a consequence of the Nevanlinna representation theorem for self-maps of the upper half plane (see e.g. \cite[Chapter III]{AKh65} or \cite[Lemma 3.1]{BES18}). We omit the proof, though the details can be found in \cite[Theorem 10]{MS17}.

For $\nu$ a finite positive Borel measure, we recall that the Cauchy transform admits the inversion formula
\begin{align} \label{eq:cauchy_inversion}
\nu(I) = \lim_{\e \to 0} -\frac{1}{\pi} \int_I \! \Im G_\nu(x+\bi \e) \, dx
\end{align}
for any interval $I \subset \R$ with nonempty interior such that the endpoints of $I$ are not atoms of $\nu$ (see e.g. \cite[Theorem 2.4.3]{AGZ10}).

\begin{proposition} \label{thm:omega:nevanlinna}
Let $\omega_1,\ldots,\omega_n$ be as in \Cref{thm:subord} and suppose none of the measures $\bm^{(1)},\ldots,\bm^{(n)}$ are single atoms. Then there exist finite positive Borel measures $\nu_1,\ldots,\nu_n$ such that
\begin{align} \label{subord_nevanlinna}
\omega_i(z) - z = \kappa_i + \int_\R \frac{d\nu_i(z)}{x - z}, \quad \quad z \in \C \setminus \supp \nu_i,
\end{align}
for some $\kappa_i \in \R$ and $\supp \nu_i \subset [E_-(\bm),E_+(\bm)]$ for $i = 1,\ldots,n$. In particular $\omega_i$ is real-valued and increasing on $(E_+(\bm),\infty)$ and $(-\infty,E_-(\bm))$.
\end{proposition}

\begin{proof}
Fix $i \in \{1,\ldots,n\}$ and let $G_i := G_{\bm^{(i)}}$. From the asymptotic expansions of $G(z)$ and $G_i^{[-1]}(u)$ near $z = \infty$ and $u = 0$ respectively, we have
\begin{align} \label{subord_expansion}
\omega_i(z) = G_i^{[-1]}(G(z)) = z - \kappa_i - c_i/z + O(1/|z|^2)
\end{align} 
in a neighborhood of $\infty$ for some $\kappa_i,c_i \in \R$ by \eqref{omega:G_rep}. Since none of $\bm^{(1)},\ldots,\bm^{(n)}$ are single atoms, \Cref{thm:subord} implies $\Im \omega_i(z) > \Im z$ for $z \in \C^+$ so that $\omega_i(z) - z$ is a map from $\C^+$ to $\C^+$. \Cref{thm:nevanlinna} and \eqref{omega:infty} imply that there exists a finite positive Borel measure $\nu_i$ such that
\[ z - \omega_i(z) - \kappa_i = G_{\nu_i}(z) \]
and $\nu_i(\R) = c_i$. We note that $c_i$ must be positive ($\nu_i$ is nonzero), otherwise violating the strict inequality $\Im \omega_i(z) > \Im z$.

We thus arrive at the integral representation \eqref{subord_nevanlinna} which extends to $z \in \C \setminus \supp \nu_i$. From \eqref{omega:subord} we know that $\Im \omega_i(z) = 0$ for $z \in \C \setminus [E_-(\bm),E_+(\bm)]$. By \eqref{eq:cauchy_inversion}, we obtain $\supp \nu_i \subset [E_-(\bm),E_+(\bm)]$.
\end{proof}

The monotonicity of $\omega_i$ implies that the limit
\[ \omega_i(E_+(\bm)) := \lim_{t \searrow E_+(\bm)} \omega_i(t) \]
is well-defined. By \eqref{eq:cauchy_inversion}, $(E_+(\bm^{(i)}),\infty)$ is the maximal right semi-infinite interval on which $G_{\bm^{(i)}}$ is real-valued. Since $G_{\bm^{(i)}}$ is also real-valued on $(\omega_i(E_+(\bm)),\infty)$ by \eqref{omega:subord}, we must have
\begin{align}
\omega_i(E_+(\bm)) \ge E_+(\bm^{(i)})
\end{align}
Plugging in $z = G^{[-1]}(u)$ into \eqref{omega:relation} and using \eqref{omega:G_rep}, we obtain
\begin{align} \label{eq:Ginv_relation}
G^{[-1]}(u) = G_{\bm^{(1)}}^{[-1]}(u) + \cdots + G_{\bm^{(n)}}^{[-1]}(u) + \frac{1 - n}{u}
\end{align}
for $u$ in a punctured neighborhood of $0$ containing $(G(E_-(\bm)),G(E_+(\bm))) \setminus \{0\}$.

\subsection{Proof of Theorem \ref{thm:submaster}}

Let $G$ denote the Cauchy transform of $\bm$, $E_\pm := E_\pm(\bm)$, and $E_\pm^{(i)} := E_\pm(\bm^{(i)})$.

\textbf{Proof of (i).} We rule out $\fz < G(E_+)$ and $\fz > G(E_+)$. If $\fz > G(E_+)$, then by minimality of $\fz$ and monotonicity \eqref{eq:G_monotone} of $G$, we have that $\cA$ is analytic with nonzero derivative in a neighborhood of $(0,G(E_+)]$. By \eqref{eq:Ginv_relation}, we know that $\cA = G^{[-1]}$ in a neighborhood of $0$. Then $\cA$ has an inverse in a connected neighborhood $U$ of $(0,G(E_+)]$ which must be $G$. Since $\cA$ is real valued on $U \cap \R$, this implies that $G$ is real valued in a neighborhood of $E_+$ contradicting the inversion formula \eqref{eq:cauchy_inversion} and that $\sup \supp \bm = E_+$. If $\fz < G(E_+)$, then there exists $x_0 \in (E_+,\infty)$ such that $\fz = G(x_0)$ so that
\[ \cA'(\fz) = \frac{1}{G'(x_0)} < 0 \]
which contradicts $\fz$ is a critical point. Thus $\fz = G(E_+)$.

\textbf{Proof of (ii).} Let $\vec{\ell}^{(i)} := \vec{\ell}^{(i)}(N)$ for $1 \le i \le n$. By iterated applications of \Cref{bessel_gf}, the multivariate Bessel generating function for $\vec{\ell}$ is given by
\[ \prod_{i=1}^n \frac{\cB_{\vec{\ell}^{(i)}}(z_1,\ldots,z_N)}{\cB_{\vec{\ell}^{(i)}}(0^N)} \]
so a supersymmetric lift is given by
\[ \wt{S}_N(z_1,\ldots,z_{N+k}/w_1,\ldots,w_k) = \prod_{i=1}^n \frac{\cB_{\vec{\ell}^{(i)},\fp_i}(z_1,\ldots,z_{N+k}/w_1,\ldots,w_k)}{\cB_{\vec{\ell}^{(i)}}(0^N)} \]
Set $\vec{u} = (u_1,\ldots,u_k), \vec{c} = (c_1,\ldots,c_k)$, write $\cS_{i,u} = \cS_{\mathrm{m}^{(i)},u}$ and $G_i(z) = G_{\mathrm{m}^{(i)}}(z)$. By \Cref{bessel_asymp}, if $u_1,\ldots,u_k \in \fO_{\bm^{(i)},\fp_i}$, then
\begin{align*}
& \frac{\cB_{\vec{\ell}^{(i)},\fp_i}(N \vec{u} + N^{2/3} \vec{c},0^N/N\vec{u})}{\cB_{\vec{\ell}^{(i)}}(0^N)} \\
& \quad \quad = F_N^{(i)}(\vec{u}) \prod_{j=1}^k \exp\Bigg( N \Big(\cS_{i,u_j + N^{-1/3} c_j} \big(G_i^{[-1]}(u_j + N^{-1/3} c_j)\big) - \cS_{i,u_j}\big(G_i^{[-1]}(u_j)\big)\Big) \Bigg) \\
& \quad \quad  = F_N^{(i)}(\vec{u}) \prod_{j=1}^k \exp\Bigg( N \int_{u_i}^{u_i + N^{-1/3} c_i} \left( G_i^{[-1]}(z) - \frac{1}{z} \right) dz \Bigg)
\end{align*}
where
\begin{align*}
F_N^{(i)}(\vec{u}) &= \frac{D\big(\vec{u} + N^{-1/3} \vec{c};-\vec{u}\big)}{\Delta\big(\vec{u} + N^{-1/3}\vec{c} \big) \Delta\big(-\vec{u}\big)} \frac{\Delta\big(G_i^{[-1]}(\vec{u} + N^{-1/3} \vec{c})\big) \Delta\big(-G_i^{[-1]}(\vec{u})\big)}{D\big(G_i^{[-1]}(\vec{u} + N^{-1/3} \vec{c});-G_i^{[-1]}(\vec{u})\big)} \\
& \quad \quad \times \prod_{j=1}^k \frac{1}{\sqrt{G_i'(G_i^{[-1]}(u_j + N^{-1/3} c_j)) G_i'(G_i^{[-1]}(u_j))}}
\end{align*}
and we use the fact that
\[ \frac{d}{du} \cS_{i,u}(G_i^{[-1]}(u)) = G_i^{[-1]}(u) - \frac{1}{u}. \]
Suppose $u,v$ tend to some $\fu \in \fO_{\bm^{(i)},\fp_i}$ as $N\to\infty$. Then
\[ G_i^{[-1]}(u) - G_i^{[-1]}(v) = \frac{1}{G_i'(G_i^{[-1]}(\fu))}(u - v)(1 + o(1)). \]
It follows that if $u_1,\ldots,u_k \to \fu$ as $N\to\infty$, then
\[ F_N^{(i)}(\vec{u}) = 1 + o(1). \]
Observe that
\[ \frac{u + N^{-1/3} c}{u} = \exp\left( \int_u^{u + N^{-1/3} c} \frac{1}{z} \, dz \right), \]
for $c > 0$ fixed, $u \ne 0$, and where the integral is over a horizontal line segment from $u$ to $u+N^{-1/3} c$. Thus, $\wt{S}_N$ satisfies \eqref{ssym_mbgf_limit} with 
\[ A_N(z) = \sum_{i=1}^n G_i^{[-1]}(z) + \frac{1 - n}{z} \]
which clearly converges to $\cA(u)$ on compact subsets of $\Omega = \bigcap_{i=1}^n \fO_{\bm^{(i)},\fp_i}$. Using \eqref{eq:curve_exists}, we conclude that the multivariate Bessel generating functions for $\vec{\ell}$ are Airy edge appropriate.

\textbf{Proof of (iii).} Suppose \eqref{eq:set_inclusion} holds. We must show \eqref{eq:curve_exists}, i.e. construct a curve $\gamma$ in
\[ \Omega := \bigcap_{i=1}^n \fO_{\bm^{(i)},\fp_i} \]
such that
\begin{enumerate}[(a)]
    \item $\fz \in \gamma$ with $\cA'(\fz) = 0$ and $\cA''(\fz) > 0$,
    \item $\gamma$ is positively oriented around $0$, and
    \item $\Re \cA(\fz) > \Re \cA(z)$ for $z \in \gamma \setminus \{\fz\}$.
\end{enumerate}
Since we assume that $\cA'(\fz) = 0$ and $\cA''(\fz) > 0$, (a) is satisfied by any contour through $\fz$.

Using the representation from \Cref{thm:omega:nevanlinna}, we have
\begin{align} \label{omega:representation}
\omega_i(z) - z = \kappa_i + \int_\R \frac{d\nu_i(x)}{x - z}
\end{align}
where $\nu_i$ is a finite Borel measure with $\supp \nu_i \subset [E_-,E_+]$. Thus we view $\omega_i$ as a continuous function on $\C \setminus [E_-,E_+)$ where the extension to $E_+$ is by monotonicity (we can of course extend to $E_-$ as well). From the representation above, we have $\Re \omega_i(x + \bi t)$ is increasing as $|t|$ increases for $x \ge E_+$. This gives us
\begin{align} \label{realpart_bdd_below_at_E+}
\omega_i(E_+) < \Re \omega_i(E_+ + \bi t), \quad \quad t \ne 0.
\end{align}
and
\begin{align} \label{omega:line}
\begin{split}
& G(\{z \in \C: \Re z \ge E_+\}) = G_i( \omega_i(\{z \in \C: \Re z \ge E_+\})) \\
& \quad \quad \subset G_i(\{z \in \C: \Re z \ge \omega_i(E_+)\}) = G_i(\{z \in \C: \Re z \ge G_i^{[-1]}(G(E_+)) \}) \subset \fO_{\bm^{(i)},\fp_i}
\end{split}
\end{align}
where the first equality uses \eqref{omega:subord}, the second equality uses \eqref{omega:G_rep}, and the final inclusion is our assumption \eqref{eq:set_inclusion}. In particular, the monotonicity of $G$ and positivity of $\fz$ gives 
\[ (0,G(E_+)] = G([E_+,\infty)) \subset \fO_{\bm^{(i)},\fp_i}. \]
By \Cref{thm:domain} (i),
\[ (0,G(E_+)] \subset (0,G_i(E_+^{(i)})). \]
Applying $G_i^{[-1]}$, we have
\[ [G_i^{[-1]}(G(E_+)),\infty) \subset (E_+^{(i)},\infty). \]
Using \eqref{omega:G_rep} again, we can choose $\e > 0$ so that
\[ \omega_i(E_+) - \e > E_+^{(i)}, \quad \quad i = 1,\ldots,n. \]

\begin{figure}[ht]
    \centering
\begin{subfigure}{.5\textwidth}
  \centering
  \includegraphics[width=0.8\linewidth]{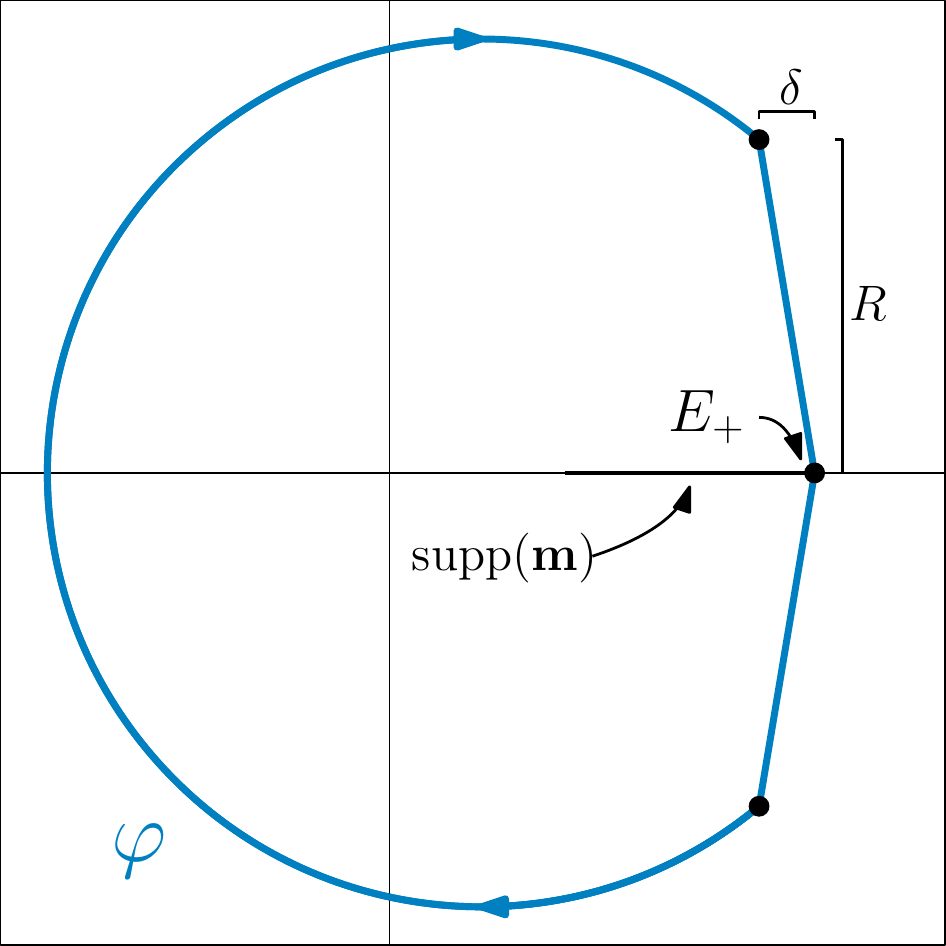}
\end{subfigure}%
\begin{subfigure}{.5\textwidth}
  \centering
  \includegraphics[width=0.8\linewidth]{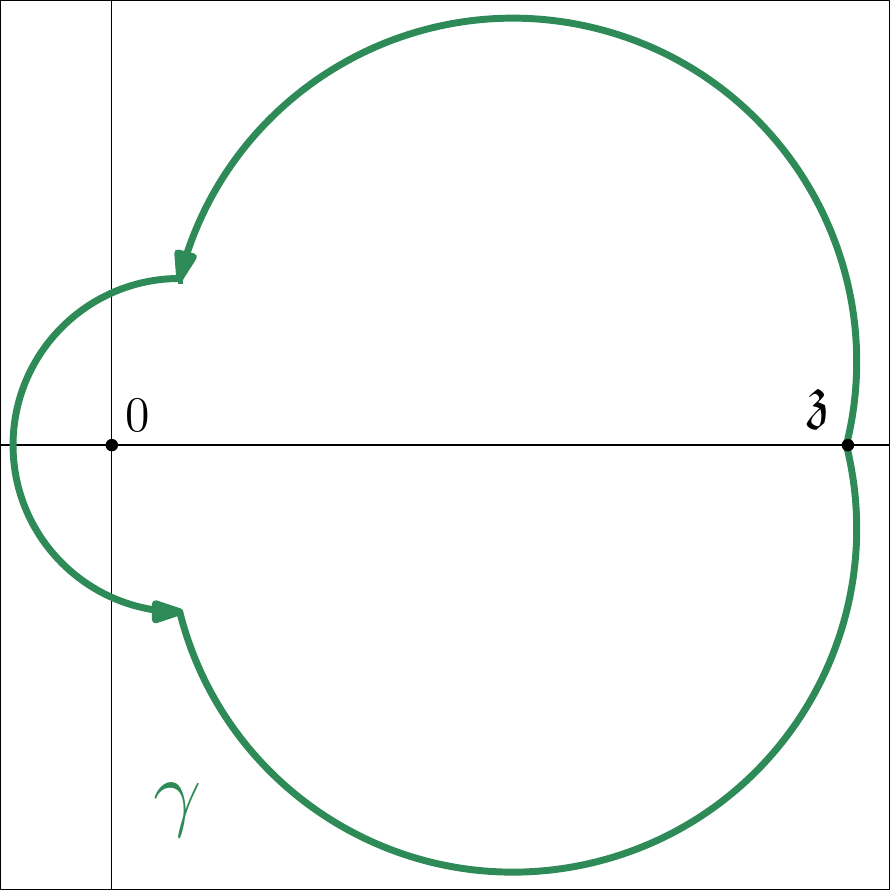}
\end{subfigure}
\caption{A depiction of the curve $\varphi$ (left) and its image $\gamma$ (right) under $G$.}
\label{fig:PhiGamma}
\end{figure}

Since $\omega_i(z) = z + O(1)$ as $z \to \infty$, we can choose $R > 0$ so that
\[ G_i(\omega_i(z)) \in \fO_{\bm^{(i)},\fp_i}, \quad \quad |z| \ge R \]
by \Cref{thm:domain} (ii). Fix a small $\delta > 0$ to be determined. Let $\varphi := \varphi^1 \cup \varphi^2$ be the curve where (i) $\varphi^1$ is the piecewise linear curve which is a line segment from $E_+ - \delta + \bi R$ to $E_+$ and from $E_+$ to $E_+ - \delta - \bi R$, and (ii) $\varphi^2$ is a clockwise arc from $E_+ - \delta - \bi R$ to $E_+ - \delta + \bi R$ with large enough radius so that $\varphi^2 \subset \{|z| \ge R\}$. Clearly, the real part of $\varphi$ is uniquely maximized at $E_+$. We have
\[ G(\varphi^j) = G_i(\omega_i(\varphi^j)) \subset \fO_{\bm^{(i)},\fp_i}, \quad \quad i = 1,\ldots,n \]
for $j = 2$ by our choice of $R$. By \eqref{omega:line}, the openness of $\fO_{\bm^{(i)},\fp_i}$, and the continuity of $G_i\circ \omega_i$ at $E_+$, we can choose $\delta > 0$ sufficiently small so that the above holds for $j = 1$ as well. Thus $\gamma := G(\varphi) \subset \Omega$, see \Cref{fig:PhiGamma} for an illustration of $\varphi$ and $\gamma$. Clearly $G(E_+) = \fz \in \gamma$, thus $\gamma$ satisfies (a). Since $\varphi$ is positively oriented around $\infty$, $\gamma$ satisfies (b). Finally, we have
\begin{align*}
\cA(G(z)) &= G^{[-1]}(G(z)) = z, \quad \quad z \in \varphi.
\end{align*}
Since the real part of $\varphi$ is maximized at $E_+$, we have
\[ \cA(G(E_+)) > \cA(u), \quad \quad u \in \gamma \setminus \{G(E_+)\}. \]
This proves $\gamma$ satisfies (c), completing the proof.

\qed

\subsection{Subordination for Quantized Operations and Proof of Theorem \ref{thm:submaster_q}}
Due to the parallels in the proofs, we point out only the modifications needed to adapt the proof of \Cref{thm:submaster} to that of \Cref{thm:submaster_q}. We emphasize that in the quantized setting, we have
\[ \bm = \bm^{(1)} \otimes \cdots \otimes \bm^{(n)} \]
and
\[ \cA(u) = G_{\bm^{(1)}}^{[-1]}(u) + \cdots + G_{\bm^{(n)}}^{[-1]}(u) + \frac{1 - n}{1 - e^{-u}} \]
in the setting of \Cref{thm:submaster_q}. We discuss analytic subordination in the quantized setting:

By definition of quantized free convolution, we have
\[ \cQ \bm = \cQ \bm^{(1)} \boxplus \cdots \boxplus \cQ \bm^{(n)}. \]
Thus \eqref{omega:subord}, \eqref{omega:relation} give us the existence of analytic maps $\omega_1,\ldots,\omega_n \in \C^+ \to \C^+$ such that
\[ G_{\cQ \bm}(z) = G_{\cQ \bm^{(i)}}(\omega_i(z)), \quad \quad z = \omega_1(z) + \cdots + \omega_n(z) + \frac{1 - n}{G_{\cQ \bm}(z)}. \]
Recall that $\omega_i$ extends to $\C \setminus \supp \cQ \bm$. By \eqref{eq:MK_correspondence}, we obtain
\begin{align} \label{omega:subord_q}
G_\bm(z) = G_{\bm^{(i)}}(\omega_i(z)), \quad \quad \frac{1}{G_{\cQ \bm}(z)} = \frac{1}{1 - e^{-G_\bm(z)}}
\end{align}
which imply
\[ \omega_i(z) = G_{\bm^{(i)}}^{[-1]}(G_\bm(z)), \quad \quad z = \omega_1(z) + \cdots + \omega_n(z) + \frac{1 - n}{1 - e^{-G_\bm(z)}} \]
where the former holds for $z$ in a neighborhood of $\infty$ and the latter for $z \in \C^+$. Setting $z = G_\bm^{[-1]}(u)$ for $u$ in a punctured neighborhood of $0$ in the latter, we obtain
\begin{align} \label{eq:Ginv_relation_q}
G_\bm^{[-1]}(u) = G_{\bm^{(1)}}^{[-1]}(u) + \cdots + G_{\bm^{(n)}}^{[-1]}(u) + \frac{1}{1 - e^{-u}}
\end{align}
for $u$ in some punctured neighborhood of $0$.

\textbf{Proof of (i).} The proof is the same as that of \Cref{thm:submaster} (i), where we use $\cA(u) = G_\bm^{[-1]}(u)$ from \eqref{eq:Ginv_relation_q}.

\textbf{Proof of (ii).}  Let $\lambda^{(i)} := \lambda^{(i)}(N)$ for $1 \le i \le n$. By iterated applications of \Cref{schur_gf}, the Schur generating function for $\lambda$ is given by
\[ \prod_{i=1}^n \frac{s_{\lambda^{(i)}}(e^{z_1},\ldots,e^{z_N})}{s_{\lambda^{(i)}}(1^N)} \]
so a supersymmetric lift is given by
\[ \wt{S}_N(z_1,\ldots,z_{N+k}/w_1,\ldots,w_k) = \prod_{i=1}^n \frac{s_{\lambda^{(i)},\fp_i}(e^{z_1},\ldots,e^{z_{N+k}}/e^{w_1},\ldots,e^{w_k})}{s_{\lambda^{(i)}}(1^N)}. \]
We can use \eqref{eq:normalized_bessel_schur} to write
\begin{align*}
& \wt{S}_N(u_1,\ldots,u_k,0^N/v_1,\ldots,v_k) = \prod_{i=1}^n \Bigg[ \prod_{j=1}^k \frac{1}{e^{v_j}} \left( \frac{u_j}{e^{u_j} - 1} \cdot \frac{e^{v_j} - 1}{v_j} \right)^N
\times \frac{D(e^{u_1},\ldots,e^{u_k};-e^{v_1},\ldots,-e^{v_k})}{D(u_1,\ldots,u_k;-v_1,\ldots,-v_k)} \\
& \quad \times \frac{\Delta(u_1,\ldots,u_k) \Delta(v_1,\ldots,v_k)}{\Delta(e^{u_1},\ldots,e^{u_k}) \Delta(e^{v_1},\ldots,e^{v_k})} \frac{\cB_{\lambda^{(i)}+\delta_N,N\fp_i}(u_1,\ldots,u_k,0^N/v_1,\ldots,v_k)}{\cB_{\lambda^{(i)}+\delta_N,N\fp_i}(0^N)} \Bigg].
\end{align*}
Set $\vec{u} = (u_1,\ldots,u_k)$, $\vec{c} = (c_1,\ldots,c_k)$, write $\cS_{i,u} = \cS_{\mathrm{m}^{(i)},u}$ and $G_i(z) = G_{\mathrm{m}^{(i)}}(z)$. We can show that
\begin{align*}
& \wt{S}_N(\vec{u} + N^{-1/3} \vec{c},0^N/\vec{u}) = F_N(\vec{u})\prod_{j=1}^k \left( \frac{u_j + N^{-1/3} c_j}{u_j} \frac{e^{u_j} - 1}{e^{u_j + N^{-1/3} c_j}} \right)^{nN} \\
& \quad \quad \times \prod_{i=1}^n \exp\Bigg( N \int_{u_i}^{u_i + N^{-1/3} c_i} \left( G_i^{[-1]}(z) - \frac{1}{z} \right) dz \Bigg)
\end{align*}
for $u_1,\ldots,u_k \in \Omega := \bigcap_{i=1}^n \fO_{\bm^{(i)},\fp_i}$,
where
\[ F_N(\vec{u}) = 1 + o(1) \]
if $u_1,\ldots,u_k \to \fu$ as $N\to\infty$ for some $\fu \in \Omega$. To see this, we use the fact that
\begin{align*}
\frac{\cB_{\lambda^{(i)}+\delta_N,N\fp_i}(\vec{u} + N^{-1/3} \vec{c},0^N/\vec{u})}{\cB_{\lambda^{(i)}+\delta_N,N\fp_i}(0^N)} = \frac{\cB_{\frac{\lambda^{(i)}+\delta_N}{N},\fp_i}(N\vec{u} + N^{2/3} \vec{c},0^N/N\vec{u})}{\cB_{\frac{\lambda^{(i)}+\delta_N}{N},\fp_i}(0^N)}
\end{align*}
where the right hand side has the same limit as in the proof of \Cref{thm:submaster} (ii) with $(\lambda^{(i)} + \delta_N)/N = \vec{\ell}^{(i)}$.

Observe that
\[ \frac{u + N^{-1/3} c}{u} = \exp\left( \int_u^{u + N^{-1/3} c} \frac{dz}{z} \right), \quad \quad 
\frac{e^{u + N^{-1/3} c} - 1}{e^{u_i} - 1}  = \exp\left( \int_u^{u+N^{-1/3}c} \frac{dz}{1 - e^{-z}} \right) \]
for $c > 0$ fixed, $u \ne 0$, and where the integrals are over a horizontal line segment from $u$ to $u+N^{-1/3} c$. Thus, $\wt{S}_N$ satisfies \eqref{schur_version_lim} with
\[ A_N(z) = \sum_{i=1}^n G_i^{[-1]}(z) + \frac{1 - n}{1 - e^{-z}} \]
which clearly converges to $\cA(u)$ on compact subsets of $\Omega$. Using \eqref{eq:curve_exists_q}, we conclude that the Schur generating functions for $\lambda$ are Airy edge appropriate.

\textbf{Proof of (iii).} The proof is verbatim identical to that of \Cref{thm:submaster} (iii), where we use the subordination functions $\omega_1,\ldots,\omega_n$ obtained in the quantized setting.

\section{Unitarily Invariant Random Matrices} \label{sec:applications}

In this section, we prove \Cref{thm:multisum,thm:two_sum}. We start by reformulating \Cref{thm:multisum} in terms of Airy edge appropriateness. Given $\mu \in \cM$, let
\begin{align} \label{eq:tau}
x(\mu) := 4(E_+(\mu) - E_-(\mu)) + E_+(\mu), \quad \quad \tau(\mu) := \left( 1 + \frac{G(x(\mu))^2}{G'(x(\mu))} \right)^{-1}.
\end{align}

\begin{theorem} \label{thm:application}
Under the hypotheses of \Cref{thm:multisum} with $\tau(\cdot)$ as above, the multivariate Bessel generating functions for $\vec{\ell}$ form an Airy appropriate sequence where
\[ A_N(u) = \sum_{i=1}^n G_{\mathrm{m}^{(i)}}^{[-1]}(u) + \frac{1 - n}{u}, \]
$\mathrm{m}^{(i)} := \tfrac{1}{N} \sum_{j=1}^N \delta_{\ell_j^{(i)}}$, and $\fz_N$ is the minimal positive critical point of $A_N$.
\end{theorem}

\begin{remark} \label{rmk:tau_optimize}
By tracking the arguments, the conclusion of \Cref{thm:application} still holds if we replace $x(\mu)$ with $c(E_+ - E_-) + E_+$ in the definition of $\tau(\mu)$ where $c$ solves $\log c - \frac{c+1}{c} = 0$. Then $c \approx 3.59112$, reducing the value of $\tau(\mu)$.
\end{remark}

\begin{remark} \label{rmk:tau_optimal}
We believe that \Cref{thm:samelimit} (and also \Cref{thm:multisum}) should continue to hold with $\tau(\mu)$ defined as above and $x(\mu)$ replaced by $E_+(\mu) + \e$ for any $\e > 0$. We expect this to be optimal for Airy universality in the setting of \Cref{thm:samelimit}. 
\end{remark}

We state a general theorem about Airy fluctuations for two matrix summands which implies \Cref{thm:two_sum}.

\begin{definition}
Let $\fM$ be the set of $\bm \in \cM$ such that
\begin{itemize}
    \item $\bm$ has a continuous nonzero density $f$ on $(E_-,E_+)$ satisfying
    \begin{align} \label{jacobi_hypothesis}
    C^{-1} \le \frac{f(x)}{(E_+ - x)^t} \le C, \quad \quad x \in (E_+ - \e, E_+)
    \end{align}
    for some $C \ge 1$, $0 \le t < 1$, and $\e > 0$ sufficiently small, and

    \item for each $\xi > E_+$,
    \begin{align} \label{eq:set_inclusion_xi}
    G_\bm(\{ z \in \C: \Re z \ge \xi\}) \subset \fO_\fp
    \end{align}
    for some $\fp \in \C \setminus \supp \bm$.
\end{itemize}
\end{definition}

\begin{remark}
The set $\fM$ is closed under translation and dilation.
\end{remark}

\begin{theorem} \label{thm:general_twosum}
Under \Cref{assum:rmt_conv}, suppose $\bm^{(1)},\bm^{(2)} \in \fM$ and $\vec{\ell} \in \cW_{\R}^N$ are the eigenvalues of
\[ X^{(1)}_N + X^{(2)}_N. \]
Then the multivariate Bessel generating functions for $\vec{\ell}$ form an Airy appropriate sequence where
\[ A_N(u) := G_{\mathrm{m}^{(1)}}^{[-1]}(u) + G_{\mathrm{m}^{(2)}}^{[-1]}(u) - \frac{1}{u}, \quad \quad u \in \fO_{\bm^{(1)}} \cap \fO_{\bm^{(2)}}, \]
$\mathrm{m}^{(i)} := \tfrac{1}{N} \sum_{j=1}^N \delta_{\ell_j^{(i)}}$, and $\fz_N$ is the minimal critical point of $A_N$.
\end{theorem}

The remainder of this section is devoted to proofs of \Cref{thm:application,thm:general_twosum}. The fact that \Cref{thm:general_twosum} implies \Cref{thm:two_sum} is shown in \Cref{ssec:two_sums}. The main work is in analyzing the Cauchy transform and subordination functions to show that the hypotheses of \Cref{thm:submaster} are satisfied. Throughout this section, we write
\[ G'(E_+) = \lim_{z \searrow E_+} G'(z) \]
where the right hand side limit exists in $[-\infty,\infty)$ by monotonicity.

\subsection{Many Self Convolutions, Stability Under Addition, and Proof of Theorem \ref{thm:application}}

We use two intermediate results in the proof of \Cref{thm:application} beyond \Cref{thm:submaster}. The first (\Cref{thm:high_compression}) asserts that the inverse Cauchy transform of $\mu^{\boxplus k}$ has a minimal real positive critical point with positive second derivative for $k$ suitably large. The second (\Cref{thm:free_convolution_cp}) gives conditions under which this property can be propagated by free convolutions. The latter result is used again later to prove \Cref{thm:general_twosum}.

We start with a useful lemma.

\begin{lemma} \label{variance}
Fix $\mu \in \cM$ and let $G := G_\mu$. For $z \in \R \setminus [E_-,E_+]$, we have
\begin{gather*}
- \frac{G(z)^2}{G'(z)} \le 1, \\
G(z) G''(z) - 2 G'(z)^2 \ge 0.
\end{gather*}
Furthermore, $-G(z)^2/G'(z)$ is increasing on $(E_+,\infty)$ and decreasing on $(-\infty,E_-)$ such that
\[ \lim_{z \to \pm\infty} - \frac{G(z)^2}{G'(z)} = 1. \]
The inequalities and monotonicity statements are strict if $\mu$ is not a single atom.
\end{lemma}
\begin{proof}
By Cauchy-Schwarz, for $z \in \R \setminus [E_-,E_+]$ we have the inequalities
\[ G(z)^2 \le -G'(z), \quad \quad G'(z)^2 \le \frac{1}{2} G(z) G''(z) \]
where we use the fact that the sign of $\frac{1}{z - x}$ for $x \in \supp(\mu)$ is negative for $z < E_-$ and positive for $z > E_+$. The inequalities above imply the desired inequalities. In particular, we have
\[ - \frac{d}{dz} \frac{G(z)^2}{G'(z)} = \frac{G(z)^2 G''(z) - 2G(z) G'(z)^2}{G'(z)^2} \quad \left\{ \begin{array}{ll}
\ge 0 & \mbox{if $z > E_+$}, \\
\le 0 & \mbox{if $z < E_-$}. \end{array} \right. \]
This proves the monotonicity statement. Observe that the inequalities are equalities if and only if $\mu$ is a single atom. The asymptotics $G(z)^2/G'(z) \to 1$ as $z \to \pm \infty$ follow from $G(z) \sim -1/z^2$ and $G(z) \sim 1/z$ in leading order.
\end{proof}

\begin{proposition} \label{thm:high_compression}
Let $\mu \in \cM$, $G := G_\mu$, and $E_\pm := E_\pm(\mu)$. If $n \ge \tau(\mu)$, then
\[ \mathsf{A}(u) := nG^{[-1]}(u) + \frac{1 - n}{u} \]
has a minimal real positive critical point $\mathsf{z} \in (0,G(E_+))$ such that $\mathsf{A}''(\mathsf{z}) > 0$ and
\begin{align} \label{z_distance}
G^{[-1]}(\mathsf{z}) - E_+ \ge 4(E_+ - E_-).
\end{align}
Furthermore, $\mathsf{z} = G_{\mu^{\boxplus n}}(E_+(\mu^{\boxplus n}))$ and $G_{\mu^{\boxplus n}}'(E_+(\mu^{\boxplus n})) = -\infty$.
\end{proposition}

\begin{proof}
We first show that $\mathsf{A}$ has a minimal real positive critical point. Plug in $u = G(z)$ to obtain
\begin{align*}
\mathsf{A}(G(z)) &= n z + \frac{1 - n}{G(z)} \\
\frac{d}{dz} \mathsf{A}(G(z)) &= n + (n - 1) \frac{G'(z)}{G(z)^2}
\end{align*}
valid in a neighborhood of $(E_+,\infty)$. By \Cref{variance}, $G'/G^2$ is negative and strictly increasing on $(E_+,\infty)$ with limit $-1$ at $\infty$. The inequality $n \ge \tau(\mu)$ implies
\[ \frac{1}{1 - 1/n} \le \frac{1}{1 - 1/\tau(\mu)} = - \frac{G'(x(\mu))}{G(x(\mu))^2} \]
so that
\[ \lim_{z \nearrow \infty} \frac{d}{dz} \mathsf{A}(G(z)) = 1, \quad \quad \frac{d}{dz} \mathsf{A}(G(z)) \Big|_{z=x(\mu)} \le 0. \]
Thus, there is a unique critical point $\xi$ of $\mathsf{A}\circ G$ in $[x(\mu),\infty)$, and
\[ \frac{d^2}{dz^2} \mathsf{A}(G(z)) > 0. \]
Then $\mathsf{z} := G(\xi)$ is the minimal positive critical point of $\mathsf{A}$ in $(0,G(x(\mu))]$ and
\[ \mathsf{A}''(\mathsf{z}) > 0. \]
By definition of $x(\mu)$ and the fact that $\xi \ge x(\mu)$, we obtain \eqref{z_distance}.

By \Cref{thm:submaster} (i), $\mathsf{z} = G_{\mu^{\boxplus n}}(E_+(\mu^{\boxplus n}))$. By \eqref{eq:Ginv_relation}, we have
\[ \mathsf{A}(u) = G_{\mu^{\boxplus n}}^{[-1]}(u) \]
for $u \in (0,\mathsf{z})$ on which it is strictly decreasing from $\infty$ to $E_+(\mu^{\boxplus n})$. On the other hand
\[ 1 = \frac{d}{dz} \mathsf{A}(G_{\mu^{\boxplus n}}(z)) = G_{\mu^{\boxplus n}}'(z) \mathsf{A}'(G_{\mu^{\boxplus n}}(z)), \quad \quad z \in (0,E_+(\mu^{\boxplus n})). \]
Letting $z \searrow E_+(\mu^{\boxplus n})$, we see that $G_{\mu^{\boxplus n}}'(E_+(\mu^{\boxplus n})) = -\infty$ because $\mathsf{A}'(\mathsf{z}) = 0$.
\end{proof}

\begin{remark} \label{rmk:power_law}
Let $\mu \in \cM$, $G := G_\mu$, and $E_+ := E_+(\mu)$. Using the ideas from the proof of \Cref{thm:high_compression}, we see that if
\begin{align} \label{eq:sq_root_necessary}
n < \left(1 + \frac{G(E_+)^2}{G'(E_+)} \right)^{-1},
\end{align}
then $G_{\mu^{\boxplus n}}'(E_+(\mu^{\boxplus n})) > -\infty$. Then $\mu^{\boxplus n}$ cannot have \emph{square root behavior} at its right edge, that is if $f(x)$ is the density of $\mu^{\boxplus n}$ then it cannot be the case that
\[ C^{-1} < \frac{f(x)}{(E_+(\mu^{\boxplus n}) - x)^{1/2}} < C \]
for $x \in [E_+(\mu^{\boxplus n}) - \e,E_+(\mu^{\boxplus n})]$ and some constant $C \ge 1$. Indeed, if $\mu^{\boxplus n}$ has square root behavior, then one can show $G_{\mu^{\boxplus n}}'(E_+(\mu^{\boxplus n})) = -\infty$. By direct computation, if $d\mu(x) = (1 - x)^p \1_{[0,1]}(x)\,dx$ for $p > 1$, then
\[ \left(1 + \frac{G(1)^2}{G'(1)} \right)^{-1} = p^2. \]
Thus for any given large $n$, one can always find a pair $(n,\mu)$ such that $\mu^{\boxplus n}$ does not exhibit square root behavior, by choosing $p$ correspondingly large enough.
\end{remark}

\begin{proposition} \label{thm:free_convolution_cp}
Let $\bm^{(1)},\bm^{(2)} \in \cM$, $\bm := \bm^{(1)} \boxplus \bm^{(2)}$ and $\omega_1,\omega_2$ be as in \Cref{thm:subord_add}. If
\[ \lim_{z \searrow E_+(\bm^{(i)})} \frac{G_{\bm^{(i)}}'(z)}{G_{\bm^{(i)}}(z)^2} = -\infty, \quad i = 1,2, \]
and
\[ \cA(u) := G_{\bm^{(1)}}^{[-1]}(u) + G_{\bm^{(2)}}^{[-1]}(u) - \frac{1}{u}, \quad \quad u \in \fO_{\bm^{(1)}} \cap \fO_{\bm^{(2)}}, \]
then $\omega_i(E_+(\bm)) > E_+^{(i)}(\bm)$ for $i = 1,2$, and $\fz := G_\bm(E_+(\bm)) \in \left(0, \min_{i=1,2} G_{\bm^{(i)}}(E_+(\bm^{(i)}))\right)$ is the minimal positive critical point of $\cA$. Moreover,
\[ \cA''(\fz) > 0. \]
\end{proposition}

\begin{proof}
Write $G_i := G_{\bm^{(i)}}, E_\pm^{(i)} := E_\pm(\bm^{(i)})$ for $i = 1,2$, $G := G_\bm$, $E_\pm := E_\pm(\bm)$. Let
\[ I(z) := \left(1 - \frac{\Im z}{\Im \omega_1(z)}\right)\left(1 - \frac{\Im z}{\Im \omega_2(z)}\right), \quad \quad z \in \C \setminus [E_-,E_+]. \]
By \eqref{subord_imz}, we have
\begin{align} \label{eq:variance_bd}
I(z) \le 1, \quad \quad z \in \C \setminus [E_-,E_+].
\end{align}
From the equalities
\begin{gather*}
\frac{\int \frac{d\bm^{(i)}(x)}{|\omega_i(z) - x|^2} - \left| \int \frac{d\bm^{(i)}(x)}{\omega_i(z) - x} \right|^2}{\left| \int \frac{d\bm^{(i)}(x)}{\omega_i(z) - x} \right|^2} = \frac{\Im \frac{1}{G_i(\omega_i(z))}}{\Im \omega_i(z)} - 1 = \frac{\Im \frac{1}{G(z)}}{\Im \omega_i(z)} - 1, \quad \quad i = 1,2, \\
\left( \frac{\Im \frac{1}{G(z)}}{\Im \omega_1(z)} - 1 \right) \left( \frac{\Im \frac{1}{G(z)}}{\Im \omega_2(z)} - 1 \right) = \left( \frac{\Im \omega_2(z)}{\Im \omega_1(z)} - \frac{\Im z}{\Im \omega_1(z)} \right) \left( \frac{\Im \omega_1(z)}{\Im \omega_2(z)} - \frac{\Im z}{\Im \omega_2(z)} \right) = I(z)
\end{gather*}
where the latter follows from \eqref{subord_relation}, we obtain
\begin{align} \label{variance_equality}
I(z) = - \frac{G_i'(\omega_i(z))}{G_i(\omega_i(z))^2} - 1 = - \frac{G_i'(\omega_i(z))}{G(z)^2} - 1, \quad \quad z: \omega_i(z) \in \R \setminus [E_-^{(i)},E_+^{(i)}]
\end{align}

By monotonicity, $\omega_i(E_+) := \lim_{z \searrow E_+} \omega_i(z)$ exists, for $i = 1,2$. We claim that
\begin{align} \label{eq:omega>E}
\omega_i(E_+) > E_+^{(i)}, \quad \quad i = 1,2.
\end{align}
Assume for contradiction that $\omega_1(E_+) \le E_+^{(1)}$. By strict monotonicity, $\omega_1([E_+,\infty)) = [\omega_1(E_+),\infty)$, and there exists a unique $x_1 \in [E_+,\infty)$ such that $\omega_1(x_1) = E_+^{(1)}$. If $\omega_2(E_+) \le E_+^{(2)}$, then similarly there is a unique $x_2 \in [E_+^{(2)},\infty)$ such that $\omega_2(x_2) = E_+^{(2)}$. In this case, assume without loss of generality that $x_1 \ge x_2$. Then, in any case, $\omega_2([x_1,\infty)) \subset [E_+^{(2)},\infty)$. Thus
\[ \lim_{z \searrow x_1} \left( -\frac{G_1'(\omega_1(z))}{G_1(\omega_1(z))^2} - 1 \right)\left( -\frac{G_2'(\omega_2(z))}{G_2(\omega_2(z))^2} - 1 \right) = \infty \]
because, as $z$ approaches $x_1$ from above, both factors are increasing by \Cref{variance}, and the first factor diverges to $\infty$ by our hypothesis. However, this contradicts \eqref{eq:variance_bd} by way of \eqref{variance_equality}. This proves \eqref{eq:omega>E}.

We now know that $G(E_+)$ is finite because $G(E_+) = G_i(\omega_i(E_+)) < G_i(E_+^{(i)})$. By \Cref{thm:domain} (i), $(0,G(E_+)] \subset \fO_{\bm^{(1)}} \cap \fO_{\bm^{(2)}}$. By \eqref{eq:Ginv_relation}, we have
\[ \cA(u) = G^{[-1]}(u) \]
on $(0,G(E_+))$. In particular, $\cA$ is strictly decreasing on $(0,G(E_+))$ by \eqref{eq:G_monotone}. So if $G(E_+)$ is a critical point of $\cA$, it is the minimal positive critical point.

We show that $G(E_+)$ is a critical point of $\cA$. By \eqref{subord_relation}, we may write
\[ \cA(G(z)) = z = \omega_1(z) + \omega_2(z) - \frac{1}{G(z)} \]
valid for $z$ in a neighborhood of $(E_+,\infty)$. Then
\[ \cA'(G(z)) = \frac{\omega_1'(z)}{G'(z)} + \frac{\omega_2'(z)}{G'(z)} + \frac{1}{G(z)^2}, \quad \quad z \in (E_+,\infty) \]
Taking derivatives on either side of \eqref{subord_cauchy} implies
\[ \cA'(G(z)) = \frac{1}{G_1'(\omega_1(z))} + \frac{1}{G_2'(\omega_2(z))} + \frac{1}{G(z)^2}. \]
Thus the critical point equation we seek is
\[ \frac{1}{G_1'(\omega_1(E_+))} + \frac{1}{G_2'(\omega_2(E_+))} + \frac{1}{G(E_+)^2} = 0. \]
This is equivalent to
\begin{align} \label{variance_saturation}
I(E_+) = 1
\end{align}
by \eqref{variance_equality}. In other words, we want to show that \eqref{eq:variance_bd} is saturated at $E_+$.

We now prove \eqref{variance_saturation}. Since $\omega_i(E_+) > E_+^{(i)}$, we may choose $\e > 0$ small enough so that $\omega_i([E_+ - \e, E_+])$ is contained in a compact subset of $\{ \omega \in \C^+ \cup \R: \Re \omega > E_+^{(i)} \}$ for $i = 1,2$. Then by the observation
\begin{align} \label{inversion_subord}
\lim_{t \searrow 0} \Im G(x + \bi t) = \lim_{t \searrow 0} \Im G_i(\omega_i(x + \bi t)), \quad \quad x \in [E_+ - \e, E_+], \quad i = 1,2,
\end{align} 
we have for $x \in [E_+ - \e, E_+]$ that $\omega_1(x + \bi t) \in \R$ if and only if $\omega_2(x + \bi t) \in \R$. Indeed \eqref{inversion_subord} is $0$ exactly when $\omega_i(x)$ is real for $i = 1,2$ --- note that we used the fact that $\omega_i(x)$ is away from the support of $\bm^{(i)}$ when $x \in [E_+ - \e, E_+]$.

We claim there exists a real monotone sequence $\{x_n \}$ in $(E_+ - \e,E_+)$ such that $x_n \nearrow E_+$ as $n \to \infty$ and $\Im \omega_1(x_n) > 0$; therefore $\Im \omega_2(x_n) > 0$ on this sequence as well. Assume for contradiction that such a sequence $\{x_n\}$ does not exist. Then there would exist some $\delta \in (0,\e)$ such that $\omega_1(x) \in \R$ for $x \in [E_+ - \delta, E_+)$. From \eqref{inversion_subord}, we would have
\[ \lim_{t \searrow 0} \Im G(x + \bi t) = 0, \quad \quad x \in [E_+ - \delta,E_+). \]
The inversion formula for the Cauchy transform would then imply $[E_+ - \delta, E_+) \subset \R \setminus \supp \bm$ which can only happen if $\bm$ has an atom at $E_+$. However, this would contradict the finiteness of $G(E_+)$, proving the claim. On this sequence $\{x_n\}$, we have
\[  \lim_{t \searrow 0} \left( 1 -  \frac{t}{\Im \omega_1(x_n + \bi t)} \right) \left( 1 -  \frac{t}{\Im \omega_2(x_n + \bi t)} \right) = 1. \]
Then $I(x_n) = 1$ since $\omega_i(x_n) > E_+^{(i)}$ for $i = 1,2$. By taking $n\to\infty$, we obtain \eqref{variance_saturation}. We have shown that $E_+$ is the minimal positive critical point of $\cA$.

It remains to show that $\cA''(G(E_+)) > 0$. It will be convenient to set
\[ F_i(z) = \frac{1}{G_i(z)}, \quad \quad i = 1,2. \]
Since $G(E_+) \in \fO_{\bm^{(2)}}$ and $\fO_{\bm^{(2)}}$ is open, we have $G_2^{[-1]}$ defined in a neighborhood of $G(E_+)$, and likewise the functional inverse of $F_2$ is defined in a neighborhood of $F(E_+)$. Moreover,
\[ F_2^{-1}(F_1(\omega_1(z))) = G_2^{[-1]}(G_1(\omega_1(z))) = \omega_2(z) \]
in a neighborhood of $E_+$ by \eqref{subord_cauchy}. Then
\[ \cA(G_1(\omega)) = \omega + G_2^{[-1]}(G_1(\omega)) - \frac{1}{G_1(\omega)} = \omega + F_2^{-1}(F_1(\omega)) - F_1(\omega) \]
in a neighborhood of $\omega_1(E_+)$. We have
\[ \frac{d^2}{d\omega^2} \cA(G_1(\omega)) = \frac{F_1''(\omega)}{(F_2'\circ F_2^{-1}\circ F_1)(\omega)} - \frac{F_1'(\omega)^2}{(F_2'\circ F_2^{-1}\circ F_1)(\omega)^3} (F_2''\circ F_2^{-1}\circ F_1)(\omega) - F_1''(\omega). \]
The left hand side can also be expressed as
\[ \frac{d^2}{d\omega^2} \cA(G_1(\omega))  = \cA''(G_1(\omega)) G_1'(\omega)^2 + \cA'(G_1(\omega)) G_1''(\omega). \]
Using \eqref{subord_cauchy} and the fact that $G(E_+)$ is a critical point of $\cA$, we arrive at
\[ \cA''(G(E_+)) G_1'(\omega_1(E_+))^2 = - \frac{F_1''(\omega_1(E_+))}{F_2'(\omega_2(E_+))} \Big(F_2'(\omega_2(E_+)) - 1 \Big) - \frac{F_1'(\omega_1(E_+))^2}{F_2'(\omega_2(E_+))^3} F_2''(\omega_2(E_+)) \]
upon evaluating at $\omega = \omega_1(E_+)$. The right hand side is positive because
\begin{align*}
F_i'(\omega_i(E_+)) - 1 &= \frac{- G_i'(\omega_i(E_+)) - G_i(\omega_i(E_+))^2}{G_i(\omega_i(E_+))^2} > 0 \\
F_i''(\omega_i(E_+)) &= \frac{- G_i''(\omega_i(E_+)) G_i(\omega_i(E_+)) + 2 G_i'(\omega_i(E_+))^2}{G_i(\omega_i(E_+))^3} < 0
\end{align*}
for $i = 1,2$,
where both lines follow from \Cref{variance} and $\omega_i(E_+) > E_+^{(i)}$. Thus $\cA''(G(E_+)) > 0$.
\end{proof}

\begin{proof}[Proof of \Cref{thm:application}]
Our goal is to apply \Cref{thm:submaster} (iii) to
\[ \cA(u) := G_{\bm^{(1)}}^{[-1]}(u) + \cdots + G_{\bm^{(n)}}^{[-1]}(u) + \frac{1 - n}{u} = G_{\mu_1^{\boxplus n_1}}^{[-1]}(u) + \cdots + G_{\mu_k^{\boxplus n_k}}^{[-1]}(u) + \frac{1 - k}{u} \]
where the equality follows from
\[ \sum_{j=n_1 + \cdots + n_{i-1}+1}^{n_1 + \cdots + n_i} G_{\bm^{(j)}}^{[-1]}(u) + \frac{1 - n_i}{u} = n_i G_{\mu_i}^{[-1]}(u) + \frac{1 - n_i}{u} = G_{\mu_i^{\boxplus n_i}}^{[-1]}(u) \]
coming from our assumption and \eqref{eq:Ginv_relation}.

Let $\cG_j := G_{\mu_1^{\boxplus n_1} \boxplus \cdots \boxplus \mu_j^{\boxplus n_j}}$ and $\mathsf{E}_j := E_+(\mu_1^{\boxplus n_1} \boxplus \cdots \boxplus \mu_j^{\boxplus n_j})$ for $j = 1,\ldots,k$. Set
\[ \cA_j(u) := G_{\mu_1^{\boxplus n_1}}^{[-1]}(u) + \cdots + G_{\mu_j^{\boxplus n_j}}^{[-1]}(u) + \frac{1 - j}{u}. \]
We claim that for each $j = 1,\ldots,k$, $\cG_j(\mathsf{E}_j) \in \left(0, \min_{1 \le i \le j} G_{\mu_i^{\boxplus n_i}}(E_+(\mu_i^{\boxplus n_i}))\right)$ is the minimal real positive critical point of $\cA_j$ and
\[ \cA_j''(\cG_j(\mathsf{E}_j)) > 0. \]
We proceed by induction. The claim for $j = 1$ follows from \Cref{thm:high_compression}. Suppose we know the claim holds for $j-1 \in\{1,\ldots,n-1\}$. Since
\[ \cA_{j-1}(u) = \cG_{j-1}^{[-1]}(u) \]
for $u \in (0,\cG_{j-1}(\mathsf{E}_{j-1}))$, we take a derivative and send $u \to \cG_{j-1}(\mathsf{E}_{j-1})$ to see that $\cG_{j-1}'(\mathsf{E}_{j-1}) = -\infty$. Thus
\[ \lim_{z \searrow \mathsf{E}_{j-1}} \frac{\cG_{j-1}'(\mathsf{E}_{j-1})}{\cG_{j-1}(\mathsf{E}_{j-1})^2} = - \infty, \quad \quad \lim_{z \searrow E_+(\mu_j^{\boxplus n_j})} \frac{G_{\mu_j^{\boxplus n_j}}'(z)}{G_{\mu_j^{\boxplus n_j}}(z)^2} = -\infty \]
where the latter follows from \Cref{thm:high_compression}. Writing
\[ \cA_j(u) = \cG_{j-1}^{[-1]}(u) + G_{\mu_j^{\boxplus n_j}}^{[-1]}(u) - \frac{1}{u}, \]
our claim follows for $j$ by  \Cref{thm:free_convolution_cp}. In particular, for $j = k$, we have that $\cA(u)$ has a minimal positive critical point $\fz := G(E_+) \in \left(0, \min_{1 \le i \le k} G_{\mu_i^{\boxplus n_i}}(E_+(\mu_i^{\boxplus n_i}))\right)$ where $G$ is the Cauchy transform of $\mu := \mu_1^{\boxplus n_1} \boxplus \cdots \boxplus \mu_k^{\boxplus n_k}$, $E_+ := E_+(\mu)$, and $\cA''(\fz) > 0$.

By \Cref{thm:high_compression} and the assumption $n_i \ge \tau(\mu_i)$, we have
\[ G_{\mu_i}^{[-1]}(\fz) - E_+(\mu_i) > G_{\mu_i}^{[-1]}(G_{\mu_i^{\boxplus n_i}}(E_+(\mu_i^{\boxplus n_i}))) - E_+(\mu_i) > 4\left( E_+(\mu_i) - E_-(\mu_i) \right) \]
so that
\[ \mathrm{dist}(z,\mu_i) > 4\left( E_+(\mu_i) - E_-(\mu_i) \right), \quad \quad \Re z \ge G_{\mu_i}^{[-1]}(\fz) \]
for each $1 \le i \le n$. Choose 
$\fp_i \in \C \setminus \supp \mu$ so that $|\fp_i - x| \le E_+(\mu_i) - E_-(\mu_i)$ for $x \in \supp \mu_i$. Then
\[ G_{\mu_i}(\{z \in \C: \Re z \ge G_{\mu_i}^{[-1]}(\fz)\}) \subset \fO_{\mu_i,\fp_i} \]
holds by \Cref{thm:domain} (ii). Thus the hypotheses of \Cref{thm:submaster} (iii) are satisfied.
\end{proof}

\subsection{Two Free Summands, Proof of Theorem \ref{thm:general_twosum}} \label{ssec:two_sums}`

\begin{proof}[Proof of \Cref{thm:general_twosum}]
We apply \Cref{thm:submaster} (iii) to
\[ \cA(u) := G_{\bm^{(1)}}^{[-1]}(u) + G_{\bm^{(2)}}^{[-1]}(u) - \frac{1}{u}. \]
Let $\omega_1,\omega_2$ denote the associated subordination functions. Fix $i \in \{1,2\}$. Assume $t > -1$ so that the integral $\int_0^1 (1 - x)^t \, dx$ converges. Given any integer $k \ge 1$, we have
\[ \int_0^1 \frac{(1 - x)^t}{(z - x)^k} \, dx \sim
\begin{cases}
\log|z - 1| & \mbox{if $t = k - 1$}, \\
|z - 1|^{t-k+1} & \mbox{if $t < k - 1$}
\end{cases} \]
as $z$ approaches $1$ with $|\arg(z - 1)| < \pi - c$ for some small constant $c > 0$. Since $\bm^{(i)}$ satisfies \eqref{jacobi_hypothesis} for some exponent $-1 < t_i < 1$, we have
\begin{align*}
\int \frac{d\bm^{(i)}(x)}{(\omega - x)^2} & \sim (\omega - E_+^{(i)})^{t_i-1}, \quad \quad \omega \searrow E_+^{(i)}; \\
\left| \int \frac{d\bm^{(i)}(x)}{\omega - x} \right| & \sim \left\{ \begin{array}{cl}
    \big|\log(\omega - E_+^{(i)})\big| \vee (\omega - E_+^{(i)})^{t_i} & \mbox{if $-1 < t_i \le 0$} \\
1 & \mbox{if $0 < t_i < 1$}
\end{array} \right. , \quad \quad \omega \searrow E_+^{(i)},
\end{align*}
which implies
\[ \lim_{\omega \searrow E_+^{(i)}} - \frac{G_i'(\omega)}{G_i(\omega)^2} \to \infty. \]
Thus \Cref{thm:free_convolution_cp} implies $\omega_i(E_+) > E_+^{(i)}$, $\fz := G(E_+) \in \left( 0, \min_{i=1,2} G_i(E_+^{(i)}) \right)$ is the minimal positive critical point of $\cA$ and satisfies $\cA''(\fz) > 0$. By definition of $\fM$, we have $\bm^{(1)},\bm^{(2)}$ satisfy \eqref{eq:set_inclusion}, thus completing the proof.
\end{proof}

We conclude this section with examples of measures in $\fM$, including a proof of \Cref{thm:two_sum}. Fix $\bm \in \fM$. We require a preliminary discussion about the sublevel sets of $\cS$. Observe that \eqref{jacobi_hypothesis} implies $\Re \cS_u(z)$ extends continuously to $\C \setminus \{E_-\}$. Fix $z \in \C$ with $\Re z > E_+$. We check that
\[ \mathrm{D}_{G(z)}^- := \{w \in \C \setminus \{E_-\}: \Re \cS_{G(z)}(w) < \Re \cS_{G(z)}(z) \} \]
is connected and unbounded. Indeed, $\mathrm{D}_{G(z)}^-$ contains $\cD_{G(z)}^- \subset \C \setminus \supp \bm$ which has a unique unbounded component by \Cref{thm:component}. Thus if $\mathrm{D}_{G(z)}^-$ is not connected then it must have a bounded component and the harmonic principle implies that any bounded component must intersect $\supp \bm$. However, if $x_0$ is in this intersection and $\pm \Im z \ge 0$, then
\[ \Re \cS_{G(z)}(x_0 \pm \bi t) = \Re(x_0 G(z) ) \pm \Re( \bi t G(z) ) - \int \log|x_0 + \bi t - x| d\bm(x) - \Re \log G(z) \]
is strictly decreasing to $-\infty$ as $t \in (0,\infty)$ increases, since $\mp \Im G(z) \ge 0$ and $\Re G(z) > 0$ --- a contradiction. Thus $\mathrm{D}_{G(z)}^-$ must be connected.

The property $\Re z > E_+$ also implies that 
\[ \cS_{G(z)}'(w) \big|_{w=z} = 0, \quad \quad \cS_{G(z)}''(w) \big|_{w=z} = -G'(z) \ne 0. \]
Thus $\cS_{G(z)}(w)$ is locally quadratic at $w = z$, and $z$ is contained in the boundary of $\mathrm{D}_{G(z)}^-$ and two distinct components of
\[ \mathrm{D}_{G(z)}^+ := \{w \in \C \setminus \{E_-\}: \Re \cS_{G(z)}(w) > \Re \cS_{G(z)}(z) \}. \]
By the uniqueness of the unbounded component of $\cD_{G(z)}^+ \subset \mathrm{D}_{G(z)}^+$, we see that one of these components of $\mathrm{D}_{G(z)}^+$ is bounded and must therefore intersect $\supp \bm$.

\begin{lemma} \label{thm:nonpositive_cauchy}
If $\bm \in \cM$ and
\begin{align} \label{eq:D_interval}
\mathrm{D}_{G(z)}^+ \cap \supp \bm
\end{align}
is an interval containing $E_+$ for any $z \in \C$ with $\Re z > E_+$, then $\bm \in \fM$.
\end{lemma}

\begin{proof}
By our discussion above, \eqref{eq:D_interval} implies that $\mathrm{D}_{G(z)}^+$ has a single bounded component containing $E_+$ in its interior and $z$ on its boundary. Thus, we can find $\gamma \subset \C \setminus \supp \bm$ passing through $z$ such that $\gamma \setminus \{z\} \subset \cD_{G(z)}^-$. In view of the discussion above, we have $z \in \fO$ and $z = G^{-1}(G(z))$. Furthermore, for any compact subset $K$ of $\{z \in \C:  \Re z \ge \xi\}$ and any $\xi > E_+$, we can choose $\fp$ close enough to $E_+$ so that $G(K) \subset \fO_\fp$. By \Cref{thm:domain} (ii), we see that \eqref{eq:set_inclusion_xi} holds.
\end{proof}

\begin{example}[Nonpositivity of Cauchy Transform on Support] \label{ex:nonpositive}
If $\bm \in \cM$ and $\Re G_\bm(z)$ extends continuously to a nonpositive function on $(E_-,E_+)$, then $\bm \in \fM$. Indeed, for $\Re z \ge E_+$, we have
\[ \frac{\partial}{\partial x} \Re \cS_{G(z)}(x) = \Re\big( G(z) - G(x) \big) > 0 \]
for all $x \in (E_-,E_+]$. Since $\mathrm{D}_{G(z)}^+$ intersects $(E_-,E_+]$ nontrivially, we must have that $\mathrm{D}_{G(z)}^+ \cap \supp \bm$ is an interval containing $E_+$. By \Cref{thm:nonpositive_cauchy}, $\bm \in \fM$.
\end{example}

We provide a couple concrete examples of measures in $\fM$.

\begin{example}[Jacobi Measure for a Range of Exponents] \label{ex:jacobi}
We consider measures $\bm$ with density of the form
\[ \frac{1}{Z_{a,b}} x^a (1 - x)^b \1_{(0,1)} \]
where $Z_{a,b}$ is a normalization constant. To satisfy integrability, we require $a,b > -1$. We show that $\bm \in \fM$ for $a \ge -1/2$ and $-1 < b \le -1/2$. Since $\fM$ is closed under dilations and translations, this implies measures with density proportional to $(x - \alpha)^a(\beta - x)^b \1_{(\alpha,\beta)}$ with $(a,b) \in [-1/2,\infty)\times (-1,-1/2]$ are contained in $\fM$. Let $G(z)$ denote the Cauchy transform of $\bm$. Then $\Re G(z)$ extends to $\C \setminus \{a,b\}$ coinciding with the principal value integral on $(a,b)$. By \cite[\S 15.2, Eq.(33)]{EMOT54}, this can be exactly evaluated for $z \in (0,1)$ by
\[ \Re G(z) = - z^a (1 - z)^b \pi \cot(b \pi) + \frac{\Gamma(a+1)\Gamma(b)}{\Gamma(a+b+1)} {}_2F_1(-\alpha-\beta,1;1-\beta;1-z) \quad \quad z \in (0,1) \]
where we recall the Gauss hypergeometric function
\[ {}_2F_1(a,b;c;z) = \sum_{k=0}^\infty \frac{(a)_k(b)_k}{(c)_k} \frac{z^k}{k!} \]
for $z \in (0,1)$ and $(x)_k = x(x+1)\cdots(x+k-1)$. We see that both summands are nonpositive if $a \ge -1/2$ and $-1 < b \le -1/2$. Thus, $\Re G(z) \le 0$ for $z \in (0,1)$. By \Cref{ex:nonpositive}, we see that $\bm \in \fM$.
\end{example}

\begin{proof}[Proof of \Cref{thm:two_sum}]
By \Cref{thm:master}, this is a consequence of \Cref{thm:general_twosum} and \Cref{ex:jacobi}.
\end{proof}

Our final example, the uniform measure of an interval, does not satisfy the nonpositivity condition, but satisfies \eqref{eq:set_inclusion_xi}.

\begin{example}[Uniform Measure]
The uniform measure on $[0,1]$ has Cauchy transform
\[ G(z) = \log\left( \frac{z}{z - 1} \right). \]
First, observe that
\[ \frac{\partial}{\partial x}\Re \cS_{G(z)}(x) = \Re \left( G(z) - G(x) \right) = \log \left| \frac{z}{z-1} \right| - \log \left| \frac{x}{x-1} \right|\]
is decreasing for $x \in (0,1)$ which implies $\Re \cS_{G(z)}(x)$ is a unimodal function on $(0,1)$. This implies $\cD_{G(z)}^+ \cap \supp \bm$ is an interval. Thus it remains to show $1 \in \cD_{G(z)}^+$ for $\Re z > 1$. We can see that
\[ \frac{\partial}{\partial z} \left( \cS_{G(z)}(1) - \cS_{G(z)}(z) \right) = (1 - z) G'(z) = \frac{1}{z} \]
Then for $z$ with $\Re z > 1$ and $\Im z \ge 0$, we have
\[ \cS_{G(z)}(1) - \cS_{G(z)}(z) = \int_1^{\Re z} \frac{1}{s} \, ds + \int_0^{\Im z} \frac{\bi}{\Re z + \bi t} \, dt. \]
Then
\[ \Re\big( \cS_{G(z)}(1) - \cS_{G(z)}(z) \big) = \int_1^{\Re z} \frac{1}{s} \, ds + \int_0^{\Im z} \frac{t}{(\Re z)^2 + t^2} \, dt > 0. \]
The argument for $\Im z \le 0$ is similar, or use conjugate symmetry. Thus, $1 \in \cD_{G(z)}^+$ for $\Re z > 1$.

More generally, one can show that the Jacobi measure with parameters $0 \le a \le 1/2$, $-1/2 \le b \le 0$ is contained in $\fM$ using a similar unimodality argument combined with checking $1 \in \cD_{G(z)}^+$ for $\Re z > 1$. We do not a provide a rigorous proof for the sake of brevity.
\end{example}

\section{Decompositions of Representations of the Unitary Group} \label{sec:applications_q}

We conclude with a proof of \Cref{thm:multitensor}. By \Cref{thm:master}, this is an immediate consequence of the following theorem.

Given $\mu \in \cM_1$, let
\begin{align} \label{eq:tau_q}
\tau_q(\mu) := \left( 1 + \frac{(e^{G_\mu(x(\mu))} - 1)^2}{e^{G_\mu(x(\mu))} G_\mu'(x(\mu))} \right)^{-1}
\end{align}
where we recall $x(\mu) = 4(E_+ - E_-) + E_+$.

\begin{theorem} \label{thm:application_q}
Under the hypotheses of \Cref{thm:multitensor} with $\tau_q(\cdot)$ as above, the Schur generating functions for $\lambda$ form an Airy edge appropriate sequence where
\[ A_N(u) = \sum_{i=1}^N G_{\mathrm{m}^{(i)}}^{-1}(u) + \frac{1 - n}{1 - e^{-u}}, \]
$\mathrm{m}^{(i)} := \tfrac{1}{N} \sum_{j=1}^N \delta_{\frac{\lambda_j^{(i)} + N - j}{N}}$, and $\fz_N$ is the minimal positive critical point of $A_N$.
\end{theorem}

The proof of \Cref{thm:application_q} relies heavily on \eqref{eq:MK_correspondence}.

For $\mu \in \cM_1$, we note that
\begin{align} \label{eq:E_inclusion}
(E_-(\cQ \mu),E_+(\cQ \mu)) \subset (E_-(\mu),E_+(\mu))
\end{align}
Indeed, $G_\mu(z)$ is positive and real for $z \in (E_+(\mu),\infty)$ so that $G_{\cQ \mu}(z)$ is positive for $z \in (E_+(\mu),\infty)$ by \eqref{eq:MK_correspondence}. This implies $E_+(\cQ \mu) \le E_+(\mu)$ from \eqref{eq:cauchy_inversion}, and we can likewise show $E_-(\cQ \mu) \ge E_-(\mu)$.

\begin{proposition} \label{thm:high_restriction}
If $\tau \ge \tau_q(\mu)$, then
\[ \mathsf{A}(u) := n G_\mu^{[-1]}(u) + \frac{1 - n}{1 - e^{-u}} \]
has a minimal real positive critical point $\mathsf{z} \in (0,G_\mu(E_+(\mu)))$ such that $\mathsf{A}''(\mathsf{z}) > 0$ and
\[ G_\mu^{[-1]}(\mathsf{z}) - E_+(\mu) > 4(E_+(\mu) - E_-(\mu)). \]
Furthermore, $\mathsf{z} = G_{\mu}(E_+(\mu))$ and $G_{\mu}'(E_+(\mu)) = -\infty$.
\end{proposition}

\begin{proof}
We omit the details as the proof is virtually identical to that of \Cref{thm:high_compression}. The key difference is that \Cref{thm:high_restriction} relies on the strict monotonicity of $-G'/G^2 = (d/dz)(1/G)$. Here, we use the strict monotonicity of
\[ -\frac{e^{G_\mu(z)} G_\mu'(z)}{(e^{G_\mu(z)} - 1)^2} = \frac{d}{dz} \frac{1}{1 - e^{-G_\mu(z)}} \]
which follows from the fact that
\[ \frac{1}{1 - e^{-G_\mu(z)}} = \frac{1}{G_{\cQ\mu}(z)} \]
by \eqref{eq:MK_correspondence}.
\end{proof}

\begin{lemma} \label{thm:q_to_semi}
If $\mu \in \cM_1$, $0 < G_\mu(E_+(\mu)) < \infty$, and $G_\mu'(E_+(\mu)) = -\infty$, then $E_+(\mu) = E_+(\cQ \mu)$, $0 < G_{\cQ\mu}(E_+(\cQ \mu)) < \infty$, and $G_{\cQ\mu}'(E_+(\cQ\mu)) = -\infty$.
\end{lemma}

\begin{proof}
By \eqref{eq:MK_correspondence}, we have $0 < G_{\cQ \mu}(E_+(\mu)) < \infty$. Taking derivatives of \eqref{eq:MK_correspondence}, we obtain
\[ e^{-G_\mu(z)} G_\mu(z)' = G_{\cQ \mu}'(z) \]
from which we see that $G_{\cQ\mu}'(E_+(\mu)) = -\infty$. Since $G_{\cQ\mu}(z)$ is finite for $z \in (E_+(\cQ \mu),\infty)$ and $E_+(\cQ \mu) \le E_+(\mu)$ by \eqref{eq:E_inclusion}, we must have $E_+(\cQ \mu) = E_+(\mu)$.
\end{proof}

\begin{proof}[Proof of \Cref{thm:application_q}]
We show that we can apply \Cref{thm:submaster_q} (iii). We have
\[ \cA(u) := G_{\bm^{(1)}}^{[-1]}(u) + \cdots + G_{\bm^{(n)}}^{[-1]}(u) + \frac{1 - n}{1 - e^{-u}} = G_{\mu_1^{\otimes n_1}}^{[-1]}(u) + \cdots + G_{\mu_k^{\otimes n_k}}^{[-1]}(u) + \frac{1 - k}{1 - e^{-u}} \]
where the inequality follows from
\[ \sum_{j=n_1 + \cdots + n_{i-1} + 1}^{n_1 + \cdots + n_i} G_{\bm^{(j)}}^{[-1]}(u) + \frac{1 - n_i}{1 - e^{-u}} = n_i G_{\mu_i}^{[-1]}(u) + \frac{1 - n_i}{1 - e^{-u}} = G_{\mu_i^{\otimes n_i}}^{[-1]}(u) \]
coming from our assumption and \eqref{eq:Ginv_relation_q}. By \Cref{thm:high_restriction} and the assumption $n_i \ge \tau(\mu_i)$, we have
\[ 0 < G_{\mu_i^{\otimes n_i}}(E_+(\mu_i^{\otimes n_i})) < \infty, \quad \quad G_{\mu_i^{\otimes n_i}}'(E_+(\mu_i^{\otimes n_i})) = -\infty. \]

Let $\mu = \mu_1^{\otimes n_1} \otimes \cdots \otimes \mu_k^{\otimes n_k}$. Then \Cref{thm:q_to_semi} implies $E^{(i)} := E_+(\mu_i^{\otimes n_i}) = E_+((\cQ\mu_i)^{\boxplus n_i})$,
\[ 0 < G_{(\cQ\mu_i)^{\boxplus n_i}}(E^{(i)}) < \infty, \quad \quad G_{(\cQ\mu_i)^{\boxplus n_i}}'(E^{(i)}) = -\infty. \]
Iterating \Cref{thm:free_convolution_cp} as in the proof of \Cref{thm:application}, we obtain
\[ \cA_\cQ(v) := G_{(\cQ\mu_1)^{\boxplus n_1}}^{[-1]}(v) + \cdots + G_{(\cQ\mu_k)^{\boxplus n_k}}^{[-1]}(v/\tau_n) + \frac{1 - k}{v} \]
has a minimal positive critical point $\fz_\cQ = G_{\cQ\mu}(E_+(\cQ \mu)) \in \left(0,\min_{1 \le i \le n} G_{(\cQ \mu_i)^{\boxplus n_i}}(E^{(i)})) \right)$ and $\cA_\cQ''(\fz_\cQ) > 0$. Moreover, we have subordination functions $\omega_1,\ldots,\omega_n$ satisfying
\[ G_{(\cQ \mu_i)^{\boxplus n_i}}(\omega_i(z)) = G_{\cQ\mu}(z), \quad \quad i = 1,\ldots,n\]
and
\[ \omega_i(E_+(\cQ \mu)) > E^{(i)}. \]
From this, we see that $\omega_i(z)$ is real for $z \in (E_+(\cQ \mu),\infty)$. By \eqref{omega:subord_q}, these same functions satisfy
\begin{align} \label{eq:subord}
G_{\mu_i^{\otimes n_i}}(\omega_i(z)) = G_\mu(z), \quad \quad i = 1,\ldots,n.
\end{align}
It follows that $E := E_+(\cQ \mu) = E_+(\mu)$. Indeed, the inequality $E_+(\cQ\mu) \le E_+(\mu)$ from \eqref{eq:E_inclusion} cannot be strict because \eqref{eq:subord} implies $G_\mu(z)$ is real for $z \in (E_+(\cQ\mu),\infty)$ so that $E_+(\mu) \le E_+(\cQ\mu)$ by \eqref{eq:cauchy_inversion}.

We have that $\fz := G_\mu(E)$ is finite by \eqref{eq:MK_correspondence} and the finiteness of $\fz_\cQ$. Therefore, the map $G_{\cQ\mu}(G_\mu^{[-1]}(u))$ is an invertible map from $(0,\fz]$ to $(0,\fz_\cQ]$. Similarly, $G_{(\cQ \mu_i)^{\boxplus n_i}}(G_{\mu_i^{\otimes n_i}}^{[-1]}(u))$ is an invertible map from $(0,G_{\mu_i^{\otimes n_i}}(E^{(i)})]$ to $(0,G_{(\cQ \mu_i)^{\boxplus n_i}}(E^{(i)})]$ for $1 \le i \le n$. Moreover, \eqref{eq:MK_correspondence} shows that
\[ G_{\cQ\mu}(G_\mu^{[-1]}(u)) = 1 - e^{-u}, \quad \quad G_{(\cQ \mu_i)^{\boxplus n_i}}(G_{\mu_i^{\otimes n_i}}^{[-1]}(u)) = 1 - e^{-u} \]
on their respective domains. In particular, the maps agree. Thus, $\fz \in \left(0,\min_{1 \le i \le n} G_{\mu_i^{\otimes n_i}}(E^{(i)})\right)$. Furthermore,
\[ \cA_\cQ(1 - e^{-u}) = G_{\cQ\mu}^{[-1]}(1 - e^{-u}) = G_\mu^{[-1]}(u) = \cA(u), \quad \quad u \in (0,G_\mu(E_+(\mu))] \]
by \eqref{eq:Ginv_relation} and \eqref{eq:Ginv_relation_q}, so that $\fz$ is the minimal positive critical point of $\cA$ and $\cA''(\fz) > 0$. Finally, we argue as in the proof of \Cref{thm:application} to prove \eqref{eq:set_inclusion_q}. \end{proof}

\appendix

\section{Supersymmetric Lifts} \label{appendix:supersymmetric}

We describe a general family of supersymmetric lifts of the Schur functions subsuming those used in this article and a special case which gives the supersymmetric Schur functions.

Let $\delta_N = (N-1,N-2,\ldots,1,0)$ and $\lambda = (\lambda_1 \ge \cdots \ge \lambda_N) \in \Z^N$. Set $\vec{x} = (x_1,\ldots,x_{N+k})$ and $\vec{y} = (y_1,\ldots,y_k)$. Given some function $f:\Omega \subset \C \to \C$, define
\[ F(\vec{x},\vec{y}) = \left( f(x_i,y_j) \right)_{\substack{1 \le i \le N+k \\ 1 \le j \le k}} \]
and
\[ s_{\lambda,f}(\vec{x}/\vec{y}) := (-1)^{Nk} \frac{D(\vec{x};-\vec{y})}{\Delta(\vec{x}) \Delta(-\vec{y})} \det \begin{pmatrix} F(\vec{x},\vec{y}) & A_{\lambda + \delta_N}(\vec{x}) \end{pmatrix}. \]

\begin{theorem}
If 
\[ \lim_{x \to y} (x - y) f(x,y) = 1, \]
then $s_{\lambda,f}(\vec{x}/\vec{y})$ is a supersymmetric lift of $s_\lambda(x_1,\ldots,x_N)$ on $\Omega$.
\end{theorem}

\begin{proof}
The proof is identical to that of \Cref{thm:ssym_lift}. 
\end{proof}

In the case where $f(x,y) = 1/(x - y)$, we have a connection with the algebra of supersymmetric polynomials which we now briefly define. We refer the reader to \cite[I, 3, Examples 23 \& 24]{Mac}, \cite[Chapter 2]{Moens} for further details. 

Let $\Lambda_n$ denote the algebra of symmetric polynomials in $n$-variables $x_1,\ldots,x_n$ and
\[ \Lambda = \varprojlim_{n} \Lambda_n \]
be the algebra of symmetric functions where the projection $\Lambda_{n+1} \to \Lambda_n$ is the map setting $x_{n+1} = 0$. Denote by $\text{pr}_n:\Lambda \to \Lambda_n$ the resulting projections from $\Lambda$. Since the Newton power sums $p_k$ form an algebraic basis for $\Lambda$, an algebra homomorphism on $\Lambda$ is determined by its action on the Newton power sums. The \emph{algebra of supersymmetric polynomials in ordinary variables $(x_1,\ldots,x_m)$ and dual variables $(y_1,\ldots,y_n)$}, denoted $\Lambda_{m,n}$, is the image of the map $\text{pr}^{\text{super}}_{m,n}: \Lambda \to \C[x_1,\ldots,x_m,y_1,\ldots,y_n]$ determined by
\[ p_0 \mapsto 1, \quad p_k \mapsto \sum_{i=1}^m x_i^k - \sum_{j=1}^n y_j^k. \]
From this definition, it is apparent that setting $x_m = y_n$ cancels these two variables. Thus, if we fix an element $g \in \Lambda$ and a positive integer, then $\text{pr}^{\text{super}}_{n+k,k} g$ for $k \ge 0$ is a supersymmetric lift of $\text{pr}_n g$.

In particular, we consider the Schur function $s_\lambda \in \Lambda$, defined by the projective limit of $s_\lambda(x_1,\ldots,x_n) \in \Lambda_n$. The \emph{supersymmetric Schur function} $s_\lambda(x_1,\ldots,x_m/y_1,\ldots,y_n)$ is the image of $s_\lambda \in \Lambda$ by the map $\text{pr}^{\text{super}}_{m,n}$. Then $s_\lambda(x_1,\ldots,x_{n+k}/y_1,\ldots,y_k)$ for $k \ge 0$ is a supersymmetric lift of $s_\lambda(x_1,\ldots,x_n)$.

\begin{theorem}[{\cite{MJ03}}]
For $\lambda \in \GT_N$, we have
\[ s_\lambda(x_1,\ldots,x_{N+k}/y_1,\ldots,y_k) = s_{\lambda,f}(x_1,\ldots,x_{N+k}/y_1,\ldots,y_k) \]
where $f(x,y) = 1/(x-y)$.
\end{theorem}

We note that Moens and Van der Jeugt show a more general formula for $s_\lambda(x_1,\ldots,x_m/y_1,\ldots,y_n)$ where they do not require the condition $m - n = N$ for $\lambda \in \GT_N$. Their notation differs from ours by $y_i \mapsto -y_i$.

We are able to obtain contour integral expressions for the supersymmetric Schur functions which resemble the contour integral formula \Cref{thm:ssym_bessel} for the supersymmetric lift of the multivariate Bessel function used in this article. Recall that $\Delta(x_1,\ldots,x_k) = \prod_{1 \le i < j \le k}(x_i - x_j)$.

\begin{theorem}
Suppose $\lambda \in \GT_N$, $\vec{x} = (x_1,\ldots,x_k), \vec{y} = (y_1,\ldots,y_k)$, and $0 < q < 1$. Then
\begin{align} \label{supersymmetric_schur_k}
\frac{s_\lambda(\vec{x},1,q,\ldots,q^{N-1}/\vec{y})}{s_\lambda(1,q,\ldots,q^{N-1})} = \frac{\prod_{i,j=1}^k(x_i - y_j)}{\Delta(\vec{x}) \Delta(-\vec{y})} \det \left( \frac{1}{x_i - y_j} \frac{s_\lambda(x_i,1,q,\ldots,q^{N-1}/y_j)}{s_\lambda(1,q,\ldots,q^{N-1})} \right)_{i,j=1}^k.
\end{align}
If $x,y \in \C \setminus \{1,q,\ldots,q^{N-1}\}$, then
\begin{align} \label{supersymmetric_schur}
\begin{multlined}
\frac{s_\lambda(x,1,q,\ldots,q^{N-1}/y)}{s_\lambda(1,q,\ldots,q^{N-1})} \\
= (x - y) \prod_{i=1}^N \frac{y - q^{i-1}}{x - q^{i-1}} \left( \frac{1}{x-y} + \oint \oint \frac{\log q}{q^{w-z} - 1} \cdot \frac{\pi e^{\sigma \bi \pi w}}{\sin \pi w} \cdot \frac{x^z}{y^{w+1}} \left( \prod_{i=1}^N \frac{q^w - q^{\lambda_i + N - i}}{q^z - q^{\lambda_i + N - i}} \right) \,\frac{dw\,dz}{(2\pi\bi)^2} \right)
\end{multlined}
\end{align}
where $\sigma \in \{+1,-1\}$, the $z$-contour contains $\{\lambda_r + N - r\}_{r=1}^N$ and the $w$-contour is a union of two rays, with initial point $p \in (-1,0)$, positively oriented around the $z$-contour and $0,1,2,\ldots$. There is freedom in the choice of $\sigma$ and the branch of $\log x$ and $\log y$, subject only to the restriction that the integrand decays sufficiently fast as it approaches $\infty$.
\end{theorem}

\begin{proof}
Set $\ell_i = \lambda_i + N - i$ and
\[ \text{a}(x) = (x^{\ell_1},\ldots,x^{\ell_N}), \quad \text{b}(y) = \left(\frac{1}{y - q^{N-1}}, \frac{1}{y - q^{N-2}},\ldots \frac{1}{y - 1} \right) \]
and denote the transpose of a vector $\text{u}$ by $\text{u}'$. Arguing as in the proof of \Cref{thm:general_ssym_bessel}, we obtain \eqref{supersymmetric_schur_k} and
\[ \frac{s_\lambda(x,q^{N-1},q^{N-2},\ldots,1/y)}{s_\lambda(q^{N-1},q^{N-2},\ldots,1)} = (x - y) \prod_{i=1}^N \frac{y - q^{i-1}}{x - q^{i-1}} \left( \frac{1}{x - y} + \text{a}(x)' \mathsf{A}_{\vec{\ell}}^{-1} \text{b}(y) \right) \]
where
\[ \mathsf{A}_{\vec{\ell}} := \left( \begin{array}{cccc}
q^{(N-1)\ell_1} & q^{(N-1)\ell_2} & \cdots & q^{(N-1)\ell_N} \\
q^{(N-2)\ell_1} & q^{(N-2)\ell_2} & \cdots & q^{(N-2)\ell_N} \\
\vdots & \vdots & \ddots & \vdots \\
1 & 1 & \cdots & 1
\end{array} \right). \]
Using the inverse Vandermonde formula and assuming that $|y| > q^{N-1}$, we have
\begin{align*}
\text{a}(x)' \mathsf{A}_{\vec{\ell}}^{-1} \text{b}(y) &= \sum_{a,b=1}^N (-1)^{b-1} \frac{x^{\ell_a}}{y - q^{N-b}} \frac{e^{b-1}(q^{\ell_1},\ldots,\widehat{q^{\ell_a}},\ldots,q^{\ell_N})}{\prod_{j \ne a} (q^{\ell_a} - q^{\ell_j})} \\
&= \sum_{k=0}^\infty \sum_{a,b=1}^N (-1)^{b-1} \frac{x^{\ell_a}q^{k(N-b)}}{y^{k+1}} \frac{e^{b-1}(q^{\ell_1},\ldots,\widehat{q^{\ell_a}},\ldots,q^{\ell_N})}{\prod_{j \ne a} (q^{\ell_a} - q^{\ell_j})} \\
&= \sum_{k=0}^\infty \sum_{a=1}^N \frac{x^{\ell_a}}{y^{k+1}} \prod_{j \ne a} \frac{q^k - q^{\ell_j}}{q^{\ell_a} - q^{\ell_j}}
\end{align*}
where the final line follows from recognizing the generating function for the elementary symmetric functions. The sum in $a$ can be contracted into a contour yielding
\[ \sum_{k=0}^\infty \oint \frac{1}{y^{k+1}} \frac{dz}{2\pi\bi} \cdot \frac{x^z \log q}{q^{k-z} - 1} \prod_{j=1}^N \frac{q^k - q^{\ell_j}}{q^z - q^{\ell_j}} \]
where the $z$-contour is positively oriented around $\{\ell_1,\ldots,\ell_N\}$ and no other singularities of the integrand. Since the poles of $\frac{\pi}{\sin \pi z}$ have residue $(-1)^k$ for $k \in \Z$, we obtain \eqref{supersymmetric_schur}.
\end{proof}

We describe another family of lifts of multivariate Bessel functions considered in \cite{GS} satisfying the cancellation property. Though these are not supersymmetric lifts as defined in this article, they exhibit similar contour integral formulas and determinantal structure. Let $0 < b_1,\ldots,b_k \le N$ be distinct integers and $a_1,\ldots,a_k \in \C$. Given $\vec{\ell} = (\ell_1 > \cdots > \ell_N)$, define 
\[ B_{\vec{\ell}}(a_1,\ldots,a_k/b_1,\ldots,b_k) := \frac{\cB_{\vec{\ell}}(a_1,\ldots,a_k,N-1,\ldots,\wh{b_{i_1}},\ldots,\wh{b_{i_k}},\ldots,0)}{\cB_{\vec{\ell}}(N-1,\ldots,0)} \]
where $i_1,\ldots,i_k$ is the permutation of $1,\ldots,k$ such that $b_{i_1} > \cdots > b_{i_k}$. Then $B_{\vec{\ell}}$ satisfies a cancellation property:
\[ B_{\vec{\ell}}(a_1,\ldots,a_k/b_1,\ldots,b_k)|_{a_k = b_k} = B_{\vec{\ell}}(a_1,\ldots,a_{k-1}/b_1,\ldots,b_{k-1}) \]
with $B_{\vec{\ell}} = 1$ for $k = 0$.

\begin{theorem}[{\cite[Theorem 3.2]{GS}}]
Set $\vec{a} = (a_1,\ldots,a_k)$ and $\vec{b} = (b_1,\ldots,b_k)$. Then
\begin{align}
B_{\vec{\ell}}(a_1,\ldots,a_k/b_1,\ldots,b_k) = \frac{D(\vec{a},-\vec{b})}{\Delta(\vec{a})\Delta(-\vec{b})} \det \left( \frac{1}{a_i - b_j} B_{\vec{\ell}}(a_i/b_j) \right)_{i,j=1}^k 
\end{align}
If $a \in \C$ and $b \in \{0,\ldots,N-1\}$, then
\begin{align}
B_{\vec{\ell}}(a_i/b_j) = (-1)^{N-b+1} \frac{\Gamma(N-b) \Gamma(b+1) \Gamma(a-N+1) (a - b)}{\Gamma(a+1)} \oint \frac{dz}{2\pi\bi} \oint \frac{dw}{2\pi\bi} \frac{z^a w^{-b-1}}{z-w} \cdot \prod_{i=1}^N \frac{w - e^{\ell_i}}{z - e^{\ell_i}}
\end{align}
where the $z$-contour is positively oriented around $e^{\ell_i}$ and avoids the negative real axis and the $w$-contour is positively oriented around the $z$-contour and $0$.
\end{theorem}

We therefore see similar determinantal and contour integral formulas for $B_{\vec{\ell}}$ and $\cB_{\vec{\ell}}$. A similar contour integral formula exists for Schur functions \cite[Theorem 3.4]{CG} and \cite[Equation 3.10]{GS}.

\section{Laplace Transform for the Airy Point Process}

\begin{proposition} \label{airy_moments}
Let $(\fa_i)_{i=1}^\infty$ denote the Airy point process. For any $c_1,\ldots,c_n > 0$, we have
\begin{align} \label{M_airy}
\begin{split}
\E \left[ \prod_{i=1}^n \sum_{j=1}^\infty e^{c_i \fa_j} \right] &= \frac{e^{\sum_{i=1}^n c_i^3/12}}{(2\pi\bi)^n} \int_{\Re z_1 = \upsilon_1} \,\frac{dz_1}{c_1} \cdots \int_{\Re z_n = \upsilon_n} \, \frac{dz_n}{c_n} \\
&\quad \times \exp\left( \sum_{i=1}^n c_i z_i^2 \right)
\cdot \prod_{1 \le i < j \le n} \frac{(z_j + \frac{c_j}{2} - z_i - \frac{c_i}{2})(z_j - \frac{c_j}{2} - z_i + \frac{c_i}{2})}{(z_j - \frac{c_j}{2} - z_i - \frac{c_i}{2})(z_j + \frac{c_j}{2} - z_i + \frac{c_i}{2})}
\end{split}
\end{align}
where the contours are oriented with increasing imaginary part and $\upsilon_1 < \cdots < \upsilon_n$ are chosen so that
\begin{align} \label{ups_ineq}
\upsilon_j + \frac{c_j}{2} < \upsilon_i - \frac{c_i}{2}, \quad 1 \le i < j \le n.
\end{align}
\end{proposition}

\begin{remark}
Due to the exponential decay, we may choose arbitrary points $\upsilon_1,\ldots,\upsilon_n$ satisfying the inequality above.
\end{remark}

\begin{proof}
For $0 \le k \le n$, let
\begin{align*}
\sfM_n(c_1,\ldots,c_k; c_{k+1},\ldots,c_n) = \E \left[ \sum e^{\sum_{i=1}^n c_i a_{j_i}} \right]
\end{align*}
where the sum is over all $(j_1,\ldots,j_n) \in \Z_{>0}^n$ such that $j_1,\ldots,j_k$ are distinct. This gives the recursion
\begin{align} \label{M_recursion}
\begin{split}
\sfM_n(c_1,\ldots,c_k; c_{k+1},\ldots,c_n) &= \sfM_n(c_1,\ldots,c_{k+1}; c_{k+2},\ldots,c_n)\\
& \quad \quad + \sum_{\ell=1}^k \sfM_{n-1}(c_1,\ldots,c_\ell + c_{k+1},\ldots,c_k; c_{k+2},\ldots,c_n)
\end{split}
\end{align}
which follows from the fact that if $(j_1,\ldots,j_n) \in \Z_{>0}^n$ such that $j_1,\ldots,j_k$, are distinct, then $j_{k+1}$ is either distinct from $j_1,\ldots,j_k$ or is equal to $j_\ell$ for some $1 \le \ell \le k$.

Note that we want to compute
\[ \sfM_n(\,;c_1,\ldots,c_n) = \E \left[ \prod_{i=1}^n \sum_{j=1}^\infty e^{c_i \fa_j} \right]. \]
On the other hand, we can immediately compute
\begin{align} \label{R_airy}
\begin{split}
\sfM_n(c_1,\ldots,c_n;\,) & = \frac{e^{\sum_{i=1}^n c_i^3/12}}{(2\pi\bi)^n} \int_{\bi \R} \,\frac{dz_1}{c_1} \cdots \int_{\bi \R} \, \frac{dz_n}{c_n} \\
& \quad \times \exp\left( \sum_{i=1}^n c_i z_i^2 \right) \cdot\prod_{1 \le i < j \le n} \frac{(z_j + \frac{c_j}{2} - z_i - \frac{c_i}{2})(z_j - \frac{c_j}{2} - z_i + \frac{c_i}{2})}{(z_j - \frac{c_j}{2} - z_i - \frac{c_i}{2})(z_j + \frac{c_j}{2} - z_i + \frac{c_i}{2})}
\end{split}
\end{align}
where the contours are oriented with increasing imaginary part. This was shown in \cite[Page 5]{BG16} (wherein $\sfM_n(c_1,\ldots,c_n;\,)$ coincides with $R(c_1,\ldots,c_n)$). Observe that \eqref{M_airy} and \eqref{R_airy} have the same integrand but differ in the contours of integration.

Throughout, fix $\upsilon_1 < \cdots < \upsilon_n$ such that \eqref{ups_ineq} holds. We show that
\begin{align} \label{M=I}
\sfM_n(c_1,\ldots,c_k;c_{k+1},\ldots,c_n) = \sfI_n(c_1,\ldots,c_k;c_{k+1},\ldots,c_n)
\end{align}
for $0 \le k \le n$, where
\begin{align*}
\begin{split}
\sfI_n(c_1,\ldots,c_k;c_{k+1},\ldots,c_n) & = \frac{e^{\sum_{i=1}^n c_i^3/12}}{(2\pi\bi)^n} \int_{\bi\R} \,\frac{dz_1}{c_1} \cdots \int_{\bi\R} \frac{dz_k}{c_k} \int_{\bi\R + \upsilon_{k+1}} \frac{dz_{k+1}}{c_{k+1}} \cdots \int_{\bi \R + \upsilon_n} \, \frac{dz_n}{c_n} \\
& \quad \times \exp\left( \sum_{i=1}^n c_i z_i^2 \right) \cdot\prod_{1 \le i < j \le n} \frac{(z_j + \frac{c_j}{2} - z_i - \frac{c_i}{2})(z_j - \frac{c_j}{2} - z_i + \frac{c_i}{2})}{(z_j - \frac{c_j}{2} - z_i - \frac{c_i}{2})(z_j + \frac{c_j}{2} - z_i + \frac{c_i}{2})}.
\end{split}
\end{align*}
The desired equality \eqref{M_airy} is for $k = 0$. By \eqref{R_airy}, the equality \eqref{M=I} holds for $k = n$. The argument is inductive. There are two layers to the induction, we induct in $n$ and then we induct in $k$ starting from $k = n$ and decreasing $k$ to $0$. To prove \eqref{M=I} in general, it suffices to show
\begin{align} \label{I_recursion}
\begin{split}
\sfI_n(c_1,\ldots,c_k; c_{k+1},\ldots,c_n) &= \sfI_n(c_1,\ldots,c_{k+1}; c_{k+2},\ldots,c_n)\\
& \quad \quad + \sum_{\ell=1}^k \sfI_{n-1}(c_1,\ldots,c_\ell + c_{k+1},\ldots,c_k; c_{k+2},\ldots,c_n)
\end{split}
\end{align}
by \eqref{M_recursion}. The recursion \eqref{I_recursion} follows from moving the $z_{k+1}$-integral for
\[ \sfI_n(c_1,\ldots,c_k;c_{k+1},\ldots,c_n) \]
from $\bi \R + \upsilon_{k+1}$ to $\bi \R$ which produces
\begin{align} \label{I_recursion_proof}
\begin{multlined}
\sfI_n(c_1,\ldots,c_k;c_{k+1},\ldots,c_n) = \sfI_n(c_1,\ldots,c_{k+1};c_{k+2},\ldots,c_n) \\
+ \sum_{\ell=1}^k \underset{z_{k+1} = z_\ell + \frac{c_\ell}{2} + \frac{c_{k+1}}{2}}{\res} \frac{e^{\sum_{i=1}^n c_i^3/12}}{(2\pi\bi)^{n-1}} \int_\R \,\frac{dz_1}{c_1} \cdots \int_\R \frac{dz_k}{c_k} \int_{\R + \bi \upsilon_{k+2}} \frac{dz_{k+2}}{c_{k+2}} \cdots \int_{\R + \bi \upsilon_n} \, \frac{dz_n}{c_n} \\
\times \exp\left( \sum_{i=1}^n c_i z_i^2 \right) \cdot\prod_{1 \le i < j \le n} \frac{(z_j + \frac{c_j}{2} - z_i - \frac{c_i}{2})(z_j - \frac{c_j}{2} - z_i + \frac{c_i}{2})}{(z_j - \frac{c_j}{2} - z_i - \frac{c_i}{2})(z_j + \frac{c_j}{2} - z_i + \frac{c_i}{2})}
\end{multlined}
\end{align}
where the first term on the right hand side is the new integral after moving the $z_{k+1}$-integral and the summation in $\ell$ corresponds to the residues at $z_{k+1} = z_\ell + \frac{c_\ell}{2} + \frac{c_{k+1}}{2}$ that are picked up due to the constraints on $\upsilon_1,\ldots,\upsilon_n$ and the term
\[ \prod_{1 \le i < j \le n} \frac{(z_j + \frac{c_j}{2} - z_i - \frac{c_i}{2})(z_j - \frac{c_j}{2} - z_i + \frac{c_i}{2})}{(z_j - \frac{c_j}{2} - z_i - \frac{c_i}{2})(z_j + \frac{c_j}{2} - z_i + \frac{c_i}{2})}. \]

Extracting the integrand $I_\ell$ from the $\ell$th summand. We have
\begin{align*}
I_\ell &:=  \underset{z_{k+1} = z_\ell + \frac{c_\ell}{2} + \frac{c_{k+1}}{2} }{\res} \frac{e^{\sum_{i=1}^n c_i^3/12}}{(2\pi\bi)^{n-1}} \exp\left( \sum_{i =1}^n c_i z_i^2 \right) \cdot \prod_{1 \le i < j \le n} \frac{(z_j + \frac{c_j}{2} - z_i - \frac{c_i}{2})(z_j - \frac{c_j}{2} - z_i + \frac{c_i}{2})}{(z_j - \frac{c_j}{2} - z_i - \frac{c_i}{2})(z_j + \frac{c_j}{2} - z_i + \frac{c_i}{2})} \\
& \quad \quad = \frac{c_{k+1} c_\ell}{c_{k+1} + c_\ell} \frac{e^{\sum_{i=1}^n c_i^3/12}}{(2\pi\bi)^{n-1}} \exp\left( \sum_{i \ne k+1} c_i z_i^2  + c_{k+1} \left(z_\ell + \frac{c_\ell}{2} + \frac{c_{k+1}}{2} \right)^2 \right) \\
& \quad \quad \quad \times \prod_{\substack{i < j \\ i,j \ne \ell,k+1}} \frac{(z_j + \frac{c_j}{2} - z_i - \frac{c_i}{2})(z_j - \frac{c_j}{2} - z_i + \frac{c_i}{2})}{(z_j - \frac{c_j}{2} - z_i - \frac{c_i}{2})(z_j + \frac{c_j}{2} - z_i + \frac{c_i}{2})} \\
& \quad \quad \quad \times \prod_{i \ne \ell,k+1} \frac{(z_\ell + \frac{c_\ell}{2} - z_i - \frac{c_i}{2})(z_\ell - \frac{c_\ell}{2} - z_i + \frac{c_i}{2})}{(z_\ell - \frac{c_\ell}{2} - z_i - \frac{c_i}{2})(z_\ell + \frac{c_\ell}{2} - z_i + \frac{c_i}{2})} \\
& \quad \quad \quad \times \prod_{i \ne \ell,k+1} \frac{((z_\ell + \frac{c_\ell}{2} + \frac{c_{k+1}}{2})+ \frac{c_{k+1}}{2} - z_i - \frac{c_i}{2})((z_\ell + \frac{c_\ell}{2} + \frac{c_{k+1}}{2}) - \frac{c_{k+1}}{2} - z_i + \frac{c_i}{2})}{((z_\ell + \frac{c_\ell}{2} + \frac{c_{k+1}}{2}) - \frac{c_{k+1}}{2} - z_i - \frac{c_i}{2})((z_\ell + \frac{c_\ell}{2} + \frac{c_{k+1}}{2}) + \frac{c_{k+1}}{2} - z_i + \frac{c_i}{2})}
\end{align*}
We note that the factor $\frac{c_{k+1} c_\ell}{c_{k+1} + c_\ell}$ comes from the cross term where $i = \ell, j = k+1$. Cancellation between the last two lines yields
\begin{align*}
I &= \frac{c_{k+1} c_\ell}{c_{k+1} + c_\ell} \frac{e^{\sum_{i=1}^n c_i^3/12}}{(2\pi\bi)^{n-1}} \exp\left( \sum_{i \ne \ell, k+1} c_i z_i^2  + (c_\ell + c_{k+1})\left( z_\ell + \frac{c_{k+1}}{2} \right)^2 + \frac{1}{4} c_{k+1} c_\ell (c_\ell + c_{k+1}) \right) \\
&\quad \times \prod_{\substack{i < j \\ i,j \ne \ell,k+1}} \frac{(z_j + \frac{c_j}{2} - z_i - \frac{c_i}{2})(z_j - \frac{c_j}{2} - z_i + \frac{c_i}{2})}{(z_j - \frac{c_j}{2} - z_i - \frac{c_i}{2})(z_j + \frac{c_j}{2} - z_i + \frac{c_i}{2})} \\
& \quad \times \prod_{i \ne \ell, k+1} \frac{((z_\ell + \frac{c_\ell}{2} + \frac{c_{k+1}}{2})+ \frac{c_{k+1}}{2} - z_i - \frac{c_i}{2}) (z_\ell - \frac{c_\ell}{2} - z_i + \frac{c_i}{2})}{((z_\ell + \frac{c_\ell}{2} + \frac{c_{k+1}}{2}) + \frac{c_{k+1}}{2} - z_i + \frac{c_i}{2})(z_\ell - \frac{c_\ell}{2} - z_i - \frac{c_i}{2})}.
\end{align*}

Changing variables by replacing $z_\ell + \frac{c_{k+1}}{2}$ with $z_\ell$, we obtain
\begin{align*}
I &= \frac{c_{k+1} c_\ell}{c_{k+1} + c_\ell} \frac{e^{\sum_{i\ne k, \ell}^n c_i^3/12 + (c_k + c_\ell)^3/12}}{(2\pi\bi)^{n-1}} \exp\left( \sum_{i \ne \ell, k+1} c_i z_i^2  + (c_\ell + c_{k+1}) z_\ell^2 \right) \\
&\quad \times \prod_{\substack{i < j \\ i,j \ne \ell,k+1}} \frac{(z_j + \frac{c_j}{2} - z_i - \frac{c_i}{2})(z_j - \frac{c_j}{2} - z_i + \frac{c_i}{2})}{(z_j - \frac{c_j}{2} - z_i - \frac{c_i}{2})(z_j + \frac{c_j}{2} - z_i + \frac{c_i}{2})} \\
& \quad \times \prod_{i \ne \ell, k+1} \frac{(z_\ell + \frac{c_\ell}{2} + \frac{c_{k+1}}{2} - z_i - \frac{c_i}{2}) (z_\ell - \frac{c_\ell}{2} - \frac{c_{k+1}}{2} - z_i + \frac{c_i}{2})}{(z_\ell + \frac{c_\ell}{2} + \frac{c_{k+1}}{2} - z_i + \frac{c_i}{2})(z_\ell - \frac{c_\ell}{2} - \frac{c_{k+1}}{2} - z_i - \frac{c_i}{2})}.
\end{align*}
where the $z_\ell$-variable is now being integrated over $\bi \R - \frac{c_{k+1}}{2}$. Since the poles of $z_\ell$ occur at
\[ z_i - \frac{c_i}{2} - \frac{c_\ell}{2} - \frac{c_{k+1}}{2}, \quad z_i + \frac{c_i}{2} + \frac{c_\ell}{2} + \frac{c_{k+1}}{2} \]
we can shift the $z_\ell$-contour back to $\bi \R$ without picking up any additional residues. Recalling that $I$ corresponds to the integrand of the $\ell$th summand in \eqref{I_recursion_proof}, we have shown \eqref{I_recursion}.
\end{proof}

\bibliographystyle{alpha_abbrvsort}
\bibliography{mybib}

\end{document}